\documentclass{amsart}
\usepackage[small, nohug]{diagrams}
\usepackage[mathscr]{eucal}
\usepackage{amsfonts}
\usepackage{amsmath}
\usepackage{amsthm}
\usepackage{amssymb}
\usepackage{latexsym}
\newfont{\msam}{msam10}

\newtheorem{theorem}[]{Theorem}
\newtheorem{proposition}[]{Proposition}
\newtheorem{corollary}[]{Corollary}
\newtheorem{lemma}[]{Lemma}
\theoremstyle{definition}
\newtheorem{definition}[]{Definition}
\newtheorem{remark}[]{Remark}
\newtheorem{example}[]{Example}

\def\remark{\noindent\textbf{Remark.}}
\let\nc\newcommand

\nc{\la}{\label}
\def\bg{\begin}
\def\bthm{\begin{theorem}}
\def\ethm{\end{theorem}}
\def\blemma{\begin{lemma}}
\def\elemma{\end{lemma}}
\def\bproof{\begin{proof}}
\def\eproof{\end{proof}}
\def\bprop{\begin{proposition}}
\def\eprop{\end{proposition}}

\def\Gr{\mbox{\rm{Gr}}^{\mbox{\scriptsize{\rm{ad}}}}}

\def\dlog{{\tt dlog}}
\def\F{\mathcal F}
\def\P{\mathcal{P}}
\def\Z{\mathbb{Z}}

\def\F{\mathcal{F}}

\def\O{\mathcal{O}}
\def\A{\mathbb{A}}

\def\L{\mathcal{L}}
\def\N{\mathbb{N}}
\def\I{\mathcal{I}}
\def\J{\mathcal{J}}
\def\R{\mathscr{I}}
\def\U{\mathcal{U}}
\def\D{\mathcal{D}}
\def\DO{\,\overline{\!\D\!}\,}
\def\DD{\widetilde{\mathcal D}}
\def\MM{\widetilde{M}}
\def\LL{\widetilde{L}}
\def\NN{\widetilde{N}}

\def\RR{R}
\def\DB{(\Omega^1 B)_{\natural}}

\def\DC{\mathscr{D}}
\def\MO{\,\overline{\!M\!}\,}
\def\RO{\overline{R}}

\def\LO{\overline{L}}
\def\NO{\,\overline{\!N\!}\,}

\def\fo{\overline{f}}
\def\tA{\tilde{A}}
\def\tsigma{\tilde{\sigma}}
\def\teta{\tilde{\eta}}
\def\ttau{\tilde{\tau}}
\def\hnu{\bar{\nu}}

\def\tvarrho{\tilde{\varrho}}

\def\AA{A^{\otimes 2}}
\def\BB{B^{\otimes 2}} 

\def\oA{A^{\mbox{\scriptsize{\rm{o}}}}}
\def\reg{\mbox{\scriptsize{\rm{reg}}}}
\def\eA{A^{\mbox{\scriptsize{\rm{e}}}}}
\def\eB{B^{\mbox{\scriptsize{\rm{e}}}}}

\def\tM{\tilde{I}}
\def\tL{\tilde{I}}
\def\tf{\tilde{f}}

\def\tbtheta{\tilde{i}}
\def\biota{\boldsymbol{\iota}}
\def\btheta{i}
\def\bi{\boldsymbol{\eta}}
\def\bt{\boldsymbol{t}}
\def\tbt{\tilde{\boldsymbol{t}}}
\def\tbi{\tilde{\boldsymbol{\eta}}}
\def\ba{\boldsymbol{\alpha}}
\def\tba{\tilde{\boldsymbol{\alpha}}}
\def\bp{\boldsymbol{\pi}}

\def\v{\mbox{\rm v}}
\def\hv{\hat{\mbox{\rm v}}}
\def\w{\mbox{\rm w}}
\def\hw{\hat{\mbox{\rm w}}}
\def\vp{\varphi}
\def\bvp{\widetilde{\psi}}
\def\ha{\hat{a}}
\def\hr{\hat{r}}
\def\hd{\hat{d}}
\def\hsigma{\bar{\sigma}}

\def\hdel{\hat{\Delta}}
\def\sdel{\bar{\Delta}}
\def\sX{\bar{X}}
\def\sY{\bar{Y}}
\def\sZ{\bar{Z}}
\def\sw{\bar{\w}}
\def\sv{\bar{\v}}

\def\m{\mathfrak{m}}
\def\c{\mathbb{C}}
\def\CC{\mathcal{C}}

\def\GrM{\mbox{\sf GrMod}}
\def\Tors{\mbox{\sf Tors}}
\def\Fdim{\mbox{\sf Tors}}
\def\Qgr{\mbox{\sf Qgr}}
\def\bV{\boldsymbol{V}}
\def\bP{\boldsymbol{P}}
\def\DDer{{\mathbb D}{\rm er}}
\nc{\Hom}{{\rm{Hom}}}
\nc{\Ext}{{\rm{Ext}}}
\nc{\htau}{{\bar{\tau}}}
\nc{\HOM}{\underline{\rm{Hom}}}
\nc{\EXT}{\underline{\rm{Ext}}}
\nc{\TOR}{\underline{\rm{Tor}}}
\nc{\End}{{\rm{End}}}
\nc{\GL}{{\rm{GL}}}
\nc{\PGL}{{\rm{PGL}}}
\nc{\G}{{\rm{G}}}
\nc{\Rep}{{\rm{Rep}}}
\nc{\ad}{{\rm{ad}}}
\nc{\dlim}{\varinjlim}
\newcommand{\Mod}{{\tt{Mod}}}
\newcommand{\Mat}{{\tt{Mat}}}

\newcommand{\HH}{H}
\newcommand{\res}{{\rm{res}}}

\newcommand{\Frac}{{\rm{Frac}}}
\newcommand{\eps}{\varepsilon}
\newcommand{\Tor}{{\rm{Tor}}}
\newcommand{\Sym}{{\rm{Sym}}}

\newcommand{\Spec}{{\rm{Spec}}}
\newcommand{\Pic}{{\rm{Pic}}}
\newcommand{\Aut}{{\rm{Aut}}}
\newcommand{\id}{{\rm{Id}}}
\newcommand{\Der}{{\rm{Der}}}

\newcommand{\rk}{{\rm{rk}}}
\newcommand{\Tr}{{\rm{Tr}}}
\newcommand{\tr}{{\rm{tr}}}
\newcommand{\Ker}{{\rm{Ker}}}
\newcommand{\Coker}{{\rm{Coker}}}

\newcommand{\im}{{\rm{Im}}}
\newcommand{\balpha}{{\boldsymbol{\alpha}}}
\newcommand{\bmu}{{\boldsymbol{\mu}}}
\newcommand{\bn}{{\boldsymbol{n}}}
\newcommand{\bk}{{\boldsymbol{k}}}
\newcommand{\bU}{{\boldsymbol{U}}}
\newcommand{\bL}{{\boldsymbol{L}}}
\newcommand{\bM}{{\boldsymbol{M}}}

\newcommand{\be}{{\boldsymbol{\rm e}}}
\newcommand{\Ui}{U_{\infty}}
\newcommand{\blambda}{{\boldsymbol{\lambda}}}
\newcommand{\bgg}{{\boldsymbol{g}}}
\newcommand{\Vi}{V_{\infty}}

\newcommand{\ei}{e_{\infty}}
\newcommand{\lambdai}{\lambda_{\infty}}

\newcommand{\into}{\,\,\hookrightarrow\,\,}
\newcommand{\too}{\,\,\longrightarrow\,\,}
\newcommand{\onto}{\,\,\twoheadrightarrow\,\,}
\numberwithin{equation}{section}
\numberwithin{theorem}{section}
\numberwithin{lemma}{section}
\numberwithin{proposition}{section}
\numberwithin{corollary}{section}
\numberwithin{example}{section}
\numberwithin{remark}{section}
\begin{document}
\title[Ideals of Differential Operators on Curves]{
Ideals of Rings of Differential Operators on Algebraic
Curves\\*[0.8ex] (With an Appendix by George Wilson)}
\date{\today}
\author{Yuri Berest}
\address{Department of Mathematics, Cornell University, Ithaca,
NY 14853-4201, USA}
\email{berest@math.cornell.edu}
%
%
\author{Oleg Chalykh}
\address{School of Mathematics, University of Leeds, Leeds LS2 9JT, UK}
\email{oleg@maths.leeds.ac.uk}
%
%
\address{Mathematical Institute, 24--29 St Giles, Oxford OX1 3LB, UK}
\email{wilsong@maths.ox.ac.uk}
%
%
%
%
%
\maketitle
\section{Introduction}
\la{intro}
Let $ X $ be a smooth affine irreducible curve over $ \c $, and
let $ \D = \D(X) $ be the ring of global differential operators on
$X$. In this paper, we give a geometric classification of left
ideals in $ \D $ and study the natural action of the Picard group
of $\,\D \,$ on the space of isomorphism classes of such ideals.
Our results generalize the classification of left ideals of the
first Weyl algebra $ A_1(\c) $ in \cite{BW1} and \cite{BW2};
however, our approach is quite different.

As shown in \cite{BW1, BW2}, the ideal classes of $ A_1(\c) $ are
parametrized by finite-dimensional algebraic varieties $ \CC_n $
called the Calogero-Moser spaces. The starting point for the present
paper was the observation of Crawley-Boevey (see \cite{CB1}) 
that the same varieties $ \CC_n $  parametrize finite-dimensional irreducible 
representations of certain (infinite-dimensional) algebras associated to graphs.
Specifically, the algebras in question are deformed
preprojective algebras $\, \Pi^\blambda(Q) \,$ (see \cite{CBH}); 
the corresponding graph $ Q $ is the framed Dynkin diagram of simplest type 
$ \tilde{A}_0 $.

Trying to understand the relation between the ideals of $ A_1(\c) $
and irreducible representations of $ \Pi^\blambda(Q) $, we
came up with a new construction of the Calogero-Moser correspondence,
which, besides the Weyl algebra, applied to noncommutative deformations
of Kleinian singularities corresponding to
Dynkin diagrams of other types (see \cite{BCE}). In this paper,
we develop a geometric version of this construction in which
graphs are replaced by algebraic curves.

We begin with a brief overview of our main results. Let $ \R(\D) $ be
the set of isomorphism classes of left ideals in $ \D $. Since $
\D $ is a Noetherian hereditary domain, every ideal of $\D$ is a
projective $\D$-module of rank $1$, so $ \R(\D) $ can be
equivalently defined as the set of isomorphism classes of such
modules. The Grothendieck group $ K_0(\D) $ of finite rank
projective $ \D$-modules is isomorphic to the (algebraic) $K$-group
$\, K_0(X) \,$ of $X$, while $\, K_0(X) \cong \Z \oplus \Pic(X)
\,$, where $ \Pic(X) $ is the Picard group of $X$. Combining these
isomorphisms, we may assign to each ideal class $\, [M] \in \R(\D)
\,$ an element of $ \Pic(X) $ which determines $ [M] $ up to
equivalence in $ K_0(\D) $. In other words, there is a natural map
$\,\gamma:\,\R(\D) \to \Pic(X) \,$, whose fibres are precisely the
{\it stable} isomorphism classes of ideals in $ \D $. 
Our problem reduces thus to describing the fibres of $ \gamma $.

We approach this problem in two steps. First, we introduce the
Calogero-Moser spaces $\, \CC_n(X, \I) \,$ for an arbitrary curve
$ X $ and a line bundle $ \I $ on $ X $, building on the
observation of Crawley-Boevey. For any associative algebra $B$,
there is a `universal' construction of deformed preprojective
algebras $\,\Pi^\blambda(B) \,$ over $B$, with parameters
$\,\blambda \in \c \otimes_{\Z} K_0(B) \,$ (see \cite{CB} and 
Section~\ref{DPA} below). Using this construction, we define $\,
\CC_n(X, \I) \,$ as representation varieties of $\,
\Pi^\blambda(B) \,$ over a triangular matrix extension of the ring
$ A = \O(X) $ of regular functions on $X$ by the line bundle $\I$.
This extension $ B = A[\I] $  abstracts the idea of `framing' a
quiver by adjoining a distinguished new vertex `$ \infty $' and
arrows from $ \infty \,$; geometrically, it can be thought of as a
noncommutative thickening\footnote{For a natural 
interpretation of this kind of construction in terms of derived algebraic 
geometry, see \cite{Kon}.} of $\,\Spec(A \times \c) = X \bigsqcup\,{\rm pt}\,$. We
note that $\, \CC_n(X, \I) \,$ behaves functorially with respect to
$ \I $; in particular, the quotient spaces $\,\overline{\CC}_n(X,\I) :=
\CC_n(X, \I)/\Aut_X(\I) \,$ depend only on the class of $ \I $ in
$ \Pic(X) $. We write $\, \overline{\CC}_n(X) \,$  for the disjoint union of 
$\,\overline{\CC}_n(X, \I)\,$ over $\, \Pic(X) \,$.

Our first main result is a generalization to an arbitrary $X$ of a known 
theorem of G.~Wilson (see \cite{Wi}).
\begin{theorem}[see Theorem~\ref{T6}]
\la{IT6}
For each $ n \geq 0 $ and $ [\I]\in \Pic(X) $, $\,
\CC_{n}(X,\I) \,$ is a smooth affine irreducible variety of
dimension $ 2n $.
\end{theorem}

Now, in view of functoriality of $ \Pi^\blambda$-construction,
there is a natural map $\,\Pi^{\blambda}(B) \to \Pi^1(A) \,$ lifting
the extension $\, B \to A \,$. On the other hand, by a theorem of Crawley-Boevey 
(see \cite{CB}), $\, \Pi^1(A) $ can be identified with the ring $ \D $ of differential
operators on $X$. The resulting algebra homomorphism $\,\btheta:\,\Pi^{\blambda}(B) \to \D\,$ 
relates the module categories of $ \Pi^{\blambda}(B)$ and $ \D $ in a fairly interesting way. 
To be precise, we will prove 
\begin{theorem}[see Theorem~\ref{P1}]
\la{IP1}
The canonical functors $\,(\btheta^*,\,\btheta_*\,,\,\btheta^!)\,$
induced by $\, \btheta: \, \Pi^{\blambda} \to \D \,$ on the (bounded) derived categories 
form a {\it recollement set-up} in the sense of \cite{BBD}{\rm :}
\begin{equation*}
\begin{diagram}[small, tight]
\DC^{b}(\Mod\, \D) \ & \ \pile{\lTo^{\btheta^*}\\
\rTo^{\btheta_{\!*}} \\
\lTo^{\btheta^!}} \ &
\ \DC^{b}(\Mod\, \Pi^\blambda)\ & \ \pile{\lTo^{j_!}\\
\rTo^{j^*}\\ \lTo^{j_*}} \ & \ \DC^{b}(\Mod\, U^\blambda) \\
\end{diagram}\ ,
\end{equation*}
where $ U^\blambda $ is a certain (spherical) subalgebra of $ \Pi^\blambda(B) $.
\end{theorem}
Originally, the recollement conditions were introduced in \cite{BBD}
to formalize a natural structure on the derived category $ \DC({\mathscr Sh}_X)$
of abelian sheaves arising from the stratification of a topological
space into a closed subspace and its open complement. 
In an algebraic setting similar to ours, these conditions were first studied 
in \cite{CPS}.

The functor $\, \btheta_* \,$ yields a fully
faithful embedding of $\,\DC^b(\Mod\,\D)\,$ into
$\,\DC^b(\Mod\,\Pi)\,$ as a `closed stratum', while the
induction functor $\, \btheta^*:\,\DC^b(\Mod\,\Pi) \to
\DC^b(\Mod\,\D) \,$ is an algebraic substitute for
the restriction of a sheaf to that stratum.
This last functor plays a
key role in our construction: it transforms irreducible
$\Pi^\blambda(B)$-modules (viewed as $0$-complexes in
$\DC^b(\Mod\,\Pi)$) to projective $\D$-modules (located in homological 
degree $-1$), inducing natural maps
$$
\omega_n:\,\overline{\CC}_n(X, \I) \to \gamma^{-1}[\I]\ .
$$ 
The main result of this paper can now be encapsulated in the following
theorem. 
\begin{theorem}[see Theorem~\ref{Tmain}]
\la{ITmain} Let $ X $ be a smooth affine irreducible curve over $\c$.

$(a)$\ For each $\,[\I] \in \Pic(X)\,$, amalgamating the maps 
$\, \omega_n\,$ for all $ n \ge 0 $ yields a bijective correspondence
$$
\omega:\ \bigsqcup_{n \geq 0} \overline{\CC}_n(X,\I)
\stackrel{\sim}{\to} \gamma^{-1}[\I] \ .
$$

$(b)$\ There is a natural action on $\, \overline{\CC}_n(X) $ 
of the Picard group $\, \Pic(\D) \,$ of the category of $\D$-modules, and the maps 
$\,\omega_n:\, \overline{\CC}_n(X) \to \R(\D) \,$  are equivariant under this action
for all $ n \ge 0 $.
\end{theorem}
Part $(a)$ of Theorem~\ref{ITmain} gives a geometric description of the fibration
$\gamma $ over a given $ [\I] \in \Pic(X) $. In the special case
when $ X $ is the affine line, $\, \Pic(X) $ is trivial: there is
only one fibre, and it is shown in \cite{BCE} that $\, \omega \,$
agrees with the Calogero-Moser map constructed in \cite{BW1, BW2}.

Part $(b)$ generalizes another aspect of the Calogero-Moser correspondence 
for the Weyl algebra: the equivariance of the Calogero-Moser map
under the action of the automorphism group\footnote{By a theorem of Stafford \cite{St}, 
the group $ \Aut_\c(A_1)$ is known to be isomorphic to $ \Pic(A_1) $.} $ \Aut_\c(A_1)\,$.
The importance of this result is that it allows one to classify the 
algebras Morita equivalent to $ \D $ up to isomorphism. Precisely,
Theorem~\ref{ITmain}$(b)$ implies that the isomorphism classes of domains 
$ \D' $ Morita equivalent to $ \D $ are in one-to-one correspondence with the
orbits of $ \Pic(\D) $ on the Calogero-Moser spaces 
$\,\overline{\CC}_n(X) \,$. For example, for 
$\, n = 0 \,$, we have $\,\overline{\CC}_0(X) = \Pic(X) \,$, and the
action of $ \Pic(\D) $ is transitive on $ \Pic(X) $ (see Proposition~\ref{piceq} below);
this implies a theorem of Cannings and Holland (\cite{CH1}, Theorem~1.10)
that $ \D' \cong \D $ if and only if $\, \D' \cong \End_{\D}(\I \D)\,$
for some line bundle $ \I $. For an arbitrary $ n > 0 $, the structure
of orbits of $ \Pic(\D) $ in $ \overline{\CC}_n(X) $  is complicated; however, one can 
still define a complete set of isomorphism invariants for the algebras $ \D' $ in terms of 
Hochschild homology of $ \Pi^\lambda(B) $. We will discuss this construction elsewhere.

We should now explain how our results relate to earlier work.

The problem of classifying ideals of $ \D(X) $ for a smooth affine curve $X$ 
was first addressed by Cannings and Holland (see \cite{CH, CH1}) who identified 
the space $ \R(\D) $ with a certain infinite-dimensional Grassmannian. 
In the special case when $X = \A^1$, this Grassmannian was introduced 
independently (and for a different purpose) by Wilson (see \cite{W1}), who called 
it the {\it adelic Grassmannian} $ \Gr $. Motivated by earlier work on integrable systems 
\cite{AMM, CC, Kr, KKS}, Wilson showed (see \cite{Wi}) that $ \Gr $ can be decomposed into a countable 
union of smooth varieties $\, \CC_n $, which are now called the Calogero-Moser spaces. 
It is important to understand 
that the Calogero-Moser decomposition is entirely different from the obvious 
stratification of $ \Gr $ by 
finite-dimensional Grassmannians considered in \cite{CH1}. Its relevance for the Weyl algebra 
$ \D(\A^1) = A_1(\c) $ became clear in \cite{BW1}, where it was shown that,
under the Cannings-Holland bijection, the spaces $ \CC_n $ correspond to the orbits of the natural 
action of the Dixmier group $ \Aut_\c(A_1) $ on $ \R(A_1) $. A different approach to the problem of
classifying ideals of $ A_1 $, which does not use $ \Gr $, was developed in \cite{BW2}. The main idea 
of \cite{BW2} --- to use noncommutative projective geometry (specifically, 
a noncommutative version of Beilinson's equivalence) --- was inspired by \cite{LeB} and \cite{KKO} 
and was later generalized to many other classes of quantum algebras (see  \cite{BGK1}, \cite{BGK2}, 
\cite{NvdB}, \cite{NS}, \cite{BN1} and references therein). While the present paper was in preparation, 
a new very interesting paper \cite{BN} by Ben-Zvi and Nevins has appeared.
In \cite{BN}, the authors use a noncommutative Beilinson equivalence to classify torsion-free 
$\D$-modules on projective curves. Although this last problem is similar to (in fact, somewhat 
more general than) the one addressed in the present paper, our methods and results are different. 
Apart from describing explicitly the space $ \R(\D)$ of ideals, we also describe the action of the Picard 
group on $ \R(\D)$ and prove the equivariance of the Calogero-Moser correspondence. Comparing our constructions 
to those of \cite{BN} is an interesting problem, which will be hopefully clarified elsewhere. We should also 
mention that the methods of the present paper apply to a more general class of formally smooth
algebras, including the path algebras of quivers. Some of these versions of the Calogero-Moser
correspondence will be a subject of a forthcoming work. Finally, in the 
existing literature, there are (at least) two other definitions of
Calogero-Moser spaces associated to curves. The first one, due to V.~Ginzburg, employs the classical Hamiltonian
reduction (see \cite{FG}, or \cite{BN}, Def.~1.2) and is, in fact, closely related to ours (see Remark in the end 
of Section~\ref{GCMS}). 
The second, due to P.~Etingof (see \cite{E}, Example~2.19), is given in terms of
generalized Cherednik algebras (in the style of \cite{EG}). We will discuss the relation of 
Etingof's definition to ours in \cite{BC1}.

We now proceed with a summary of the contents of the paper.

Section~\ref{prelim} is preliminary: it introduces notation and
reviews the material needed for the rest of the paper. While most
results in this section are known, some are (apparently) new.
In particular, Theorem~\ref{lift} and Proposition~\ref{sm} 
did not appear in the literature in this form and generality. 

In Section~\ref{CM}, after recollections on differential
operators (Section~\ref{Diff}) and $K$-theoretic
classificaion of ideals of $\,\D\,$ (Section~\ref{SCI}), we define
the Calogero-Moser spaces $ \CC_n(X,\I) $ and establish their basic properties, 
including Theorem~\ref{IT6}.

The main results of the paper are gathered in Section~\ref{CMMap}. 
First, in Section~\ref{Re}, we explain the relation between the algebras 
$ \Pi^\blambda(B) $ and $ \D $, including Theorem~\ref{IP1}. Then, in 
Sections~\ref{ad} and~\ref{TMT}, we describe the action of the Picard 
group $ \Pic(\D) $ on the Calogero-Moser spaces $ \overline{\CC}_n(X) $ 
and state our main Theorem~\ref{ITmain}. 

The proof of Theorem~\ref{ITmain} occupies the whole of
Section~\ref{proof}. We refer the reader to the introduction of
that section for a summary of the proof.

In Section~\ref{explicit}, we give an alternative description of
the map $ \omega $ and consider a number of explicit examples.  
Perhaps, the most interesting example is that of a general plane curve (see
Section~\ref{genpc}). In this case, the varieties $ \CC_n(X, \I) $ 
can be described in terms of matrices, generalizing the classical 
Calogero-Moser matrices, and the map $ \omega $ is given by an 
explicit formula involving characteristic polynomials of these matrices (see \ref{genpc},
\eqref{kappz}). This last formula can be viewed as a generalization of
Wilson's formula for the rational Baker function of the KP integrable
hierarchy (see \cite{Wi}); however, our method of derivation is different from 
that of \cite{Wi}: it extends our earlier calculations in \cite{BC} in the case 
of the Weyl algebra.

The last section of the paper is an appendix written by G.~Wilson.
It clarifies the relation between deformed preprojective algebras
and rings of differential operators on curves, which, strictly
speaking, we did not use in this paper but probably should have.
As explained above, our map $\, \omega \,$ is
naturally induced by the algebra extension
$\,\btheta:\,\Pi^\blambda(B) \to \D \,$. Unfortunately, this
extension is not entirely canonical: it depends on the choice of
an identification of $\, \Pi^1(A) \,$ with $ \D = \D(X)$.
By a theorem of Crawley-Boevey (see \cite{CB}, Theorem~4.7), $\,
\Pi^1(A) \,$ is indeed isomorphic to $ \D$ as a filtered algebra,
but, in general, there seems to be no natural isomorphism between
these algebras. To remedy this problem, one should replace $ \D $
by the ring $ \D(\Omega_X^{1\!/ \!2})$ of {\it twisted}
differential operators on half-forms on $X$. As
was first observed by V.~Ginzburg (see \cite{G}, Sect.~13.4),
$\, \Pi^1(A) $ is {\it canonically} isomorphic to
$ \D(\Omega_X^{1\!/ \!2})\,$; however, the construction of the
isomorphism depends on a fact (Proposition~\ref{main} below) whose
proof in \cite{G} is very sketchy.  A complete proof can be
found in the Appendix, which may be read independently of the
rest of the paper.

\subsection*{Acknowledgement}
We would like to thank W.~Crawley-Boevey, P.~Etingof, V.~Ginzburg,
I.~Gordon, G.~Muller, R.~Rouquier, G.~Segal for interesting
discussions and comments. We are especially indebted to G.~Wilson:
without his ideas, questions and encouraging criticism this paper
would have probably never appeared. We also thank G.~Wilson for
kindly providing the appendix.

This research was supported in part by the NSF grant DMS 0901570.

\section*{Notation and Conventions}
Throughout this paper, we work over the base field $\c$. Unless
otherwise specified, an algebra means an associative algebra over
$ \c $, a module over an algebra $A$ means a {\it left} module
over $A$, and $ \Mod(A)$ denotes the category of such modules. All
bimodules over algebras are assumed to be symmetric over $ \c $,
and we use the abbreviation $\, \otimes\, $ for $\, \otimes_\c \,$
whenever it is convenient.

\section{Preliminaries}
\la{prelim}
\subsection{Deformed preprojective algebras}
\la{DPA} If $ A $ is an algebra, its tensor square $ \AA $ has two
commuting bimodule structures: one is defined by $\, a.(x \otimes
y).b = ax \otimes yb \,$ and the other by $\, a.(x \otimes y).b =
xb \otimes ay\,$, where $ a, b \in A$. We will refer to these
structures as {\it outer} and {\it inner}, respectively. Any
bimodule over $A$ can be viewed as either left or right module
over the enveloping algebra $ \eA := A \otimes \oA $;
if we interpret the outer bimodule structure on $ \AA $ as a
left $\eA$-module structure and the inner as a right one, then the
canonical map $\, \AA \to \eA \,$ is an isomorphism of $\eA$-bimodules.
We will
often use this isomorphism to identify $\, \AA \cong \eA \,$.

Following \cite{CEG}, we let $\, \DDer(A) := \Der(A, \AA) \,$ denote
the space of linear derivations $\, A \to \AA
\,$ taken with respect to the outer bimodule structure on $ \AA $.
This space is a bimodule with respect to the inner structure, so we can
form the tensor algebra $\, T_A \DDer(A)
\,$. Now, in $\DDer(A)$, there is a canonical derivation $\, \Delta = \Delta_A\,$,
sending  $ x \in A $ to $ (x \otimes 1 - 1 \otimes x) \in \AA $. For any $\,
\lambda \in A \,$, we can consider then the 2-sided ideal
$\,\langle \Delta - \lambda \rangle \,$ in $\,T_A \DDer(A)\,$
and define $\, \Pi^{\lambda}(A) := T_A \DDer(A)/\langle \Delta - \lambda \rangle \,$.
It turns out that, up to isomorphism, the algebra
$\,\Pi^\lambda(A)\,$ depends
only on the class of $ \lambda $ in the Hochschild homology  $\,
\HH_0(A) := A/[A,A]\,$ (see \cite{CB}, Lemma~1.2).
Moreover, instead of elements of $ \HH_0(A) $, it is
convenient to parametrize $ \Pi^\lambda(A) $ by the elements
of $\, \c \otimes_{\Z} K_0(A) \,$, relating this last vector space
to $ \HH_0(A) $ via a Chern character map. To be
precise, let $\, \Tr_A: K_0(A) \to \HH_{0}(A) \,$ be the
map, sending the class of a projective module $ P $ to the class
of the trace of any idempotent $\, e \in \Mat(n, A)\,$,
satisfying $\, P \cong A^{n} e $. By additivity, this extends to a
linear map $\, \c \otimes_{\Z} K_0(A) \to \HH_0(A) \,$
to be denoted also $ \Tr_A $. Following \cite{CB}, we call the
elements of $\, \c \otimes_{\Z} K_0(A) \,$ {\it weights} and
define the {\it deformed preprojective algebra of weight} $\,
\blambda \in \c \otimes_{\Z} K_0(A)\,$ by
\begin{equation}
\la{E20} \Pi^{\blambda}(A) := T_A \DDer(A)/\langle \Delta - \lambda \rangle \ ,
\end{equation}
where $\, \lambda \in A \,$ is any lifting of
$\,\Tr_A(\blambda)\,$ to $A$. Note, if $ A $ is commutative, then
$\, \HH_0(A) = A \,$, and $\, \lambda \,$ is uniquely
determined by $\,\Tr_A(\blambda)\,$.

The algebras $ \Pi^{\blambda}(A) $ are usually ill-behaved
unless one imposes some `smoothness' conditions on $A$.
In this paper, following \cite{KR}, we say that an algebra $ A $ is
{\it smooth} if it is
quasi-free and finitely generated: technically, this implies that
$ \Omega^1 A $
--- the kernel of the multiplication map of $ A $ --- is a
f. g. projective bimodule.
For basic properties and examples of the algebras $
\Pi^{\blambda}(A) $ the reader is referred to \cite{CB}. Here, we
state only one important theorem from \cite{CB}, which will play a
role in our construction. We recall that a ring homomorphism $
i:\, B \to A $ is called {\it pseudo-flat} if
$\,\Tor^B_1(A,\,A) = 0\,$. We also recall that any
ring homomorphism $ i:\, B \to A $ induces a homomorphism of
abelian groups $\, i^*: K_0(B) \to K_0(A)\,$, which extends
(by linearity) to a map of $\c$-vector spaces $\, i^*:\,\c
\otimes_{\Z} K_0(B) \to \c \otimes_{\Z} K_0(A)\,$.
\begin{theorem}[\cite{CB}, Theorem~9.3 and Corollary~9.4]
\la{TCB1} Let $\, i: B \to A \,$ be a pseudo-flat ring
epimorphism. Then, for any $\, \blambda \in \c \otimes_{\Z} K_0(B)
\,$, there is a canonical algebra map $\, \btheta:
\Pi^{\blambda}(B) \to \Pi^{\blambda'}(A)\,$, where $\, \blambda' =
i^{*}(\blambda)\,$. If $ B $ is smooth, then $ \btheta $ is
also a pseudo-flat epimorphism, and the diagram
\begin{equation}
\la{DD5}
\begin{diagram}[small, tight]
B                 &  \rTo^{}          &  A \\
\dTo^{}           &                         & \dTo_{} \\
\Pi^{\blambda}(B) & \rTo^{}          & \Pi^{\blambda'}(A) \\
\end{diagram}
\end{equation}
is a push-out in the category of rings.
\end{theorem}

We now prove a few general results on representations of deformed
preprojective algebras, which may be of independent interest. Our
first lemma is probably well known to the experts (see, for
example, \cite{CEG}); we record it to fix the notation.


%
\begin{lemma}
\la{Lcomm} If $A$ is smooth, then $ \Delta $ lies
in the commutator space $\, [A,\,\DDer(A)] \,$ of the
bimodule $ \DDer(A) $.
\end{lemma}
\begin{proof}
Composing the multiplication map $ \mu: \AA \to A $ with
derivations $ A \to \AA $ yields a linear map $ \mu_*: \DDer(A) \to \Der(A) $, with $\,
\Delta \in \Ker(\mu_*)\,$. This map factors through
the natural projection $ \DDer(A) \onto \DDer(A)_{\natural}\,$,
where $ \DDer(A)_{\natural}:= \DDer(A)/[A,\,\DDer(A)] \,$.
If $ A $ is smooth, the induced map
$ \bar{\mu}_*: \DDer(A)_{\natural} \to \Der(A) $
is an isomorphism. Indeed, identifying $ \AA \cong \eA $ and writing
$\,\Omega^1 A
\subseteq \eA \,$ for $\, \Ker(\mu)\,$, we have $\,\Der(A) \cong
\Hom_{\eA}(\Omega^1 A, A) \,$ and $\, \DDer(A) \cong
\Hom_{\eA}(\Omega^1 A, \eA) \,$. Under the last isomorphism, the
bimodule structure on $ \DDer(A) $ corresponds to the
natural right $ \eA$-module structure on $ (\Omega^1 A)^{\star} :=
\Hom_{\eA}(\Omega^1 A, \eA) $ and
$\,\DDer(A)_{\natural} \cong (\Omega^1 A)^\star \otimes_{\eA} A\,$.
The quotient map $ \bar{\mu}_* $ now becomes
$\, (\Omega^1 A)^\star \otimes_{\eA} A \, \to \,
\Hom_{\eA}(\Omega^1 A, A) \,$.
Since $ A $ is smooth, $ \Omega^1 A $ is a f.~g.
projective $\eA$-module, so the last map is an
isomorphism. This implies
that $ \Ker(\mu_*) = [A,\,\DDer(A)] $, and hence $\,\Delta \in
[A,\,\DDer(A)]$.
\end{proof}

For any $ \lambda \in A $, the algebra $ \Pi^{\lambda}(A) $ is an
$A$-ring: it is equipped with a canonical algebra homomorphism
$\, A \to \Pi^\lambda(A) \,$. Every representation of
$\Pi^\lambda(A)$ can thus be regarded as a representation of $A$.
Conversely, given a representation of $A$, one can ask whether it
lifts to a representation of $ \Pi^\lambda(A)$. The following
theorem provides a simple homological criterion for the existence
and uniqueness of such liftings.
\begin{theorem}
\la{lift} Let $ A $ be a smooth algebra, and let $\, \varrho: \, A
\to \End(V) \,$ be a representation of $ A $ on a (not necessarily
finite-dimensional) vector space $ V $. Then $ \varrho $ can be
extended to a representation of $ \Pi^{\lambda}(A) $ if and only
if the homology class of $\, \varrho(\lambda) \,$ in $\,
\HH_0(A,\,\End\,V) \,$ is zero, i.e. $\, \varrho(\lambda) \in
[\varrho(A),\,\End(V)]\,$. If it exists, an extension of $\,
\varrho \,$ to $\, \Pi^{\lambda}(A)\, $ is unique if and only if
$\, \HH_1(A,\,\End\,V) = 0 \,$.
\end{theorem}
\begin{proof}
We will use the notation of
Lemma~\ref{Lcomm}. Thus, for a fixed $ \lambda \in A $, we
identify $\, \Pi^{\lambda}(A) = T_A (\Omega^1 A)^\star/\langle
\Delta_A - \lambda \rangle \,$, with $ \Delta_A \in (\Omega^1
A)^\star $ corresponding to the natural inclusion $ \Omega^1 A
\into \eA $. A representation $\, \varrho: \, A \to \End(V) \,$
can be extended then to a representation of $ \Pi^\lambda(A) $ if
and only if there is an $A$-ring map $\, \tvarrho:\, T_A(\Omega^1
A)^\star \to \End(V) \,$, such that $ \tvarrho(\Delta_A) =
\varrho(\lambda) $. By the universal property of tensor algebras,
such a map is uniquely determined by its restriction to $
(\Omega^1 A)^\star $. Thus, regarding $ \End(V) $ as a bimodule
over $ A $ via $\, \varrho \,$, we conclude that $\, \varrho \,$
lifts to $ \Pi^\lambda(A) $  iff there is $\, \tvarrho \in
\Hom_{\eA}((\Omega^1 A)^\star ,\, \End\,V)$, mapping $ \Delta_A $
to $ \varrho(\lambda) $. Here, the bimodule $ \End(V) $ is
interpreted as a {\it right} $\eA$-module.

Now, since $A$ is smooth, the canonical map $\, \Omega^1 A \to
(\Omega^1 A)^{\star \star} \,$ is an isomorphism, and we can
identify $\,\Hom_{\eA}((\Omega^1 A)^\star ,\, \End\,V) \cong
\End(V) \otimes_{\eA} \Omega^1 A \,$.
Under this identification, the condition $\, \tvarrho(\Delta_A) =
\varrho(\lambda) \,$ becomes
\begin{equation}
\la{eee}
 \exists\, f_i \otimes d_i \in \End(V)
\otimes_{\eA}\Omega^1 A \ :\ \sum_i f_i \,\Delta_A(d_i) =
\varrho(\lambda) \ .
\end{equation}
Tensoring the exact sequence of $\eA$-modules $\, 0 \to \Omega^1 A
\to \eA \to A \to 0 \,$ with $ \End(V) $, we get
\begin{equation}
\la{exlift} 0 \to \HH_1(A,\, \End\,V) \to \End(V) \otimes_{\eA}
\Omega^1 A \xrightarrow{\partial} \End(V)
\stackrel{p}{\longrightarrow} \HH_0(A,\,\End\,V) \to 0 \,,
\end{equation}
with map in the middle given by $\, \partial:\, f \otimes d
\mapsto f\,\Delta_A(d) \,$, and $\, p \,$ being the canonical
projection. The condition \eqref{eee} now says that
$\,\varrho(\lambda) \in \im(\partial) $, and, by exactness of
\eqref{exlift}, this is equivalent to $\,\varrho(\lambda) \in
\Ker(p)\,$. Thus, $\, \varrho \,$ can be extended to $
\Pi^\lambda(A) $ if and only $ \varrho(\lambda) $ vanishes in $
\HH_0(A,\,\End\,V) $.

The fibre of $ \partial $ over $\,
\varrho(\lambda) \in \End(V) \,$ consists of different
liftings of the given action $ \varrho $ to $ \Pi^\lambda(A) $.
Again, by exactness of \eqref{exlift}, this fibre can be
identified with $\, \HH_1(A,\, \End\,V) \,$. In particular, if $
\varrho $ admits an extension to $ \Pi^\lambda(A) $, this
extension is unique if and only if $\, \HH_1(A,\, \End\,V) = 0\,$.
\end{proof}
As an immediate corollary of Theorem~\ref{lift}, we get
\begin{corollary}
If $\, \blambda \in \c \otimes_{\Z} K_0(A)\,$, then $\, \varrho:
\, A \to \End(V) \,$ can be extended to $\, \Pi^\blambda(A)\,$ if
and only if $\, \varrho_* \, \Tr_A(\blambda) = 0 \,$, where
$\,\varrho_*: \,\HH_0(A) \to \HH_0(A,\, \End\,V) \,$ is
the map induced by $ \varrho $ on Hochschild homology.
\end{corollary}

We now apply Theorem~\ref{lift} to finite-dimensional
representations. The next result is a generalization of
\cite{CB2}, Theorem~3.3, which deals with path algebras of
quivers.
\begin{proposition}
\la{liftfd} Let $A$ be a smooth algebra, and let $ \varrho: \, A
\to \End(V) \,$ be a representation of $A$ on a finite-dimensional
vector space $V$. Then $\, \varrho \,$ lifts to a representation
of $ \Pi^\lambda(A) $ if and only if the trace of
$\,\varrho(\lambda)\,$ on any $A$-module direct summand of $V$ is
zero. Moreover, if $\,\varrho \in \Rep(A, V) \, $ lifts, then the
fibre $ \,\pi^{-1}(\varrho) \,$ of the canonical map $\, \pi:\,
\Rep(\Pi^\lambda(A), V) \to \Rep(A, V) \,$  is isomorphic to $\,
\Ext_A^1(V, V)^* $.
\end{proposition}
\begin{proof}
The trace pairing on $ \End(V) $
yields a linear isomorphism $\,\End(V) \stackrel{\sim}{\to} \End(V)^* $,
which is a bimodule map with respect to the natural bimodule structures
on $ \End(V) $ and $ \End(V)^*$. This isomorphism restricts to $\,
\End_A(V) \stackrel{\sim}{\to} \HH_0(A,\,\End\,V)^*$, which, upon dualizing
with $ \c $, becomes
\begin{equation}
\la{biso}
\HH_0(A,\,\End\,V) \stackrel{\sim}{\to} \End_A(V)^*\ , \quad
\bar{f} \mapsto [\,e \mapsto \tr_{V}(ef)\,]\ .
\end{equation}

Now, let $ \varrho: A \to \End(V) $ be a representation of $ A $
on $V$ that lifts to $ \Pi^\lambda(A) $, and suppose that $V$ has
a direct $A$-linear summand, say $W$. By Theorem~\ref{lift}, the
class of $ \varrho(\lambda) $ in $ \HH_0(A,\,\End\,V) $ is zero,
and hence so is its image under \eqref{biso}. Taking $\, e \in
\End_A(V)\,$ to be a projection onto $W$, we get $\,
\tr_V[e\varrho(\lambda)] = \tr_W[\varrho(\lambda)] = 0 \,$,
which proves the first implication of the theorem.

For the converse, it suffices to consider only indecomposable
representations $ \varrho: \, A
\to \End(V) $. By Fitting's Lemma, $ \End_A(V) $ is
then a local ring: every $\, e \in \End_A(V) \,$ can be written as
$\, e = c\,\id_V + \btheta \,$, with $ c \in \c $ and $ \btheta $
being nilpotent. Now, if we assume that $\,
\tr_V[\varrho(\lambda)] = 0 \,$, then $\, \tr_{V}[e
\varrho(\lambda)] = 0 \,$ for any $\, e \in \End_A(V)\,$. The
class of $ \varrho(\lambda) $ in $ \HH_0(A,\,\End\,V) $ lies thus
in the kernel of \eqref{biso} and hence is zero. By
Theorem~\ref{lift}, we conclude that $\, \varrho \,$ lifts to a
representation of $\Pi^\lambda (A) $.

For the last statement, note that $\, \pi^{-1}(\varrho)
\cong \HH_1(A,\, \End\,V)\,$ by exactness of \eqref{exlift}.
On the other hand, we have
\begin{equation}
\la{hoch}
\HH_1(A,\, \End\,V) \cong \Tor^{A}_1(V^*,\,V) \cong
\Ext^1_A(V,V)^*\ ,
\end{equation}
which is standard homological algebra (see \cite{CE}, Cor.~4, p.~170,
and Prop.~VI, 5.3, respectively).
\end{proof}
\begin{remark}
In the special case, when $\, A \,$ is the path algebra of a
quiver, Proposition~\ref{liftfd} was proven earlier, by a different method,
in \cite{CB2}. With identifications \eqref{biso} and \eqref{hoch}, our basic
sequence \eqref{exlift} becomes
\begin{equation}
\la{exlift1}
 0 \to \Ext_A^1(V,\,V)^* \to \End(V) \otimes_{\eA} \Omega^1 A
\to \End(V) \to \End_A(V)^* \to 0 \ ,
\end{equation}
which, in the quiver case, agrees with \cite{CB2}, Lemma~3.1.
\end{remark}

\subsection{One-point extensions}
 \la{1-point}
If $ A $ is a unital associative algebra, and $ I $ a left module
over $ A $, we define the {\it one-point extension} of $ A $ by $
I $ to be the ring of triangular matrices
\begin{equation}
\la{E16} A[I] := \left(
\begin{array}{cc}
A & I \\
0 & \c
\end{array}
\right)
\end{equation}
with matrix addition and multiplication induced from the module
structure of $ I $. Clearly, $ A[I] $ is a unital associative
algebra, with identity element being the identity matrix. There
are two distinguished idempotents in $ A[I]\, $: namely
\begin{equation}
\la{idem01} e := \left(
\begin{array}{cc}
 1  & 0 \\
 0  & 0
\end{array}
\right) \quad \mbox{and} \quad \ei := \left(
\begin{array}{cc}
 0 & 0 \\
 0 & 1
\end{array}
\right) \ .
\end{equation}
If $ A $ is indecomposable (e.g., $A$ is a commutative integral
domain), then \eqref{idem01} form a complete set of primitive
orthogonal idempotents in $ A[I] $.

A module over $ A[I] $ can be identified with a triple $\, \bV =
(V,\, \Vi, \, \varphi) \,$, where $ V $ is an $A$-module, $ \Vi $
is a $\c$-vector space and $ \varphi: \, I \otimes \Vi \to V $ is
an $A$-module map. Using the standard matrix notation, we will
write the elements of $ \bV $ as column vectors $ (v, w)^T $ with
$ v \in V $ and $ w \in \Vi $; the action of $ A[I] $ is then
given by
$$
\left(
\begin{array}{cc}
a & b \\
0 & c
\end{array}
\right) \left(
\begin{array}{c}
v \\
w
\end{array} \right)
= \left(
\begin{array}{c}
a.v + \varphi(b \otimes w) \\
c w
\end{array} \right)\ .
$$
Now, if $\,\bU = (U, \,\Ui, \,\varphi_U)\,$ and $\,\bV = (V,
\,\Vi,\, \varphi_V)\,$ are two $A[I]$-modules, a homomorphism $\,
 \bU \to \bV \,$ is determined by a pair of maps
$\, (f, f_\infty)\,$, with $\, f \in \Hom_A(U, V) \,$ and $\,
f_{\infty} \in \Hom_\c(\Ui, \Vi)\,$, making the following diagram
commutative
\begin{equation}
\la{1-hom}
\begin{diagram}[small, tight]
I \otimes  \Ui                  & \rTo^{\varphi_U}     & U \\
\dTo^{\id \otimes f_\infty}     &                      & \dTo_{f} \\
I \otimes \Vi                   & \rTo^{\varphi_V}     & V \\
\end{diagram}
\end{equation}
If $ \bV $ is finite-dimensional, with $\,\dim_{\c} V = n\,$ and
$\,\dim_{\c} \Vi = n_{\infty}\,$, we call $\,\bn = (n,\,
n_{\infty}) \,$ the {\it dimension vector} of $ \bV $.

The next lemma gathers together basic properties of one-point
extensions.
\blemma
\la{LL6}
$(1)\,$ $ A[I] $ is canonically isomorphic to $\,
T_{\tA}(I) \,$, where $ \tA := A \times \c $.

$(2)\,$  If $ A $ is smooth and $ I $ is a f.~g. projective
$A$-module, then $ A[I] $ is smooth.

$(3)\,$ $\, I \mapsto A[I] \,$ is a functor from $ \Mod(A) $ to
the category of associative algebras.

$(4)\,$ The natural projection $\,i:\,A[I] \to A \,$ is a
flat ring epimorphism.

$(5)\,$ There is an isomorphism of abelian groups $ K_0(A[I])
\cong K_0(A) \oplus \Z \,$. \elemma
\begin{proof}

$(1)$ We identify $ \tA $ with the subalgebra of diagonal matrices
in $ A[I] $ and $ I $ with the complementary nilpotent ideal $ \tM
\subset A[I]\,$:
\begin{equation}
\la{EE1}
\tA = \left(
\begin{array}{cc}
A & 0 \\
0 & \c
\end{array}\right)\ , \quad
\tM := \left(
\begin{array}{cc}
0 & I \\
0 & 0
\end{array}\right)\ .
\end{equation}
By the universal property of tensor algebras, the inclusions $\,
\tA \into A[I] \,$ and $\, \tM \into A[I] \,$ can then be extended
to an algebra map $\, \phi: T_{\tA}(\tM) \to A[I] \,$, which is a
required isomorphism.

$ (2) $ By $(1)$ and \cite{CQ}, Prop.~5.3, it suffices to show that
$ \tM $ is a projective $\tA$-bimodule. But if $ I $ is a
projective $A$-module, then it is isomorphic to a direct summand
of a free module $ \, A \otimes V \,$ and  $\, \tM \,$ is
isomorphic to a direct summand of $\, \tA e \otimes V \otimes \ei
\tA \,$. The latter is a projective $\tA$-bimodule, since it is a
direct summand of $\, \tA \otimes V \otimes \tA \,$.

$(3)$ Any $A$-module map $\,f: I_1 \to I_2 \,$ gives rise to an
$\tA$-bimodule map $\, \tf: \tL_1 \to \tM_2 \,$. Identifying $
A[I_1] = T_{\tA}(\tL_1) $ and $ A[I_2] = T_{\tA}(\tL_2) $, we may
extend $\, I \mapsto A[I] \,$ to morphisms by $ A[f] :=
T_{\tA}(\tf) $. As $ T_{\tA} $ is a functor on bimodules, the
result follows.

$(4)$ The map $ i $ is given by
\begin{equation}
\la{E17} i\,:\, A[I] \to A  \ , \quad \left(
\begin{array}{cc}
a & b \\
0 & c
\end{array}
\right) \mapsto a\ .
\end{equation}
It is immediate from \eqref{E17} that $\, A \cong A[I]\,e\,$ as a
left $A[I]$-module via $ i $. Since $e$ is an idempotent, $\,A[I]\,e\,$ is
projective and hence flat.

$(5)$ The diagonal projection $\, \tbtheta:\, A[I] \to \tA \,$ has a
nilpotent kernel (equal to $\, \tM \,$). By \cite{B},
Prop.~IX.1.3, it then induces isomorphisms $\,K_i(A[I])
\cong K_i(\tA)\,$ for all $\, i \,$. In
particular, $\,K_0(A[I]) \cong K_0(A) \oplus \Z \,$.
\end{proof}
We will also need the next lemma relating homological properties of $A$ and $A[I]$.
\begin{lemma}
\la{euler} Let $A$ be a finitely generated hereditary algebra, and
let $\, B := A[I] \,$ be the one-point extension of $A$ by a f.~g.
projective $A$-module. Then, for any finite-dimensional
$B$-modules $ \,\bU = (U, \Ui) $ and $ \bV = (V, \Vi) \,$, we have
\begin{equation}
\la{eu}
\chi_{B}(\bU, \bV) = \chi_A(U, V) + \dim(\Ui)\, [\,\dim
(\Vi) - \dim\,\Hom_A(I, V)\,]\ ,
\end{equation}
where $\, \chi_A\,$ and $\, \chi_B\, $ denote the Euler
characteristics for the Ext-groups over the algebras $ A $ and $ B
$ respectively.
\end{lemma}
\begin{proof}  By \cite{Ben}, Th\'eor\`eme~1.1 (bis), there is a 5-term exact sequence
\begin{eqnarray}
\la{mvseq}
0 &\to&  \Hom_{B}(\bU, \bV) \to  \Hom_A(U, V)\, \oplus  \,
\Hom_{\c}(\Ui, \Vi) \to \\*[1ex]
&& \to \Hom_{\c}(\Ui,\, \Hom_{A}(I, V))
\to \Ext_B^1(\bU, \bV) \to \Ext_A^1(U, V) \to 0 \nonumber \ ,
\end{eqnarray}
and isomorphisms $\, \Ext_B^{k}(\bU, \bV) = \Ext_{A}^k(U,\,V) =  0 \,$ for all $\,k \ge 2\,$,
since $A$ is hereditary. Now, since $A$ is finitely generated, $ \Hom_A(U, V) $ and $ \Ext_A^1(U,V) $
are finite-dimensional whenever $ U $ and $ V $ are finite-dimensional. It follows from \eqref{mvseq} that
$\,\chi_B(\bU, \bV) \,$ is well defined and
related to $ \chi_A(U, V) $ by \eqref{eu}.
\end{proof}
\subsection{Representation varieties}
\la{modvar}
We recall the definition of representation varieties in the form
they appear in representation theory of associative algebras (see \cite{K}, Chap.~II, Sect.~2.7).

Let $ R $ be a finitely generated algebra. Fix $\,S
\,$, a finite-dimensional semisimple subalgebra of $ R $, and $ V
$, a finite-dimensional $S$-module. By definition, the {\it
representation variety} $\, \Rep_{S}(R,V) \,$ of $ R $ over $ S $
parametrizes all $R$-module structures on the vector space $V$
extending the given $S$-module structure on it. The $S$-module
structure on $ V $ determines an algebra homomorphism $ S \to
\End(V) $ making $ \End(V) $ an $S$-algebra. A point of
$\,\Rep_{S}(R,V)\,$ can thus be interpreted as an $S$-algebra map
$\, \varrho:\, R \to \End(V) $, and the tangent vectors at
$ \varrho $ can be identified with $S$-linear derivations
$\, R \to \End(V)$,\, i.~e.
\begin{equation}
\la{tgd}
T_{\varrho}\,\Rep_{S}(R, V) \cong \Der_{S}(R,\,\End\,V)\ .
\end{equation}
If $ S = \c $, we simply write $ \Rep(R, V) $ for $ \Rep_{\c}(R,
V) $. Clearly, $ \Rep(R, V) $ is an affine variety\footnote{Here, by an
affine variety we mean an affine scheme of finite type over $
\c$.}. For any semisimple $\,S \subseteq R\,$,
$\,\Rep_{S}(R,V) \,$ can then be identified with a fibre of the
canonical morphism $\,\Rep(R, V) \to
\Rep(S,V)\,$, and hence it is an affine variety as well.

The group $ \Aut_S(V) $ of $S$-linear automorphisms of $ V $ acts
on $\, \Rep_{S}(R,V) \,$ in the natural way, with scalars $\,
\c^{*} \subseteq \Aut_S(V) \,$ acting trivially. We set $\,
\G_S(V) := \Aut_S(V)/\c^* $. The
orbits of $\G_S(V) $ on $\, \Rep_{S}(R,V) \,$ are in 1-1
correspondence with isomorphism classes of $R$-modules, which are
isomorphic to $\, V \,$ as $S$-modules. The stabilizer of a point
$\, \varrho:\,R \to \End(V) \,$ in $\, \Rep_{S}(R,V) \,$ is
canonically isomorphic to $ \Aut_R(V_\varrho)/\c^* \subseteq
\G_S(V)$, where $ V_\varrho $ is the left $R$-module corresponding
to $ \varrho $. The space $\, \Rep_{S}(R,V)/\!/\,\G_S(V) \,$ of
closed orbits in $\, \Rep_{S}(R,V) \,$ is an affine variety, whose
points are in bijection with isomorphism classes of
semisimple $R$-modules $M$ isomorphic to $\, V \,$ as $S$-modules.

Typically, representation varieties of $R$ are defined over
subalgebras $ S = \oplus_{i \in I} \c \,e_i \subset R \,$
spanned by idempotents. A finite-dimensional $S$-module is then
isomorphic to a direct sum
$\, \c^{\bn} := \oplus_{i \in I} \c^{n_i} $, each $\, e_i \,$
acting as the projection onto the $i$-th component. The
variety $\,\Rep_{S}(R,\c^\bn)\,$,
which we simply denote by $ \Rep_S(R, \bn) $ in this case,
parametrizes all algebra maps $\, R \to \End(\c^{\bn})$, sending $
e_i $ to the projection onto $ \c^{n_i} $. The group $
\G_S(\c^\bn) $ (to be denoted $\, \G_S(\bn) \,$) is isomorphic to
$\,\prod_{i\in I} \GL(n_i, \c)/\c^* \,$, with $ \c^* $ embedded
diagonally.

We will need a few general results on geometry of representation
varieties. First, we recall the following well-known fact (see,
for example, \cite{G}, Proposition~19.1.4).
\bthm
\la{salg}
If $R$ is a smooth algebra, then $\Rep(R,\,V)$ is a smooth variety.
More generally, let $ S $ be a semisimple subalgebra of $R$, and let
$\, \varrho \in \Rep_S(R, V)
\subseteq \Rep(R, V)\,$. If $\, \Rep(R,V) \,$
is smooth at $\, \varrho \,$, then so is $\, \Rep_S(R,V) $.
\ethm

Now, given an algebra $A$, we set $\, R := T_A \DDer(A) \,$,
see Section~\ref{DPA}. If $ A $ is finitely generated, then so is
$R$, and we consider the variety $ \Rep(R, V) $ of representations
of $R$ on a vector space $V$. The following result is proved in
\cite{CEG}, Section~5 (see also \cite{vdB}).
\bthm
\la{repcot}
If $A$ is smooth, $\, \Rep(R,\,V) $ is canonically
isomorphic to the cotangent bundle of $\, \Rep(A,\,V) $.
In particular, $\,\Rep(R,\,V)\,$ is smooth.
\ethm

Recall that $ R $ contains a
distinguished element: the derivation $\,\Delta_A:\, A \to \AA\,$
defined by $\, x \mapsto x \otimes 1 - 1 \otimes x \,$. We write
$$
\mu:\, \Rep(R, V) \to \End(V)\ , \quad \varrho \mapsto
\varrho(\Delta_A)\ ,
$$
for the evaluation map at $ \Delta_A $ and consider its fibre $\,
F_\xi := \mu^{-1}[\mu(\xi)] \,$ for some fixed representation $
\xi \in \Rep(R,V) $.

\begin{proposition}
\la{sm} If $ A $ is smooth, then $ F_\xi $ is smooth at
$ \varrho \in \Rep(R, V) $ if and only if $ \End_R(V_\varrho) \cong \c$.
\end{proposition}
\begin{proof}
By Theorem~\ref{repcot}, the variety $\,\Rep(R,V) $ is smooth.
By Lemma~\ref{Lcomm}, we also have that
$\, \Delta_A \in [A,\,\DDer(A)] \,$, and therefore
$\, \tr_V [\varrho(\Delta)] = 0 \,$ for any $ \varrho \in \Rep(R, V)$.
It follows that
\begin{equation}
\la{eval} \mu:\, \Rep(R, V) \to \End(V)_0 \ ,
\end{equation}
where $\, \End(V)_0 := \{f \in \End(V)\, :\, \tr_V(f) = 0\}\,$.

To compute the differential of \eqref{eval} we use
\eqref{tgd} and also $\,T_{\mu}\, \End(V)_0 \cong \End(V)_0 \,$.
With these identifications, it is easy to see that
\begin{equation}
\la{momdiff} d \mu_{\varrho}:\ \Der(R, \,\End\,V) \to \End(V)_0\ ,
\quad \delta \mapsto \delta(\Delta_A) \ .
\end{equation}
Now, observe that the map $\, d \mu_{\varrho}^* \,$ dual to
\eqref{momdiff} fits into the commutative diagram
\begin{equation}
\la{Dmom}
\begin{diagram}
\End(V)^*_0 & \rTo^{d \mu_{\varrho}^*}     & \Der(R,\,\End\,V)^* \\
\uTo^{\tr_V}&                  & \uTo_{i(\Tr \,\hat{\omega})} \\
\End(V)/\c  & \rTo^{\ad} & \Der(R,\,\End\,V) \\
\end{diagram}
\end{equation}
with vertical arrows being {\it isomorphisms}. Here,  $\,
\tr_V \,$ comes from the trace pairing on $ \End(V) $ (and hence,
it is obviously an isomorphism), and $\, \ad\,$ is induced by the
canonical map, assigning to $\, f \in \End(V) \,$ the inner
derivation $\, \ad(f):\, a \mapsto [f,\,\varrho(a)]\,$. The
crucial isomorphism $\, i(\Tr \,\hat{\omega}) \,$ is constructed\footnote{
To avoid confusion, here we use the same
notation for this map as in \cite{CEG}.} in \cite{CEG} (see {\it
loc. cit.},  the proof of Theorem~6.4.3). Instead of repeating
this construction, we simply notice that \eqref{eval} can be
viewed as a {\it moment map} for the natural action of $
\GL(V)/\c^* $ on $ \Rep(R, \,V)$.  The commutativity of
\eqref{Dmom} is then equivalent to the defining equation
for moment maps (see \cite{CEG}, (6.4.7)). Now, it remains to note
that $ F_\xi $ is smooth at $ \varrho $ if and only if
$ d \mu_{\varrho} $ is surjective. This is equivalent to
$ d \mu_{\varrho}^*$ being injective, and hence, in view of \eqref{Dmom},
to $ \Ker(\ad) = \{0\} $. Since $\, \Ker(\ad) \cong \End_R(V)/\c \,$, 
this last condition holds if and only if $ \End_R(V_\varrho) \cong \c $. 
The proposition follows.
\end{proof}

\section{The Calogero-Moser Spaces}
\la{CM}

\subsection{Rings of differential operators}
\la{Diff} Let $ X $ be a smooth affine irreducible curve over $ \c
$, with coordinate ring $ \O = \O(X)$,  and let $ \D = \D(X) $ be
the ring of differential operators on $ X $. We recall that $ \D $
is a filtered algebra $\, \D = \bigcup_{k \ge 0}\, \D_k \,$, with
filtration components $\,  0 \subset \D_0 \subset \ldots \subset
\D_{k-1} \subset \D_{k} \subset \ldots \,$ defined inductively by
\begin{equation*}
\D_k := \{\, D \in \End_{\c}\,\O\ :\ [D, f] \in \D_{k-1}\
\mbox{for all}\ f \in \O \,\}\ .
\end{equation*}
The elements of $ \D_k $ are called {\it differential operators of
order} $ \le k $. In particular, $\, \D_0 = \O \,$ consists of
multiplication operators by regular functions on $X$, and
$ \D_1 $ is spanned by $ \O $ and the space $
\Der(\O)$ of derivations of $ \O \,$ (the algebraic vector fields
on $X$). As $ X $ is smooth, $ \O $ and $ \Der(\O) $ generate $ \D
$ as an algebra, and $ \D $ shares many properties with the Weyl
algebra $ A_1(\c) = \D(\A^1)$. For example,
like $ A_1(\c)$, $\, \D $ is a simple Noetherian domain of
global dimension $ 1 $; however, unlike $ A_1(\c) $, $\,\D\,$ has a
nontrivial $K$-group.

We write $\, \DO := \bigoplus_{k=0}^\infty \D_k/\D_{k-1} \,$ for
the associated graded ring of $\, \D\, $: this is a commutative
algebra isomorphic to the coordinate ring of the cotangent bundle
$ T^* X $ of $ X $. Given a $\D$-module $ M $ with a filtration $
\{ M_{k} \} $, we also write $\, \MO := \bigoplus_{k=0}^\infty
M_k/M_{k-1} $ for the associated graded $\DO$-module. Using the
standard terminology, we say that $ \{M_{k}\} $ is a {\it good}
filtration if $ \MO $ is finitely generated (see, e.~g. \cite{Bj}).

\subsection{Stable classification of ideals}
\la{SCI}
Let $K_0(X)$ and $\Pic(X)$ denote the Gro\-then\-dieck
group and the Picard group of $ X $ respectively. By definition, $
K_0(X) $ is generated by the stable isomorphism classes of
(algebraic) vector bundles on $X$, while the elements of $ \Pic(X)
$ are the isomorphism classes of line bundles. As $ X $ is affine,
we may identify $ K_0(X) $ with $ K_0(\O)$, the
Grothendieck group of the ring $\O$, and $ \Pic(X) $ with $
\Pic(\O)$, the ideal class group of $ \O $. There are two natural
maps $\, \rk: K_0(X) \to \Z \,$ and $\, \det: K_0(X) \to \Pic(X)
\,$ defined by taking the rank and the determinant of a vector
bundle respectively.  In the case of curves, it is well-known that
$\, \rk \oplus \det:\, K_0(X) \stackrel{\sim}{\to} \Z \oplus
\Pic(X)\,$ is a group isomorphism.

Now, let $ \R(\D) $ denote the set of isomorphism classes of
(nonzero) left ideals of $\D $. Unlike $ \Pic $ in the commutative
case, $ \R(\D) $ carries no natural structure of a group. However,
since $ \D $ is a hereditary domain, $ \R(\D) $ can be identified
with the space of isomorphism classes of rank $1$ projective
modules, and there is a natural map relating $ \R(\D) $ to $
\Pic(X) $:
\begin{equation}
\la{E10}
\gamma:\, \R(\D) \to K_0(\D) \xrightarrow{\iota_*^{-1}} K_0(X)
\stackrel{\,\det\,}{\longrightarrow} \Pic(X) \ .
\end{equation}
The first arrow in \eqref{E10} is the tautological map
assigning to the isomorphism class of an ideal in $
\R(\D) $ its stable isomorphism class in $K_0(\D)$.
The second arrow $\, \iota_*^{-1}
\,$ is the inverse of the Quillen isomorphism $\,\, \iota_*:\,
K_0(X) \stackrel{\sim}{\to} K_0(\D)\,$ induced by the inclusion
$\, \iota:\,\O \into \D\,$.
The role of  $ \gamma $ becomes clear from the following
result proved in \cite{BW}.
\begin{theorem}[see \cite{BW}, Proposition~2.1]
\la{T4} Let $ M $ be a projective $\D$-module of rank $1$ equipped
with a good filtration such that $ \MO $ is
torsion-free. Then

$(a)$ there is a unique (up to isomorphism) ideal $ \I_{M}
\subseteq \mathcal{O} \,$, such that $ \MO $ is isomorphic to a
sub-$\DO$-module of $  \DO \I_M $ of finite codimension (over
$\c$);

$(b)$ the class of $ \I_M $ in $ \Pic(X) $ and the codimension $\,
n := \dim_{\c}\,[\DO \I_M/\MO]\,$ are independent of the choice of
filtration on $ M $, and we have $\,\gamma [M] = [\I_M] \,$;

$(c)$ if $ M $ and $ N $ are two projective $\D$-modules of rank
$1$, then
$$
[\,M\,] = [\,N\,] \ \mbox{in}\ K_0(\D)
\quad \Longleftrightarrow \quad [\I_{M}] = [\I_{N}]
\ \mbox{in}\ \Pic(X)\ .
$$
\end{theorem}
Theorem~\ref{T4} shows that the fibres of  $ \gamma $ are
precisely the stable isomorphism classes of ideals of $ \D \,$:
thus, up to isomorphism in $ K_0(\D) $, the ideals of $ \D $ are
classified by the elements of $ \Pic(X) $. Our goal is to refine
this classification by describing the fibres of $ \gamma $ in
geometric terms. As we will see in Section~\ref{CMMap}, each fibre
$ \gamma^{-1}[\I] $ naturally breaks up into a countable union of
affine varieties $ \overline{\CC}_n(X, \I) $. In the next section,
we introduce these varieties and study their geometric properties.

\subsection{The definition of Calogero-Moser spaces}
\la{GCMS} Given a curve $ X $ with a line bundle $ \I $, we set $
A := \O(X) $ and form the one-point extension $ B := A[\I] $. By
Lemma~\ref{LL6}$(2)$, $\, B \,$ is a smooth
algebra, since so is $ A $ and $ \I$ is a f.~g.
projective $A$-module. As in Section~\ref{1-point}, we will
identify the subalgebra of diagonal matrices in $ B $ with $ \tA
:= A \times \c $, and let $ \tbtheta: \, B \onto \tA $ denote the
natural projection, see \eqref{EE1}. Since $ \tbtheta $ is a
nilpotent extension, it is suggestive to think of `$\Spec\,B $'
 as a (noncommutative) infinitesimal `thickening' of
$\,\Spec\,\tA = X \bigsqcup \,{\rm pt}\,$.

We now prove two auxiliary lemmas. The first lemma implies that $
B $ is determined, up to isomorphism, by the class of $\,\I\,$ in
$\, \Pic(X) $ and is independent of $ \I $ up to Morita
equivalence. The second lemma computes the Euler characteristics
for representations of $ B $, refining the result of Lemma~\ref{euler}.
\begin{lemma}
\la{L10}
For line bundles $ \I $ and $ \J $, the algebras  $ A[\I] $ and $
A[\J] $ are

$(a)$ Morita equivalent;

$(b)$ isomorphic if and only if $ \J \cong \I^\tau $ for some $
\tau \in \Aut(X) $, where $ \I^\tau := \tau^*\I$.
\end{lemma}
\begin{proof}
$(a) $ Given $ \I $ and $ \J $, we set $ \L := \Hom_A(\I,\,\J) $,
which is a line bundle on $ X $ isomorphic to $\, \J\,\I^\vee = \J
\otimes_A \I^\vee \,$, where $\, \I^\vee \,$ is the dual of $ \I
$. Then, we extend $ \L $ to a line bundle over $ \tA $, letting
$\,\tilde{\L} := \L \times \c \,$, and define $\, P := \tilde{\L}
\otimes_{\tA} B\,$, where $ B = A[\I]$. Clearly, $ P $ is a f.~g.
projective $B$-module. On the other hand, since $ A $ is a
Dedekind domain, $\, \L \oplus \L \cong A \oplus \L^2 \,$, where
$\,\L^2 = \L \otimes_A \L\,$, and hence $\, \tilde{\L} \oplus
\tilde{\L} \cong \tA \oplus \tilde{\L}^2 \,$. It follows that $ B
$ is isomorphic to a direct summand of $ P \oplus P $,  so $ P $
is a generator in the category of right $B$-modules. By Morita
Theorem, the ring $ B $ is then equivalent to $ \End_B(P) $. Now,
since $\,\End_B(P) = \Hom_B(\tilde{\L} \otimes_{\tA} B,\,P) \cong
\Hom_{\tA}(\tilde{\L},\,P)\,$ and $\, P \cong {\tilde \L}
\oplus (0,\,\L\,\I)  \,$ as a (right) $\tA$-module, we have
$\,\End_{B}(P) \cong A[\J]$.

$(b)$ If $ \J \cong \I $, then $\, A[\J] \cong A[\I]\,$, by
Lemma~\ref{LL6}$(3)$. Without loss of generality, we may
therefore identify $ \I $ and $ \J $ with ideals in $ A $. Given
$\, \tau \in \Aut(X) = \Aut(A)\,$, we have then $
\I^\tau  = \tau^{-1}(\I) $, and the natural map
$\,\tau^{-1}: A[\I] \to  A[\I^\tau]\,$ is a required isomorphism.
The converse statement is left as an exercise to the reader.
\end{proof}
\begin{lemma}
\la{Lchi} For any finite-dimensional $B$-modules $ \,\bU = (U, \Ui) $ and 
$ \bV = (V, \Vi) \,$, we have
$$ 
\chi_B(\bU, \bV)  = \dim(\Ui)\,[\,\dim (\Vi) - \dim(V)\,]\ .
$$
\end{lemma}
\begin{proof}
First, observe that $ \chi_A(U, V) = 0 $  for any pair of
finite-dimensional $A$-modules. Indeed, if $ U $ and $ V $ have
disjoint supports, then $\,\Hom_A(U, V) = \Ext_A^1(U, V) = 0 \,$,
and certainly $\, \chi_A(U,V) = 0 \,$. By additivity of $ \chi_A
$, it thus suffices to see that $\, \chi_A(U, V) = 0 \,$ for
modules $\, U \,$ and $\, V \,$ supported at one point. If $ \m $
is the maximal ideal of $ A $ corresponding to that point, we have
$\,\Ext^i_A(U,V) \cong \Ext_{A_\m}^i(U,V)\,$ and
$\,\Ext^i_{A_\m}(U, V) \cong \Tor_i^{A_\m}(V^*, U)^* \,$
for all $\, i \geq 0 \,$. Thus
$$
\chi_A(U,V) = \chi_{A_\m}(U,V) = \sum\, (-1)^i \dim_\c
\Tor_i^{A_\m}(V^*, U) \ .
$$
The vanishing of $ \chi_{A}(U,V) $ follows now from standard
intersection theory, since $A_\m$ is a regular local ring of
(Krull) dimension $1$, while $\, \dim_{A_\m}(U) + \dim_{A_\m}(V^*)
= 0 \,$ (see \cite{S}, Ch.~V, Part~B, Th.~1).

Identifying $ \I $ with an ideal in $ A $ and dualizing
$\, 0 \to \I \to A \to A/\I \to 0 \,$ by $ V $, we get
\begin{equation}
\la{homA}
\dim_\c\Hom_A(\I, V) = \dim_\c(V) - \chi_A(A/\I, V) =
\dim_\c(V)\ .
\end{equation}
The result follows now from Lemma~\ref{euler}.
\end{proof}

Next, we introduce deformed preprojective algebras over $B$. For
this, we need to compute the trace map $\,\Tr_B : K_0(B)
\to \HH_0(B) $. Recall that $\,\Tr_{*} :\, K_0 \to \HH_0 \,$
is a natural transformation of functors on the category of associative
algebras, so $\, \tbtheta:\, B \to \tA \,$ gives
rise to the commutative diagram
\begin{equation}
\la{D6}
\begin{diagram}[small, tight]
K_0(B)             & \rTo^{\Tr_B}     & \HH_0(B) \\
\dTo^{}&                  & \dTo_{} \\
K_0(\tA)   & \rTo^{\Tr_{\tA}\ } &
\HH_0(\tA) \\
\end{diagram}
\end{equation}
The two vertical arrows in \eqref{D6} are isomorphisms: the first
one is given by Lemma~\ref{LL6}$(6)$, while the second has the
obvious inverse (induced by the inclusion $\, \tA \into B \,$).
We will use these isomorphisms to identify $\,\HH_0(B)
\cong \HH_0(\tA)= \tA \subset B \,$ and
\begin{equation}
\la{K}
K_0(B) \cong K_0(\tA) \cong \Z \oplus \Z \oplus \Pic(X)\ .
\end{equation}

Now, for any commutative algebra (e.g., $ \tA $), the trace map
factors through the rank. Hence, with above
identifications, $\,\Tr_{B}\,$ is completely determined by its
values on the first two summands in \eqref{K}, while vanishing on
the last. Since $\,\Tr_B[(1,0)] = e \,$ and $\,\Tr_B[(0,1)] = \ei
\,$, the linear map $\, \Tr_B: \c \otimes_{\Z} K_0(B) \to \HH_0(B)
\,$ takes its values in the two-dimensional subspace $ S $ of $ B
$ spanned by  $ e $ and $ \ei $. Identifying $ S $
with $ \c^2$, we may regard the vectors $\, \blambda :=
(\lambda,\, \lambdai) = \lambda e + \lambdai \ei \in S \,$ as
weights for the family of deformed preprojective algebras
associated to $ B \,$:
\begin{equation}
\la{dpa1}
\Pi^{\blambda}(B) = T_{B}\DDer(B)/ \langle \Delta_B
- \blambda \rangle\ .
\end{equation}

Since $ A $ is an integral domain, $\,\{e,\, \ei\}\,$ is a
complete set of primitive orthogonal idempotents in $
\Pi^{\blambda}(B) $, and $\, S = \c e\, \oplus\, \c \ei \,$ is the
associated semisimple subalgebra of $ \Pi^{\blambda}(B) $.
Now, for each $\, \bn = (n,\, n_\infty) \in \N^2$, we form the
variety $ \Rep_{S}(\Pi^{\blambda}(B), \bn) $ of representations of
$ \Pi^{\blambda}(B) $ of dimension vector $ \bn $ and, with
notation of Section~\ref{modvar}, define
\begin{equation}
\la{E25}
\CC_{\bn, \blambda}(X, \I) := \Rep_{S}(\Pi^{\blambda}(B),
\bn)/\!/ \,\G_S(\bn)\ .
\end{equation}
Thus, $\, \CC_{\bn, \blambda}(X, \I) \,$ is an affine scheme,
whose (closed) points are in bijection with isomorphism classes of
semisimple $\Pi^{\blambda}(B)$-modules of dimension vector $ \bn
$.
\blemma
\la{isv}
For any line bundles $ \I $ and $ \J $, the
schemes $\,\CC_{\bn, \blambda}(X, \I) \,$ and $\, \CC_{\bn,
\blambda}(X, \J) \,$ are isomorphic.
\elemma
\bproof  By Lemma~\ref{L10},  $ A[\I] $ and $ A[\J] $
are Morita equivalent: the corresponding equivalence is given by
\begin{equation}
\la{egg}
\Mod\,A[\I] \stackrel{\sim}{\to} \Mod\,A[\J]\ , \quad
\bV \mapsto
\tilde{\L} \otimes_{\tA} \bV \ ,
\end{equation}
where  $\, \tilde{\L} = \J \I^\vee \times \c \,$. The functor
\eqref{egg} induces an isomorphism of vector spaces: $\,
\HH_0(A[\I]) \stackrel{\sim}{\to} \HH_0(A[\J])\,$, which restricts
to the identity on $\, S \subset \tA \,$. By \cite{CB},
Cor.~5.5, it can then be extended (non-canonically) to a
Morita equivalence between  $\,
\Pi^{\blambda}(A[\I])\,$ and $\, \Pi^{\blambda}(A[\J])\,$ for any
$ \blambda \in S $. Now, if $ \bV = (V, \, \Vi) $ with $\,\dim_\c
V < \infty \,$, we have $\, \tilde{\L} \otimes_{\tA} \bV = ( \J
\I^\vee \otimes_A V, \,\Vi) \,$, so by formula \eqref{homA},
$$
\dim_\c[\J \I^\vee \otimes_A V] = \dim_\c\,\Hom_A(\I\J^\vee,\,V)
= \dim_\c (V) \ .
$$
This shows that \eqref{egg} preserves dimensions,
and its extension to $ \Pi^\lambda $ induces thus an isomorphism:
$\,\CC_{\bn, \blambda}(X, \I) \stackrel{\sim}{\to} \CC_{\bn,
\blambda}(X, \J) \,$. \eproof

The next lemma is a generalization of \cite{CBH}, Lemma~4.1\,: it
implies that $ \CC_{\bn, \blambda}(X, \I) $ is empty unless $\,
\blambda\cdot \bn := \lambda\, n + \lambdai n_\infty\,$ is zero.
\begin{lemma}
\la{P6} If $ \blambda\cdot \bn \not= 0 $, there are no
representations of $\Pi^{\blambda}(B)$ of dimension  $ \bn $.
\end{lemma}
\begin{proof}
If $\, \bV = V \oplus \Vi $ is a  $\Pi^{\blambda}(B)$-module of
dimension  $ \bn $, then  $ e $ and $ \ei $ act on $ \bV $
as projectors onto $ V $ and $ \Vi $ respectively. The trace of
$\, \blambda = \lambda\, e + \lambdai \, \ei \in B \,$ on $ \bV $
is then equal to $ \blambda \cdot \bn $, and it must be zero, by
Proposition~\ref{liftfd}.
\end{proof}

\begin{example}
\la{Ex1} Let $ X $ be the affine line $ \A^{\! 1} $. Any line
bundle $ \I $ on $ X $ is trivial. So, choosing a coordinate on $
X $, we may identify $\, A \cong \c[x]\,$ and $\, \I \cong \c[x]
$. The one-point extension of $ A $ by $ \I $ is then isomorphic
to the matrix algebra
$$
A[\I] \cong \left(
\begin{array}{cc}
\c[x] & \c[x]\\*[1.1ex]
0    & \c
\end{array}
\right)\ ,
$$
which is, in turn, isomorphic to the path algebra $ \c Q $ of the
quiver $ Q $ consisting of two vertices $\,\{0, \,\infty\}\,$ and
two arrows $\, X:\, 0 \to 0 \,$ and $\, \v:\, \infty \to 0 $. In
fact, the map sending the vertices $0$ and $\infty $ to the
idempotents $ e $ and $ \ei $ and
$\, X \mapsto \left(
\begin{array}{cc}
x & 0\\
0 & 0
\end{array}
\right)\,$,
$\ \v \mapsto \left(
\begin{array}{cc}
0 & 1\\
0 & 0
\end{array}
\right)$,
extends to an algebra isomorphism $\,\c Q \stackrel{\sim}{\to}
A[\I]\,$.

Now, let $ \bar{Q} $ be the double quiver of $ Q $ obtained by
adding the reverse arrows $ Y := X^* $ and $ \w := \v^* $ to the
corresponding arrows of $Q$. Then, for any $\,\blambda =
\lambda\,e + \lambdai \ei \,$, with $ (\lambda, \,\lambdai) \in
\c^2 $, the algebra $\, \Pi^{\blambda}(Q) \,$ is isomorphic to the
quotient of $ \c\bar{Q} $ modulo the relation $\, [X,Y] + [\v, \w]
= \blambda \,$ (see \cite{CB}, Theorem~3.1). The ideal generated
by this  relation is the same as the ideal generated by the
elements $\, [X,Y] + \v\,\w - \lambda \,e \,$ and $\, \w\,\v
+\lambdai \ei \,$. Thus, the $ \Pi^{\blambda}(Q)$-modules can be
identified with representations $ \bV = V \oplus \Vi $ of $
\bar{Q} $, in which linear maps $ \,\sX, \sY \in \Hom(V,\, V),
\,\sv \in \Hom(\Vi,\, V),\, \sw \in \Hom(V,\, \Vi) \,$, given by
the action of $\, X, Y, \v, \w \,$, satisfy
\begin{equation}
\la{E26} [\sX,\,\sY] + \sv \,\sw  = \lambda \, \id_{V}\quad
\mbox{and} \quad \sw\,\sv = -\lambdai \, \id_{\Vi} \ .
\end{equation}
Now, taking $ \blambda = (1, -n) $, it is easy to see that all
representations of $ \Pi^\blambda(Q) $ of dimension vector $ \bn =
(n,\,1) $ are simple, and the varieties $ \CC_{\bn,
\blambda} $ coincide (in this special case) with the classical
Calogero-Moser spaces $ \CC_n $. This coincidence was first
noticed by W.~Crawley-Boevey (see \cite{CB1}, remark on p.~45).
For explanations and further discussion of this example we refer
the reader to \cite{BCE}.
\end{example}

Motivated by the above example, we will be interested in
representations of $ \Pi^{\blambda}(B)$ of dimension $ \bn = (n,1)
$. By Lemma~\ref{P6}, such representations may exist only if $\,
\blambda = 0 \,$ or $\, \blambda = (\lambda, -n\lambda) $, with $
\lambda \not=0 $. In this last case,
$\Pi^{\blambda}(B)$'s are all isomorphic to each other, so without
loss of generality we may assume $\lambda=1$.

\begin{proposition}
\la{L11} Let $ \blambda = (1, -n) $ and $ \bn = (n,1) $ with $ n
\in \N $. Then, for any $\,\I \,$, the algebra $\Pi^{\blambda}(B)
$ has modules of dimension vector $ \bn $. Every such module is
simple.
\end{proposition}
\begin{proof}
On any $B$-module of dimension  $ \bn $, the element $\, \blambda
= e - n\, \ei  \in B \,$ will act with zero trace. Thus, by
Proposition~\ref{liftfd}, it suffices to see that there exist {\it
indecomposable} $B$-modules of this dimension vector. Now, a
$B$-module structure on $
\bV = V \oplus \Vi $ is determined by an $A$-module homomorphism
$\, \varphi_{V}:\, \I \otimes \Vi \to V \,$. If $ \bV $ is
decomposable with $\, \dim(\Vi) = 1 \,$, then one of its
summands must be of the form $ \bV' = V' \oplus \Vi $, where $ V'
$ is an $A$-module summand of $ V $ of dimension $ < n $. In that
case,  $\, \im(\varphi_V) \subseteq V' \subsetneq V $.
Thus, for constructing an indecomposable $B$-module of dimension $
\bn $, it suffices to construct a torsion $A$-module $ V $ of
length $ n $ together with a {\it surjective} $A$-module map $
\varphi: \I \to V $. Geometrically, this can be done as follows.

Identify $ \I $ with an ideal in $A$ and fix $\, n \,$ distinct
points $\, x_1,\, x_2,\,\ldots\, x_n\,$ on $X$ outside the zero
locus of $\,\I\,$. Let $\, V := A/\J\,$, where $\,\J\,$ is the
product of the maximal ideals $ \m_i \subset A $ corresponding to
$x_i$'s. Clearly, $\,A/\J \cong \oplus_{i=1}^n A/\m_i \, $ and $
\dim_{\c} V = n $. Now, since $A$ is a Dedekind domain and $\, \I
\not\subset \m_i \,$ for any $ i $, we have $\, (A/\J) \otimes_A
(A/\I) \cong \oplus_{i=1}^n  (A/\m_i) \otimes_A (A/\I)= 0 \,$ and
$\, \Tor^A_1(A/\J, A/\I) \cong (\I \cap \J)/\I\,\J = 0 \,$, so the
canonical map $\, V \otimes_A \I \to V $ is an isomorphism. On the
other hand, as $ V $ is a cyclic $A$-module, $ \I$ surjects
naturally onto $\, V \otimes_A \I \,$. Combining $\, \I \onto V
\otimes_A \I \stackrel{\sim}{\to} V \,$, we get the required
 $ \varphi $. This proves the first part of the proposition.

Now, let $ \bV = V \oplus \Vi $ be any $\Pi^{\blambda}(B)$-module
of dimension vector $ \bn $. Let $\, \bV' \,$ be a submodule of $
\bV $ of dimension vector $\, \bk = (k,\,k_\infty)\,$ (say). By
Lemma~\ref{P6}, we have then $\, \blambda \cdot \bk = k - n\,
k_\infty = 0 \,$. Since $\, 0 \leq k_\infty \leq n_\infty = 1 \,$,
there are only two possibilities: either $ \bk = 0 $ or $ \bk =
\bn $, i.e. $ \bV' $ is either $ 0 $ or $ \bV $. Hence $ \bV $ is
a simple module.
\end{proof}
\begin{remark}
1. The above argument shows that a $ B$-module $ \bV  $ of
dimension vector $ \bn = (n,1) $ lifts to a module over $
\Pi^\blambda(B) $ if and {\it only if}\, it is indecomposable.\\
\noindent
 2. If $ \bV $ is a $B$-module  with a surjective structure
map $\, \varphi_V: \I \otimes \Vi \to V \,$, then $\, \End_B(\bV)
\subseteq \End(\Vi) \,$. (This follows immediately from the
diagram \eqref{1-hom}, characterizing $B$-module homomorphisms.)
Hence the modules $ \bV $ constructed in Proposition~\ref{L11} are
actually indecomposables with $\, \End_B(\bV) \cong \c \,$.
\end{remark}
\begin{definition}
\la{CMv} The variety $\, \CC_{\bn, \blambda}(X,\I) \,$ with
$\,\blambda = (1,-n) $ and $ \bn = (n,1) $ will be denoted $
\CC_n(X,\I) $ and called the {\it $n$-th Calogero-Moser space} of
type $ (X,\I) $.
\end{definition}
In view of Proposition~\ref{L11}, the varieties $ \CC_n(X,\I) $
parametrize the isomorphism classes of simple $
\Pi^{\blambda}(B)$-modules of dimension $ \bn = (n,1) $; they are
non-empty for any $\, [\I] \in \Pic(X) \,$ and $ n \geq 0 $. In
the special case, when $ X $ is the affine line, $ \CC_n(X,\I) $
coincide with the ordinary Calogero-Moser spaces $ \CC_n $ (see
Example~\ref{Ex1}).\\

\begin{remark}
\la{sympl}
It follows from our discussion in Section~\ref{modvar} (see also
\cite{CEG} and \cite{vdB}) that the variety $ \CC_n(X, \I) $ 
can be obtained by symplectic reduction from the cotangent bundle of 
$ \Rep(B, \bn) $. This links our definition of Calogero-Moser spaces
to the one proposed by V.~Ginzburg (see \cite{BN}, Definition~1.2).
\end{remark}
\subsection{Smoothness and irreducibility}
\la{SmIr}
One of the main results of \cite{Wi} says that each
$\CC_{n} $ is a smooth affine irreducible variety of dimension $
2n $. Theorem~\ref{T6} below shows that this holds in general, for
an arbitrary curve $X$. To prove the irreducibility we will use
the approach of Crawley-Boevey \cite{CB2}, the starting point of
which is the following observation.
\begin{lemma}[\cite{CB2}, Lemma~6.1]
\la{irredu} If $X$ is an equidimensional variety, $Y$ is an
irreducible variety and $\,f:\,X \to Y\,$ is a dominant morphism
with all fibres irreducible of constant dimension, then $ X $ is
irreducible.
\end{lemma}
\begin{theorem}
\la{T6} For each $ n \geq 0 $ and $ [\I]\in \Pic(X) $, $\,
\CC_{n}(X,\I) \,$ is a smooth affine irreducible variety of
dimension $ 2n $.
\end{theorem}
\begin{proof}
The varieties $ \CC_n(X,\I) $ are affine by definition; we need
only to show that these are smooth and irreducible. Fix $ n \in \N $
and $ [\I] \in \Pic(X) $. To simplify the
notation write $\, \Pi \,$ for $\, \Pi^\blambda(B) \,$ with $
\blambda = (1,-n) $. Then, by Proposition~\ref{L11}, every
$\Pi$-module $ \bV $ of dimension $\,\bn = (n,1) \,$ is simple,
so, by Schur Lemma, $\, \End_{\Pi}(\bV) \cong \c\,$ and $
\,\Aut_{\Pi}(\bV) \cong \c^* $. The last isomorphism implies that
every point of $ \Rep_S(\Pi, \bn) $ has trivial stabilizer in $
\G_S(\bn) $, i.~e. the natural action of $ \GL(\bn) $ on $
\Rep_S(\Pi, \bn) $ is free. In that case, by Luna's Slice Theorem
(see \cite{Lu}, Corollaire~III.1.1), the quotient variety $\,
\CC_n(X,\I) = \Rep_S(\Pi, \bn)/\!/\,\G_S(\bn) \,$ will be smooth
if so is the original variety $\, \Rep_S(\Pi, \bn)\,$. Now, to see
that $\, \Rep_S(\Pi, \bn)\,$ is smooth, it suffices to see, by
Theorem~\ref{salg}, that $\, \Rep(\Pi, \bn)\,$ is smooth, and that
follows from Proposition~\ref{sm} of Section~\ref{modvar}. In
fact, let $ R := T_B \DDer(B) $, and let $\, \sigma:\, R \onto
\Pi \,$ be the canonical projection. Then $ \sigma $ induces the
closed embedding of affine varieties $ \,\sigma_*:\,\Rep(\Pi, \bn)
\into \Rep(R, \bn)\,$, whose image is a fibre of the evaluation
map \eqref{eval}. Since for every $\, \varrho \in \im(\sigma_*)
\,$, we have $\, \End_R(\bV) \cong \End_{\Pi}(\bV) \cong \c \,$,
the assumption of Proposition~\ref{sm} holds, and the result
follows.

Now, we show that $\,\CC_n(X, \I)\,$ is irreducible of dimension
$\,2n\,$. For this, we examine first the varieties $ \Rep_S(B,
\bn) $ and $ \Rep_S(\Pi, \bn) $. Since $B$ is smooth, $ \Rep_S(B,
\bn) $ is smooth, i.~e. for every point $\, \varrho \in \Rep_S(B,
\bn) \,$, we have
\begin{equation}
\la{dimrep} \dim_{\varrho}\, \Rep_S(B, \bn) = \dim_{\c}
T_\varrho\,\Rep_S(B, \bn)\ ,
\end{equation}
where $\,\dim_\varrho\,$ stands for the local dimension and $\,
T_\varrho\, \,$ for the Zariski tangent space of $ \Rep_S(B, \bn)
$ at $ \varrho $. To evaluate the dimension of this space
we identify $\, T_\varrho\,\Rep_S(B, \bn) \cong \Der_S(B, \End\,\bV) \,$,
as in \eqref{tgd}, and consider the standard exact sequence
$$
 0 \to \End_B(\bV) \to \End_S(\bV) \to \Der_S(B, \,\End\,\bV) \to
\HH^1(B, \,\End\,\bV) \to 0\ .
$$
Identifying now terms in this sequence
$\, \End_S(\bV) \cong \Mat(n, \c)\times \c \,$, $\, \HH^1(B,
\,\End\,\bV) \cong  \Ext^1_{B}( \bV,\,\bV) \,$ (see \cite{CE},
Prop.~4.3, p.~170), and using Lemma~\ref{Lchi}, we get
\begin{equation}
\la{dimder} \dim_\c \Der_S(B,\,\End\,\bV) = n^2 + 1 - \chi_B(\bV,
\bV) =  n^2 + n \ .
\end{equation}
Thus $\, \Rep_S(B, \bn) \,$ is a smooth equidimensional variety of
dimension $\, n^2 + n \,$. To see that it is actually irreducible,
we apply Lemma~\ref{irredu} to the canonical projection $\, f:\,
\Rep_S(B, \bn) \to \Rep(A, n) \,$. In this case, the assumptions
of Lemma~\ref{irredu} are easy to verify: since $X$ is
irreducible, so is clearly $\,\Rep(A,n)\,$, and the fibres of  $\,
f \,$ over each $\, V \in \Rep(A,n)\,$ can be identified with the
vector spaces $\,\Hom_A(I, V) \,$ and, hence, are all irreducible
of the same dimension $\, n \,$, by formula \eqref{homA}.

Next, we consider the restriction map $\,\pi:\, \Rep_S(\Pi, \bn)
\to \Rep_S(B, \bn) \,$. As remarked above, the image of $\,\pi \,$
consists exactly of indecomposable modules in $\, \Rep_S(B,
\bn)\,$, while each (non-empty) fibre $\, \pi^{-1}(\bV) \,$ is
isomorphic, by Proposition~\ref{liftfd}, to a coset of
$\,\Ext^1_B(\bV, \bV)^* \,$ and is thus irreducible of dimension
\begin{equation}
\la{dimfib} \dim\,\pi^{-1}(\bV) = \dim_\c \End_B(\bV) -
\chi_B(\bV,\bV) = \dim_\c \End_B(\bV) + n - 1\ .
\end{equation}
Now, let $\, \U \, $ be the subset of $ \,\Rep_S(B, \bn) \,$
consisting of modules $ \bV $ with $\, \End_B(\bV) \cong \c\,$. As
explained in Remark~2 (after Proposition~\ref{L11}), this subset
is non-empty. By Chevalley's Theorem (see, e.g., \cite{CB3}, p.
15), the function $\, \bV \mapsto \dim_\c \End_B(\bV)\,$ is upper
semi-continuous on $\, \Rep_S(B, \bn) \,$, i.~e.
$$
 \{\bV \in \Rep_S(B, \bn) \,:\, \dim_\c \End_B(\bV,\bV) \geq n \}
$$
are closed sets for all $\, n \in \N \,$. Hence $\, \U \,$ is open
in $\, \Rep_S(B, \bn) \,$ and therefore dense, since $\, \Rep_S(B,
\bn) \,$ is irreducible. As $\, \U \subseteq \im(\pi) \,$, this
implies that $\, \pi \,$ is dominant.

Now, $ \pi^{-1}(\U) $ is an open subset of $ \Rep_S(\Pi, \bn) $,
whose local dimension at every point $\, \varrho \in
\pi^{-1}(\U)\,$ is equal, by \eqref{dimfib}, to
$$
\dim_\varrho \,\pi^{-1}(\U)\, = \dim\,\U +
\dim\,\pi^{-1}(\pi(\varrho)) = \dim\,\Rep_S(B,\bn) + n =  n^2 +
2n\ .
$$
Thus  $\, \pi^{-1}(\U)\,$ is equidimensional and therefore, by
Lemma~\ref{irredu}, irreducible. We claim that $\, \pi^{-1}(\U)
\,$ is  dense in $\,\Rep_S(\Pi, \bn) \,$. Indeed, since $ \im(\pi)
$ coincides with the set of indecomposable $B$-modules in $\,
\Rep_S(B, \bn) \,$, we have
$$
\dim\, \pi^{-1}(\im(\pi)\!\setminus \U) < \dim\, \pi^{-1}(\U) =
n^2 + 2n\ .
$$
On the other hand,  the variety $\,\Rep_S(\Pi, \bn) \,$ can be
identified with a fibre of the evaluation map $\, \mu: \,\Rep_S(R,
\bn) \to \End_S(\bV)_0 \,$, see \eqref{eval}, so any irreducible
component of it has dimension at least
$$
\dim\,\Rep_S(R, \bn) - \dim\, \End_S(\bV)_0 = 2(n^2 + n) - n^2 =
n^2 + 2n \ .
$$
(Here, we calculated $\, \dim\,\Rep_S(R, \bn)\,$ using the identification
of Theorem~\ref{repcot}: $\ \Rep_S(R, \bn) \cong T^*\Rep_S(B, \bn)\,$.)
Thus, $\,\Rep_S(\Pi, \bn)\,$ must coincide with the closure of
$\,\pi^{-1}(\U)\,$, and hence  is also irreducible of dimension
$\, n^2 + 2n $. This certainly implies the irreducibility of $\,
\CC_n(X,\I) \,$, since $\, \CC_n(X,\I) \,$ is a quotient of
$\,\Rep_S(\Pi, \bn) \,$ by a free action of $ \G_S(\bn) $.

Finally, we have $\, \G_S(\bn) \cong [\,\GL(n, \c)\times
\GL(1,\c)\,]/\c^* \cong \GL(n,\c)\,$, so
$$
\dim \,\CC_n(X, \I) = \dim\,\Rep_S(\Pi,\,\bn) - \dim\,\G_S(\bn) = n^2
+ 2n - n^2 = 2n\ .
$$
This completes the proof of the theorem.
\end{proof}

\section{The Calogero-Moser Correspondence}
\la{CMMap}

\subsection{Recollement}
\la{Re} We begin by clarifying the relation between
the algebras $ \Pi^\blambda(B) $ and the ring $\D$ of
differential operators on $X$ (see also Appendix).
\blemma
\la{recoll}
There is a canonical map $\, \btheta:
\Pi^{\blambda}(B) \to \Pi^1(A) \,$, which is a surjective
pseudo-flat ring homomorphism, with $\, \Ker(\btheta) = \langle\,
\ei\,\rangle\,$.
\elemma
\begin{proof}
By Lemma~\ref{LL6}$(4)$, the projection $\, i:\, B
\to A \,$, see \eqref{E17}, is a flat (and hence, pseudo-flat)
ring epimorphism. Since $B$ is smooth, by Theorem~\ref{TCB1},
$\,i \,$ extends to an algebra map $\, \btheta:
\Pi^{\blambda}(B) \to \Pi^{i^*(\blambda)}(A)\,$, which is
also a pseudo-flat ring epimorphism. Now, since $ i $ is
surjective with $ \Ker(i) = \langle \ei\rangle $, the Cartesian
square \eqref{DD5} shows that $ \btheta $ is surjective and
$ \Ker(\btheta) = \langle \ei  \rangle $.
Finally, with identifications of Section~\ref{GCMS},
it is easy to see that $\, \btheta^*(\blambda) = 1 \,$.
\end{proof}
\bthm[\cite{CB}, Theorem~4.7]
\la{TCB2} If  $ A = \O(X) $ is the
coordinate ring of a smooth affine curve, then $ \Pi^1(A) $ is
isomorphic (as a filtered algebra) to $ \D = \D(X) $.
\ethm
We fix, once and for all, an isomorphism of Theorem~\ref{TCB2}
to identify $\, \D = \Pi^{1}(A)\,$. In combination with Lemma~\ref{recoll},
this yields an algebra map $\,\btheta:\,\Pi \to \D \,$. We will
use $ \btheta $ to relate the (derived) module categories of $ \Pi
$ and $ \D $, as follows (cf. \cite{BCE}).
First, we let $\, U \,$ denote the endomorphism ring of the
projective module $\, \ei \Pi \,$: this ring can be identified
with the associative subalgebra $\, \ei \Pi\, \ei \,$ of $\,\Pi
\,$ having $ \ei $ as an identity element. Next, we introduce six
additive functors $ (\btheta^*, \,\btheta_*,\, \btheta^{!}) $ and
$ (j_!, \,j^*,\, j_*) $ between the module categories of $ \Pi $,
$\,\D$ and $ U $. We let $\, \btheta_*:\,\Mod(\D)\to \Mod(\Pi)
\,$ be the restriction functor associated to
$\, \btheta \,$. This functor is fully faithful and
has both the right adjoint $\, \btheta^! := \Hom_{\Pi}(\D,\,
\mbox{---}\,)\,$ and the left adjoint $\, \btheta^* := \,\D
\otimes_{\Pi} \mbox{---}\,$, with adjunction maps $\, \btheta^*
\btheta_* \simeq \id \simeq \btheta^!\,\btheta_* \,$ being
isomorphisms. Now we define $\, j^*:\, \Mod(\Pi) \to \Mod(U) \,$
by $\, j^*\bV := \ei \bV \,$. Since $ e_\infty \in \Pi $ is an
idempotent, $ j^* $ is exact and has also the right and left
adjoints: $\, j_* := \Hom_{U}(e_\infty \Pi,
\,\mbox{---}\,)\,$ and  $\, j_! := \Pi \ei \otimes_{U}
\mbox{---}\,$, satisfying $\, j^*j_* \simeq \id
\simeq j^* j_! \,$.
The six functors $ (\btheta^*, \,\btheta_*,\, \btheta^{!}) $ and $
(j_!, \,j^*,\, j_*) $ defined above extend to 
the derived categories:
\begin{equation}
\la{Dia1}
\begin{diagram}[small, tight]
\DC^{b}(\Mod\, \D) \ & \ \pile{\lTo^{\btheta^*}\\
\rTo^{\btheta_{\!*}} \\
\lTo^{\btheta^!}} \ &
\ \DC^{b}(\Mod\, \Pi)\ & \ \pile{\lTo^{j_!}\\
\rTo^{j^*}\\ \lTo^{j_*}} \ & \ \DC^{b}(\Mod\, U) \\
\end{diagram}\ ,
\end{equation}
and their properties of these functors can be summarized in the following way (cf. \cite{BBD}, Sect.~1.4).
\bthm
\la{P1}
The diagram \eqref{Dia1} is a recollement of triangulated categories.
\ethm
Theorem~\ref{P1} follows from general results on recollement
of module categories (see \cite{Ko}) and the following
observation, which will be proved in
Section~\ref{strP} (see Lemma~\ref{Lrec1}): the multiplication map
$\,\Pi\,\ei \otimes_U \ei \Pi \to \Pi  \,$ fits into the exact
sequence
\begin{equation}
\la{scom}
0 \to \Pi\,\ei \otimes_U \ei \Pi \to \Pi \xrightarrow{\btheta} \D
\to 0\ ,
\end{equation}
which is a projective resolution of $ \D $ in the category of
(left and right) $\Pi$-modules. The existence of
\eqref{scom} implies that $\D$ has projective dimension $ 1 $ in $
\Mod(\Pi)$. Hence $\,\Tor^{\Pi}_n(\D, \D) = 0\,$ for all $\,n \ge
2 \,$. On the other hand, by Lemma~\ref{recoll}, $\,\btheta \,$ is
a pseudo-flat epimorphism, meaning that
$\,\Tor^{\Pi}_1(\D, \D) = 0 \,$ as well. Theorem~\ref{P1}
follows now from \cite{Ko}, Cor~14. As another consequence of \eqref{scom},
we have
\blemma 
\la{simp} If $ \bV $ is a finite-dimensional $\Pi$-module,
then $\, L_n \btheta^* (\bV) =  0 \,$ for $\, n \not= 1 \,$ and
\begin{equation}
\la{CMf}
L_{1} \btheta^* (\bV) \cong \Ker\left[\,\Pi\,
\ei \otimes_{U} \ei \bV \xrightarrow{\bmu} \bV \right],
\end{equation}
where $ L_n \btheta^* $ denotes the $n$-th derived functor of $ \btheta^* $
and $ \bmu $ is the natural multiplication-action map.
\elemma
\bproof Tensoring \eqref{scom} with $ \bV $ yields the exact
sequence
\begin{equation*}
\la{torr} 0 \to \Tor^{\Pi}_{1}(\D, \bV) \to \Pi \ei \otimes_U \ei
\bV \to \bV \to \D \otimes_{\Pi} \bV \to 0\ ,
\end{equation*}
and isomorphisms
$\, \Tor^\Pi_n(\D, \bV) \cong \Tor^\Pi_{n-1}(\Pi\,\ei
\otimes_U \ei \Pi, \bV)\,$ for $ n\ge 2 $.
Since $\, \Pi\,\ei \otimes_U \ei \Pi \,$ is projective (as a right
$\Pi$-module), the last $ \Tor$'s vanish. On the other hand, $\,
\dim_\c\, \bV < \infty \,$ implies that $\,\D \otimes_\Pi \bV = 0
\,$, since $ \D $ has no nonzero finite-dimensional modules. The
result follows now from the identification
$\, L_{n} \btheta^*(\bV) = \Tor_{n}^{\Pi}(\D,\,\bV)\,$, $\,n\ge
0\,$. \eproof
\remark\  Using \eqref{Dia1}, we may define the following `perverse'
$t$-structure on $ \DC^{b}(\Mod\, \Pi) $:
$$
{}^p\DC^{\le 0}(\Mod\,\Pi) := \{\,K^\bullet \in
\DC^{b}(\Mod\,\Pi)\ :\ j^*K^\bullet \in \DC^{\le 0}(U)\ ,\
\btheta^*K^{\bullet} \in \DC^{\le -1}(\D) \,\}\ ,
$$
$$
{}^p\DC^{\ge 0}(\Mod\,\Pi) := \{\,K^\bullet \in
\DC^{b}(\Mod\,\Pi)\ :\ j^*K^\bullet \in \DC^{\ge 0}(U)\ ,\
\btheta^! K^{\bullet} \in \DC^{\ge -1}(\D) \,\}\ ,
$$
where $\,\{\DC^{\le 0}(U),\, \DC^{\ge 0}(U)\}\,$ and $\,\{\DC^{\le
0}(\D),\, \DC^{\ge 0}(\D)\}\,$ denote the standard $t$-structures
on $ \DC^{b}(\Mod\,U) $ and $ \DC^{b}(\Mod\,\D) $ respectively.
Lemma~\ref{simp} shows that the $0$-complexes $\,[\,0 \to \bV \to
0\,]\,$ with $\,\dim_\c{\bV} < \infty \,$ lie in the heart of this
$t$-structure. So we may think of finite-dimensional $\Pi$-modules
as `perverse sheaves' with respect to the stratification \eqref{Dia1}.
The functor $ \btheta^* $ is then an algebraic analogue of
the restriction functor of a (perverse) sheaf to a closed subspace.

\subsection{The action of $ \Pic(\D) $ on Calogero-Moser spaces}
\la{ad}
We recall some facts about the Picard group $
\Pic(\D) $ of the algebra $ \D $ and its action on the space of
ideals $ \R(\D)$ (see \cite{BW}). It is known that $ \Pic(\D) $
has different descriptions for $ X = \A^1 $ and other curves (see \cite{CH1}).
Since the case of $ \A^1 $ is well studied, we will assume
that $ X \ne \A^1 $. Our
main theorem (Theorem~\ref{Tmain}) still holds for all curves $X$,
including $ \A^1 $.

In general, $\, \Pic(\D) $ can be identified with the group of $
\c$-linear auto-equivalences of the category $ \Mod(\D)$, and thus
it acts naturally on $ \R(\D) $ and $ K_0(\D) $. To be precise,
the elements of $\, \Pic(\D) \,$ are the isomorphism classes $
[\P] $ of invertible $\D$-bimodules, and the action of $
\Pic(\D) $ on $ \R(\D) $ and $ K_0(\D) $ is defined by $\, [M]
\mapsto [\P \otimes_{\D} M]\,$. The action of $ \Pic(\D)
$ on $ K_0(\D) $ preserves rank and hence restricts to $ \Pic(X) $
through the identification $\,K_0(\D) \cong K_0(X) \cong \Z \oplus
\Pic(X) \,$, see Section~\ref{SCI}.
\bprop[see \cite{BW}, Theorem~1.1] \la{piceq} $\, \Pic(\D) $ acts
on $ \Pic(X) $ transitively, and the map $ \gamma:\,\R(\D) \to
\Pic(X) $ defined by \eqref{E10} is equivariant under this action.
\eprop

Explicitly,  the action of $ \Pic(\D) $ on $ \Pic(X) $ can be
described as follows (cf. \cite{BW}, Prop.~3.1). By \cite{CH1},
Cor.~1.13, every invertible bimodule over $ \D $ is isomorphic to
$\, \D\L = \D \otimes_A \L \,$ as a left module, while the right
action of $ \D $  is determined by an algebra
isomorphism $\, \vp:\, \D \stackrel{\sim}{\to} \End_{\D}(\D\L)
\,$, where $\, \L \,$ is a line bundle on $X$. Following
\cite{BW}, we denote such a bimodule by $ \,  (\D\L)_{\vp} \,$.
Restricting $ \vp $ to $A$ yields an automorphism of $X$,
and the assignment
\begin{equation}
\la{Egr}
g:\,\Pic(\D) \to \Pic(X) \rtimes \Aut(X)\ ,\quad
[\, (\D\L)_{\vp}] \mapsto ([\L],\, \vp |_A)\ ,
\end{equation}
defines then a group homomorphism. On the other hand, $\,
\Pic(X) \rtimes \Aut(X)\,$ acts on $ \Pic(X) $ in the obvious way:
\begin{equation}
\la{E13}
([\L],\tau) : \, [\,\I\,] \mapsto [\,\L \,\tau(\I)\,] \ ,
\end{equation}
where $\, ([\L],\tau) \in  \Pic(X) \rtimes \Aut(X)\,$ and $\, [\I]
\in \Pic(X) \,$. Combining \eqref{Egr} and \eqref{E13},
we get an action of $ \Pic(\D) $ on $ \Pic(X) $, which agrees with
the natural action of $ \Pic(\D)$ on $ K_0(\D)$.

Now, given a line bundle $ \I $ and an invertible bimodule $\P =
(\D\L)_{\vp}\, $, we define $\,
P := \tilde{\L} \otimes_{\tA} B_{\tau} \,$, where
$\, \tilde{\L} := \L \times \c \,$, $\, \tau := \vp|_A \,$,
and $\, B_{\tau} := A[\tau(\I)] \,$. By Lemma~\ref{L10}$(a)$, $\,
P $ is a progenerator in the category of right $ B_\tau$-modules,
with $\, \End_{B_\tau}(P) \cong A[\J] \,$,
where $\, \J := \L\,\tau(\I) \,$. Associated to $ \P $ is thus the
Morita equivalence: $\, \Mod(B_\tau) \stackrel{\sim}{\to} \Mod(A[\J])\,$,
$\, \bV \mapsto P \otimes_{B_\tau} \bV \cong \tilde{\L} \otimes_{\tA} \bV\,$.

Next, we extend $\, P \,$ to the $\Pi^\blambda(B_\tau)$-module $\,
\bP := P \otimes_{B_\tau} \Pi^\blambda(B_\tau) \cong
\tilde{\L} \otimes_{\tA} \Pi^\blambda(B_\tau)$,
which is clearly a progenerator in the category of right $
\Pi^\blambda(B_\tau)$-modules. By Lemma~\ref{L10}$(b)$, $\,\tau\,$
defines an isomorphism $ B_\tau \stackrel{\sim}{\to} B $. Since $\,\tau(\blambda) = \blambda \,$
for all $\,\blambda \in S \,$, this isomorphism canonically extends to deformed
preprojective algebras $\,\ttau:\,\Pi^\blambda(B)
\stackrel{\sim}{\to} \Pi^\blambda(B_\tau) \,$, which allows us to
regard $ \bP $ as a $ \Pi^\blambda(B)$-module and identify
\begin{equation}
\la{endp}
\End_{\Pi^\blambda(B)}(\bP) \cong
\tilde{\F} \otimes_{\tA} \Pi^\blambda(B)
\otimes_{\tA}\tilde{\F}^\vee\  ,\quad \F := \L^\tau\ .
\end{equation}
With this identification, we have the embedding
\begin{equation}
\la{lso}
\ttau^{-1}:\ A[\J] \into \End_{\Pi^\blambda(B)}(\bP)\ ,
\end{equation}
and, since $\,\End_{\D}(\F \D) \cong \tilde{\F} \otimes_{\tA} \D \otimes_{\tA} \tilde{\F}^\vee
\,$, the natural map
\begin{equation}
\la{lso1}
1 \otimes \btheta \otimes 1 :\ \End_{\Pi^\blambda(B)}(\bP) \to
\End_{\D}(\F \D)\ .
\end{equation}
On the other hand, $\, \vp(\D) = \End_{\D}(\D\L) = \L^\vee \D  \L \, $ implies
$\,\D = \L\,\vp(\D)\,\L^\vee\,$, so taking
the inverse defines an isomorphism $\, \psi := \vp^{-1}:\ \D \to \F \,\D \,\F^\vee =
\End_{\D}(\F \D)\,$. Combining this last isomorphism with \eqref{lso1}, we get the
diagram of algebra maps
\begin{equation}
\la{pil}
\begin{diagram}
 \Pi^\blambda(A[\J])   &  \rDotsto{\bvp \ } & \End_{\Pi^\blambda(B)}(\bP) \\
\dTo^{\btheta}          &       & \dTo_{1 \otimes \btheta \otimes 1} \\
\D       & \rTo^{\psi\quad}       & \End_{\D}(\F \D) \\
\end{diagram}
\end{equation}
which obviously commutes when the dotted arrow is restricted to \eqref{lso}.

\bprop \la{Lex} There is a unique algebra isomorphism $\,\bvp\,$ extending
\eqref{lso} and making \eqref{pil} a commutative diagram. \eprop
We postpone the proof of Proposition~\ref{Lex} until
Section~\ref{st4}. Meanwhile, we note that the isomorphism $\,
\bvp\,$ makes $ \bP $ a left $ \Pi^\blambda(A[\J])$-module and
thus a progenerator from $ \Pi^\blambda(A[\I])$ to $
\Pi^\blambda(A[\J]) $. This assigns to $ \P = (\D\L)_{\vp} $ the
Morita equivalence:
$$
\Mod\,\Pi^\blambda(A[\I]) \to \Mod\,\Pi^\blambda(A[\J])\ ,\quad
\bV \mapsto \bP \otimes_{\Pi} \bV\ ,
$$
which, in turn, induces an isomorphism of representation varieties
\begin{equation}
\la{eigg}
f_{\P}:\ \CC_n(X, \I) \stackrel{\sim}{\to} \CC_n(X, \J)\ .
\end{equation}
\begin{remark} We warn the reader that
\eqref{eigg} {\it depends} on the choice of a specific
representative in the class $\, [\P] \in \Pic(\D) \,$, so, in
general, we do not get an action of $ \Pic(\D) $ on
$\,\bigsqcup_{[\I] \in \Pic(X)}\,\CC_n(X,\I)\,$. However, we will
see below (Proposition~\ref{omeac}) that $ f_{\P} $ induces a
well-defined action of $ \Pic(\D) $ on the reduced spaces $ \overline{\CC}_n(X,\I) $.
\end{remark}

Next, we describe a natural action of the canonical bundle $ \Omega^1 X $
on $ \CC_{n}(X, \I) $. Recall that the group homomorphism \eqref{Egr} is surjective
and fits into the exact sequence (see
\cite{CH1}, Theorem~1.15)
\begin{equation}
\la{exgr}
1 \to \Lambda \xrightarrow{\dlog} \Omega^1 X
\xrightarrow{c} \Pic(\D)\xrightarrow{g}
\Pic(X) \rtimes \Aut(X) \to 1\ ,
\end{equation}
where $\,\Lambda := A^\times\!/\c^\times\,$ is the multiplicative
group of (nontrivial) units in $ A $. The maps $\dlog $ and $ c $
in \eqref{exgr} are defined by
\begin{equation}
\la{unit}
{\tt dlog}:\ \Lambda \to \Omega^1 X\ , \ u \mapsto u^{-1} du\quad ,
\quad
c:\, \Omega^1 X \to \Pic(\D)\ ,\
\omega \mapsto [\,\D_{\hsigma_\omega}]\ ,
\end{equation}
where $\,\hsigma_\omega \,$ is the
automorphism of $\D$ acting identically on $ A $ and mapping $\,\partial
\in \Der(A) \,$ to $\, \omega(\partial) +
\partial \in \D_1 \,$. Since the action of $ \Pic(\D)$ on $
\Pic(X) $ factors through $ g $, the image of $ \Omega^1 X $ in $
\Pic(\D) $ under $ c $ stabilizes each point of $\, \Pic(X) \,$,
and therefore, by equivariance of $ \gamma $, preserves every
fibre $\, \gamma^{-1}[\I] \subseteq \R(\D) \,$. Thus, writing $
\Gamma := \Omega^1(X)/\Lambda \,$ and identifying $ \Gamma$ with $
\im(c) $, we get an action
\begin{equation}
\la{actfib}
\Gamma \times \gamma^{-1}[\I] \to \gamma^{-1}[\I]\ ,
\quad [\I] \in \Pic(X)\ .
\end{equation}

Now, let $\, \DB := \Omega^1(B)/[B, \,\Omega^1 B]\,$, where $\, B := A[\I] $.
Using the fact that $ B $ is smooth, we identify
\begin{equation*}
\la{Kar} \DB \cong B \otimes_{\eB} \Omega^1(B) \cong B
\otimes_{\eB} (\Omega^1 B)^{\star\star} \cong \Hom_{\eB}((\Omega^1
B)^{\star},\,B)\ ,
\end{equation*}
where  $\,
(\mbox{---} \,)^\star \,$ stands for the dual over $ \eB $.
Explicitly, under this identification, $\, \bar\omega
 = \omega\,(\mbox{\rm mod}\,[B, \,\Omega^1 B])\in \DB \,$ corresponds
to the homomorphism
\begin{equation}
\la{muop}
\hat{\omega}:\,\Omega^1(B)^{\star} \to B\ , \quad \delta \mapsto
\mu^{\mbox{\tiny o}}\,[\delta(\omega)]\ ,
\end{equation}
where $\,\mu^{\mbox{\tiny o}}:\, \eB \to B \,$ is the opposite multiplication map.
The additive group $\, \DB \,$ acts naturally on  $\, T_B\, (\Omega^1 B)^{\star} $:
for $\, \bar{\omega} \in \DB \,$, we have an automorphism $\,
\tsigma_\omega \,$ of $\, T_B\, (\Omega^1 B)^{\star}\,$ acting identically on $ B $
and mapping
$$
(\Omega^1 B)^{\star} \to B
\oplus (\Omega^1 B)^{\star} \into T_B\, (\Omega^1 B)^{\star}\ ,\quad
\delta \mapsto \hat{\omega}(\delta) + \delta \ .
$$
The assignment $\, \bar{\omega} \mapsto \tsigma_\omega \,$ defines then a group
homomorphism
\begin{equation}
\la{auto1} \tsigma: \DB \to \Aut_{B}[T_B\, (\Omega^1 B)^{\star}]\ .
\end{equation}

Identifying $ \Omega^1 X $ with the group of K\"ahler
differentials of $ A $, we now construct an embedding $ \Omega^1 X
\into \DB $. For this, we consider the exact sequence
\begin{equation}
\la{auto2}
0 \to \HH_1(B) \stackrel{\alpha}{\to} \DB \to B \to
\HH_0(B) \to 0\ ,
\end{equation}
obtained by tensoring $\, 0 \to \Omega^1(B) \to \eB \to B \to 0
\,$ with $ B $, and compose the connecting map $ \alpha $ in
\eqref{auto2} with natural isomorphisms (see \cite{Lo},
Th.~1.2.15 and Prop.~1.1.10, respectively)
\begin{equation}
\la{auto3} \HH_1(B) \cong \HH_1(A) \cong \Omega^1 X \ .
\end{equation}
Now, for any algebra $B$, we have (see \cite{CE}, Ex.~19, p.~126)
$$
\HH_1(B) = \Tor^{\eB}_1(B,\,B) \cong
\Omega^1(B)\,\cap\,\Omega^1(B)^{\circ}/\,\Omega^1(B) \cdot \Omega^1(B)^{\circ}\ ,
$$
where $\,\Omega^1(B)^{\circ} := \Ker (\mu^{\mbox{\tiny o}})\,$. Hence,
if $\, \bar{\omega} \in \im\,\alpha\,$ in \eqref{auto2}, then
$\, \Delta_B(\omega) \in
\Omega^1(B)\,\cap\,\Omega^1(B)^{\circ} \subseteq
\Omega^1(B)^{\circ} \,$, so by \eqref{muop}, $\,\hat{\omega}(\Delta_B) = 0\,$
and $\, \tsigma_{\omega}(\Delta_B) = \Delta_B\,$.
Thus, combining \eqref{auto1} with \eqref{auto2} and
\eqref{auto3}, we may define
\begin{equation}
\la{auto4}
\sigma:\ \Omega^1 X \stackrel{\alpha}{\into} \DB
\stackrel{\tsigma}{\to} \Aut_{B}[T_B\, (\Omega^1
B)^{\star}] \to \Aut_{S}[\Pi^\blambda(B)]\ ,
\end{equation}
where the last map is induced by the algebra projection: $\,T_B\,
(\Omega^1 B)^{\star} \onto \Pi^\blambda(B)\,$. An explicit
description of $ \tsigma $ will be given in Section~\ref{st3} (see
Lemma~\ref{actgen}).

Now, the group $\, \Aut_{S}[\Pi^\blambda(B)] \,$ acts on
$\,\Rep_S(\Pi^\blambda(B), \bn)\,$ in the natural way: if
$\,\varrho: \Pi^\blambda(B) \to \End(\bV)\,$ represents a point in
$\,\Rep_S(\Pi^\blambda(B), \bn)\,$, then $\,\sigma.\varrho =
\varrho \, \sigma^{-1} $ for $ \sigma \in \Aut_{S}
[\Pi^\blambda(B)]\, $. Clearly, this commutes with the $
\G_S(\bn)$-action on $\,\Rep_S(\Pi^\blambda(B), \bn)\,$, and hence
induces an action of $\, \Aut_{S}[\Pi^\blambda(B)] \,$ on $\,
\CC_{n} (X,\I) \,$. Restricting this last action to $ \Omega^1 X $
via \eqref{auto4}, we define
\begin{equation}
\la{act1} \sigma^*:\ \Omega^1 X \to \Aut\, [\,\CC_{n}(X, \I)\,]\ ,
\quad \omega \mapsto [\,\sigma^*_\omega:\,\varrho \mapsto
\varrho \, \sigma_\omega^{-1}\,] \ .
\end{equation}
Equivalently, $\,\sigma^*_\omega \,$ is defined on $ \CC_n(X,\I) $
by twisting the structure of $ \Pi^\blambda(B)$-modules by $
\sigma_\omega^{-1} $, i. e. $\,[\bV] \mapsto
[\bV^{\sigma_\omega^{-1}}]\,$. Restricting \eqref{act1} further to
$\, \Lambda \,$, via \eqref{unit}, we define the quotient varieties
\begin{equation}
\la{redcm}
\overline{\CC}_{n}(X,\I) := \CC_{n}(X,\I)/\Lambda\ .
\end{equation}
These varieties come equipped with the induced action of the group
$\, \Gamma = \Omega^1(X)/\Lambda \,$.

\bprop 
\la{omeac} $(1)$\ The action \eqref{act1} agrees with
\eqref{eigg}{\rm\,:\,} if $\,\P = \D_{\hsigma_\omega}\,$, then
$\, f_{\P} = \sigma^*_\omega \,$ for all $\,\omega \in \Omega^1 X\,$.

$(2)$ The map \eqref{eigg} induces an isomorphism of quotient
varieties
$\,
\bar{f}_{\P}:\ \overline{\CC}_{n}(X,\I) \stackrel{\sim}{\to}
\overline{\CC}_{n}(X,\J)\,$,
which depends only on the class of $ \,\P $ in $ \Pic(\D) $.
\eprop
We will prove Proposition~\ref{omeac} in Section~\ref{st4}. Here, we make only two
remarks.

1. It follows from Proposition~\ref{omeac} that the action of
$ \Lambda $ on $ \CC_{n}(X,\I) $ defined above coincides with the
natural action of $\, \Aut(\I) = A^\times \,$, so $
\overline{\CC}_{n}(X,\I) $ depends only on the class of $ \I $ in
$ \Pic(X) $ and the definition \eqref{redcm} agrees with the one
given in the introduction.

2. For each $\,n \ge 0\,$, let $\,\overline{\CC}_n(X)\,$ denote the
disjoint union of $\,\overline{\CC}_{n}(X,\I)\,$ over all $ [\I]\in\Pic(X)$.
By part $(2)$ of Proposition~\ref{omeac}, the assignment $\,[\P] \mapsto
\bar{f}_{\P} \,$ defines then an action of $ \Pic(\D) $ on $
\overline{\CC}_{n}(X)\,$, and part $(1)$ says that this action
restricts to the action of $ \Gamma $ on each individual fibre $
\overline{\CC}_{n}(X,\I) \,$, i.~e. $\, \bar{f}_{c(\omega)} =
\hsigma^*_{\omega}\,$ for all $\,\omega \in \Gamma \,$.

\subsection{The main theorem}
\la{TMT} We may now put pieces together and state the main result
of the present paper. We recall the functor $\, L_1\btheta^* =
\Tor^\Pi_1(\D,\,\mbox{---}):\, \Mod(\Pi) \to \Mod(\D) \,$
associated to $\,\btheta:\,\Pi \to \D
\,$: when restricted to finite-dimensional representations, this
functor is given by \eqref{CMf}.
\bthm
\la{Tmain} Let $X$ be a smooth affine irreducible curve over
$\c$.

$(a)$ For each $ n \geq 0 $ and $ [\I] \in \Pic(X) $, the functor
\eqref{CMf} induces an  injective map
$$
\omega_n:\,\overline{\CC}_n(X, \I) \to \gamma^{-1}[\I]\ ,
$$
which is equivariant under the action of the group $\,\Gamma $.

$(b)$ Amalgamating  the maps $\, \omega_n \,$ for all $\, n \ge
0\, $ gives a bijective correspondence
$$
\omega:\ \bigsqcup_{n \geq 0} \overline{\CC}_n(X,\I)
\stackrel{\sim}{\to} \gamma^{-1}[\I] \ .
$$

$(c)$ For any $ [\I] $ and $ [\J] $ in $ \Pic(X) $ and for any
$[\P] \in \Pic(\D) $, such that $ [\P] \cdot [\I] = [\J] $, there
is a commutative diagram:
\begin{equation}
\la{CMdim}
\begin{diagram}[small, tight]
 \overline{\CC}_n(X,\I)    &  \rTo^{\ \bar{f}_{\P}\ } & \overline{\CC}_n(X,\J) \\
\dTo^{\omega_n}            &                   & \dTo_{\omega_n} \\
\gamma^{-1}[\I]            & \rTo^{\ [\P]\ }       & \gamma^{-1}[\J] \\
\end{diagram}
\end{equation}
where $ \bar{f}_{\P} $ is an isomorphism induced by \eqref{eigg}.

\ethm
\begin{remark}
For technical reasons, we assumed above that $\, X \ne \A^1 $.
Theorem~\ref{Tmain} holds true, however, in general: if $ X = \A^1
$, the map $ \omega $ induced by $ \btheta^* $ agrees with the
Calogero-Moser map constructed in \cite{BW1, BW2} (see \cite{BCE},
Theorem~1). In this case, the ring $ \D $ is isomorphic to the
Weyl algebra $ A_1(\c) $, $\, \Pic(\D) $ is isomorphic to the
automorphism group $ \Aut(A_1) $ of $ A_1 $ (see \cite{St}) and $
\Gamma $ corresponds to the subgroup of KP flows in $ \Aut(A_1)$
(see \cite{BW1}). Since $ \Pic(\A^1) $ is trivial, the last part
of Theorem~\ref{Tmain} implies the equivariance of $ \omega $
under the action of $ \Aut(A_1) $.
\end{remark}
\section{Proof of the Main Theorem}
\la{proof}
We proceed in four steps. First, we show that the
functor \eqref{CMf} induces well-defined maps
$\,\tilde{\omega}_n:\,\CC_n(X,\I) \to\gamma^{-1}[\I] \,$, one for
each integer $\, n \ge 0 \,$. Second, we prove that every class
$\, [M] \in \gamma^{-1}[\I]\,$ is contained in the image of $\,
\tilde{\omega}_n \,$ for some $\, n \,$ (which is uniquely
determined by $ [M] $). Third, we check that  $\, \tilde{\omega}_n
\,$ factors through the action of $ \Lambda $ on $\,\CC_n(X,\I) $
and prove that the induced map $\,
\omega_n:\,\overline{\CC}_n(X,\I) \to \gamma^{-1}[\I] \,$ is
injective and $ \Gamma$-equivariant. Finally, we prove
Propositions~\ref{Lex} and~\ref{omeac} of Section~\ref{ad},
and show that the diagram \eqref{CMdim} in Theorem~\ref{Tmain} is
commutative.

We begin by describing the algebras $ \Pi^\blambda(B) $ in terms
of generators and relations.

\subsection{The structure of $ \Pi^\blambda(B) $}
\la{strP} Recall that, for each $\,\blambda \in S \,$, we defined
these algebras by formula \eqref{dpa1},
where $\,\Delta_B \in \DDer(B) \,$ is the distinguished derivation  
mapping $\, x \mapsto x \otimes 1 - 1 \otimes
x \,$. Now, $\,\DDer(B)\,$ contains a canonical sub-bimodule
$\, \DDer_S(B)\,$, consisting of $S$-linear derivations. We
write $\, \Delta_{B,S}: \,B \to B \otimes B \,$ for the inner
derivation $\, x \mapsto \ad_{\be}(x) \,$, with $\, \be := e
\otimes e + \ei \otimes \ei \in B \otimes B\,$. It is easy to see
that $\, \Delta_{B,S}(x) = 0 \,$ for all $ x \in S $, so $\,
\Delta_{B,S} \in \DDer_S(B)\,$.
\begin{lemma}
\la{DPApres} For any $\, \blambda \in S $, there is a canonical
algebra isomorphism
$$
\Pi^\blambda(B) \cong  T_{B}\DDer_S(B)/ \langle \Delta_{B,S} -
\blambda \rangle\ .
$$
\end{lemma}
\begin{proof}
By universal property, the natural embedding of bimodules
$ \DDer_S(B) \into \DDer(B) $ extends to their tensor
algebras. Combined with canonical projection, this yields the
algebra map
$\, \phi:\,T_B \DDer_S(B) \into T_B \DDer(B) \onto
\Pi^\blambda(B)\,$.
An easy calculation shows that $\, \Delta_{B,S} = e\, \Delta_B \,e
+ \ei\, \Delta_B \,\ei \,$ in $\, \DDer(B)\,$. So $\,
\Delta_{B,S} - \blambda = e\,(\Delta_B - \blambda)\,e + \ei
(\Delta_B -\blambda)\,\ei \,$ belongs to the ideal $\,\langle
\Delta_{B} - \blambda \rangle \subseteq T_B \DDer(B) \,$, and
hence $\, \phi\,$ vanishes on $\,\Delta_{B,S} - \blambda\,$,
inducing an algebra map
\begin{equation}
\la{opsi}  T_{B}\DDer_S(B)/ \langle
\Delta_{B,S} - \blambda \rangle \to \Pi^\blambda(B)\ .
\end{equation}
We leave as an exercise to the reader to check that
\eqref{opsi} is an isomorphism.
\end{proof}

By Lemma~\ref{DPApres}, the structure of  $\,
\Pi^\blambda(B)\,$ is determined by the bimodule $\,\DDer_S(B)\,$. 
We now describe this bimodule explicitly, in terms of $\,
A \,$, $\,\I\,$ and the dual module $\, \I^\vee = \Hom_A(\I,A)\,$.
To fix notation we begin with a few fairly obvious remarks on
bimodules over one-point extensions.

A bimodule $ \Xi $ over $ B $ is characterized by the following
data: an $A$-bimodule $ T $, a left $A$-module $ U $, a right
$A$-module $ V $ and a $\c$-vector space $ W $ given together with
three $A$-module homomorphisms $\,f_1:\, \I \otimes V \to T \,$,
$\, f_2: \, \I \otimes W \to U \,$, $\, g_1:\, T \otimes_A \I \to
U \,$ and a $\c$-linear map $\, g_2:\, V \otimes_A \I \to W \,$,
which  fit into the commutative diagram
\begin{equation}
\la{dibim}
\begin{diagram}[small, tight]
\I \otimes  V \otimes_A \I     & \rTo^{\id \otimes g_2}     & \I \otimes W \\
\dTo^{f_1 \otimes_A \id}     &                            & \dTo_{f_2} \\
T \otimes_A \I                & \rTo^{g_1}                 & U \\
\end{diagram}
\end{equation}
These data can be conveniently organized by using the matrix
notation
\begin{equation}
\la{mat1}
\Xi = \left(
\begin{array}{cc}
 T & U \\
 V & W
\end{array}
\right) \ ,
\end{equation}
with understanding that $ B $ acts on $ \Lambda $  by the usual
matrix multiplication, via the maps $\,f_1, \,f_2, \,g_1 $ and $
g_2\,$. Note that the components of $ T $ are determined by
\begin{equation}
\la{mat2} T = e \,\Xi\, e\ , \quad U = e \,\Xi\, \ei\ ,
\quad V = \ei \Xi \,e\ , \quad W = \ei \Xi \,\ei\ ,
\end{equation}
and the commutativity of \eqref{dibim} ensures the associativity
of the action of $B$.  For example,
the free bimodule $\,B \otimes B\,$ can be decomposed into
a direct sum of four bimodules, each of which is easy to identify
using \eqref{mat2}:
\begin{equation}
\la{BoB1}
Be \otimes eB \cong \left(
\begin{array}{cc}
 A \otimes A  &  A \otimes \I\\
 0 & 0
\end{array}
\right) \ ,\quad
B e \otimes \ei B \cong \left(
\begin{array}{cc}
 0 &  A\\
 0 & 0
\end{array}
\right) \ ,
\end{equation}
\begin{equation}
\la{BoB2}
B\ei \otimes eB \cong \left(
\begin{array}{cc}
\I \otimes A  &  \I \otimes \I\\
A & \I
\end{array}
\right) \ ,\quad
B\ei \otimes \ei B \cong \left(
\begin{array}{cc}
 0  &  \I \\
 0 & \c
\end{array}
\right) \ .
\end{equation}
With this notation, the bimodule $\, \DDer_S(B)\,$ can be described as follows.
\begin{lemma}
\la{DerS}
There is an isomorphism of $B$-bimodules
$$
\DDer_S(B) \cong
\left(
\begin{array}{cc}
 \DDer(A)  & \Der(A, \I\otimes A)\\*[1ex]
 0 & 0
\end{array}
\right) \bigoplus \left(
\begin{array}{cc}
 \I \otimes \I^\vee  &  \I \otimes A\\*[1ex]
 \I^\vee & A
\end{array}
\right) \ ,
$$
with $\,\Delta_{B,S}\,$ corresponding to the element
\begin{equation}
\la{dS}
 \left[ \left(
\begin{array}{cc}
 \Delta_A  & 0\\*[1ex]
 0 & 0
\end{array}
\right) \, , \, \left(
\begin{array}{cc}
 -\sum_i \v_i \otimes \w_i  & 0\\*[1ex]
 0 & 1
\end{array}
\right) \right]\ ,
\end{equation}
where $\, \{\v_i\} \,$ and $\,\{\w_i\}\,$ are dual bases for the
projective $A$-modules $ \I $ and $ \I^\vee $.
\end{lemma}
\begin{proof}
With identificaions \eqref{BoB1} and \eqref{BoB2}, it is easy to show that
\begin{equation}
\la{omegaS}
\Omega_S^1 B \cong \left(
\begin{array}{cc}
 \Omega^1 A  &  \Omega^1 A \otimes_A \I\\
 0 & 0
\end{array}
\right)
\bigoplus
 \left(
\begin{array}{cc}
 0  & \I\\
 0 & 0
\end{array}
\right)\ ,
\end{equation}
with inclusion $\,\Omega_S^1 B \into B \otimes_S B =
(Be \otimes eB)\,\oplus\, (B\ei \otimes \ei B)\,$
corresponding to the map $\,(i,\,s)\,$, where
$ i $ is the natural embedding of the first summand of
$\, \Omega_S^1 B \,$ into $ Be \otimes eB $, see \eqref{BoB1},
and $\,s\,$ is a $B$-bimodule section
$\,\tilde{\I} \to B \otimes_S B \,$ given by
$$
s:\
\left(
\begin{array}{cc}
 0 & b\\
 0 & 0
\end{array}
\right) \mapsto \left[
\left(
\begin{array}{cc}
 0  & - \sum_i \w_i(b) \otimes \v_i\\
 0 & 0
\end{array}
\right) \, , \, \left(
\begin{array}{cc}
 0  &  b\\
 0 & 0
\end{array}
\right) \right] \ ,\quad b \in \I\ .
$$
Note that $i$ is canonical, while $s$ depends on the choice of dual bases for $ \I $ and $ \I^\vee $.

To describe $\,\DDer_S(B)\,$ we now dualize
\eqref{omegaS} and use 
$\, \DDer_S(B) = \Hom_{\eB}(\Omega_S^1 B, \BB)\,$,
which after trivial identifications yields
\begin{equation}
\la{decc} \DDer_S(B) \cong \left(
\begin{array}{cc}
 \DDer(A)  &  \Der(A, \I \otimes A)\\*[1ex]
 0 & 0
\end{array}
\right) \bigoplus \left(
\begin{array}{cc}
 \I \otimes \I^\vee  &  \I \otimes \End_A(\I)\\*[1ex]
 \I^\vee & \End_A(\I)
\end{array}
\right) \, .
\end{equation}
Since $A$ is commutative and $\I$ is a rank $1$ projective, $\,\End_A(\I) = A \,$, so \eqref{decc} is the required
decomposition. With identification \eqref{omegaS}, 
the element \eqref{dS} corresponds to the embedding
$\,(i, s):\,\Omega_S^1 B  \to (B e \otimes e B)\oplus(B \ei \otimes \ei B) \into B \otimes
B \,$,
which, in turn, corresponds under \eqref{decc} to the element
$\, \Delta_{B,S}\in \DDer_S(B)\,$.
\end{proof}

Now, using the isomorphism of Lemma~\ref{DPApres}, we identify
$ \Pi^\blambda(B) $ as a quotient of the tensor algebra of 
the bimodule $ \DDer_S(B) $. Keeping the notation of Lemma~\ref{DerS}, we then have
\begin{proposition}
\la{gene}
The algebra $\,\Pi^\blambda(B)\,$ is generated by (the
images of) the following elements
$$
\ha := \left(
\begin{array}{cc}
 a & 0\\*[1ex]
 0 & 0
\end{array}
\right)\ ,\ \hv_i := \left(
\begin{array}{cc}
 0 & \v_i\\*[1ex]
 0 & 0
\end{array}
\right) \ , \ \hd :=\left(
\begin{array}{cc}
 d & 0\\*[1ex]
 0 & 0
\end{array}
\right)\ ,\ \hw_i := \left(
\begin{array}{cc}
 0 & 0\\*[1ex]
 \w_i & 0
\end{array}
\right) \ ,
$$
where $\, \ha \, ,\,\hv_i \in B \,$ and $\, \hd
\,,\, \hw_i \in \DDer_S(B)\,$ with $\,d \in \DDer(A)\,$.
Apart from the obvious relations induced by matrix multiplication,
these elements satisfy
\begin{equation}
\la{r1}
\hdel_A - \sum_{i=1}^N \hv_i \cdot \hw_i = \lambda\, e\
,\quad \sum_{i=1}^N \hw_i \cdot \hv_i = \lambdai\, \ei\ ,
\end{equation}
where `$\,\cdot $' denotes the action of  $B$ on the bimodule $
\DDer_S(B)$.
\end{proposition}
\begin{proof}
By Lemma~\ref{LL6}, the matrices $\,\{\ha\}\,$ and $\,\{\hv_i\}\,$
generate the algebra $ B $, while $\,\{\hd\}\,$ and
$\,\{\hw_i\}\,$ generate the first and the second bimodule summand
of \eqref{decc} respectively. All together they thus generate the
tensor algebra. Now, the ideal $\,\langle
\Delta_{B,S} - \blambda \rangle\,$ in $ \Pi^\blambda(B) $ is
generated by  $\,e\,(\Delta_{B,S} - \blambda )\,e =
e\,\Delta_{B,S}\,e - \lambda \, e\,$ and $\,\ei\,(\Delta_{B,S} -
\blambda )\,\ei = \ei\,\Delta_{B,S}\,\ei - \lambdai \, \ei\,$,
since the sum of these elements is equal to $\, \Delta_{B,S} -
\blambda\,$. With identification of Lemma~\ref{DerS}, we then have
$$
e\,\Delta_{B,S}\,e = \hdel_A - \sum_{i=1}^N \hv_i \cdot \hw_i \quad , \quad
\ei\,\Delta_{B,S}\,\ei = \sum_{i=1}^N \hw_i \cdot \hv_i\ ,
$$
whence the relations \eqref{r1}.
\end{proof}

Using Proposition~\ref{gene}, we now prove two technical results,
which we use repeatedly in this paper (Lemma~\ref{Lrec1} below 
was already mentioned in Section~\ref{Re}).
\begin{lemma}
\la{P7}
If $ \lambdai \not= 0 $, the algebra $\,
\Pi^{\blambda}(B)\,$ is Morita equivalent to $\, e\,
\Pi^{\blambda}(B)\,e \,$.
\end{lemma}
\begin{proof}
By standard Morita theory, it suffices to show that $\, \Pi^{\blambda}(B) \,e\,
\Pi^{\blambda}(B) = \Pi^{\blambda}(B) \,$. This last identity
holds  in $ \Pi^{\blambda}(B)$ if  $\, 1 \in
\Pi^{\blambda}(B) \,e\, \Pi^{\blambda}(B)\,$, or equivalently, if
$\,\ei \in \Pi^{\blambda}(B) \,e\, \Pi^{\blambda}(B)\,$, since $ e
+ \ei = 1 $. But if $\,\lambdai \not=0\,$, the second relation of
\eqref{r1} can be written as
\begin{equation}
\la{nei}
\ei = \frac{1}{\lambdai}\,\sum_{i=1}^N \hw_i \cdot \hv_i =
\frac{1}{\lambdai}\,\sum_{i=1}^N \hw_i \cdot e \cdot \hv_i \ ,
\end{equation}
whence the result.
\end{proof}
\blemma
\la{Lrec1}
The multiplication map $\,\Pi\,\ei \otimes_U
\ei \Pi \to \Pi  \,$ gives a projective resolution of $ \D $ in
the category of (left and right) $\Pi$-modules, see \eqref{scom}.
\elemma
\bproof We identify $\, A \cong e B e
\subset e \Pi e \,$ via $\, a \mapsto \ha \,$ and $\, \I \cong e B
\ei \subset e \Pi \ei $ via $\, \v \mapsto \hv \,$. Then 
tensoring $ \btheta $ with $ \I $ yields $\, e \Pi e \otimes_{e B
e} e B \ei \to \D \otimes_A \I \cong \D \I \,$. Since $ e B \ei $
is a projective $eBe$-module, the multiplication map $\, e \Pi e
\otimes_{e B e} e B \ei \to e \Pi \ei \,$ is an isomorphism onto
$\, e \Pi e B \ei \subseteq e \Pi \ei \,$. On the other
hand, $\, e \Pi \ei \subseteq e \Pi e B \ei \,$, by \eqref{nei}. 
Thus, identifying $ \,
e \Pi e \otimes_{e B e} e B \ei \cong e \Pi \ei \,$, we get a
surjective map of left $ A$-modules: $\,e \Pi \ei \onto \D \I \,$.
Since $ \D\I$ is projective over $A$, the last map 
has an $A$-linear section, which we denote by $\, s:\, \D \I \into e \Pi \ei \,$. 
Now, using this section, we consider 
\begin{equation}
\la{idenn}
\D\I \otimes \ei \Pi e \xrightarrow{s \otimes 1} e \Pi \ei
\otimes \ei \Pi e \onto e \,\Pi\, \ei \otimes_{U} \ei \Pi e\ ,
\end{equation}
which is a homomorphism of right $e \Pi e$-modules.
Since $\,e \Pi \ei =  e \Pi e B \ei = \sum_i e \Pi e\,
\hv_i \,$, the map \eqref{idenn} is surjective. On the other hand,
using filtrations, it is easy to show that the
composition of \eqref{idenn} with multiplication map $\,e \Pi \ei
\otimes_{U} \ei \Pi e \to e \Pi e\,$ is injective. Hence
\eqref{idenn} is injective and therefore an isomorphism. This
implies that $\, e \Pi \ei \otimes_{U} \ei \Pi e \,$ is a right
projective $e \Pi e$-module (since obviously so is $\, \D\I
\otimes \ei \Pi e \,$), and $\, 0 \to e \,\Pi\, \ei \otimes_{U}
\ei \Pi e \to e \Pi e \xrightarrow{\btheta} \D \to 0\,$ is an
exact sequence of $e \Pi e$-modules. By Morita equivalence of
Lemma~\ref{P7}, the complex $\,0 \to \Pi\,\ei \otimes_U \ei \Pi
\to \Pi \to 0 \,$ is then a projective resolution of $ \D$ in the
category of right $\Pi$-modules. A similar argument shows that
this complex is also a projective resolution of $ \D $ as a left
$\Pi$-module. 
\eproof

\subsection{The map $ \omega $ is well-defined}
\la{st1}
We show that the functor \eqref{CMf} maps the $
\Pi$-modules of dimension vector $ \bn = (n,1) $ to rank $1$
torsion-free $\D$-modules $ M $ with $ \gamma[M] = [\I] $.

Let $ \bV $ be a $ \Pi$-module of dimension vector $ \bn $, and
let $\, \bL := \Pi \ei \otimes_{U} \ei \bV $. Write $\, V := e \bV
\,$, $\, \Vi := \ei \bV \,$, and similarly $\, L := e \bL $, $\,
L_\infty := \ei \bL \,$ (so that $\,\dim\, \Vi = \dim\, L_\infty =
1 $). Fix a vector $ \xi \ne 0 $ in $ \Vi $ and define a character
$\,\varepsilon: U \to \c \,$ by $\, u.\xi = \eps(u) \xi \,$ for
all $ u \in U $ . Note that $ \eps $ does not depend on the choice
of $ \xi $ and uniquely determines $ \bL $ (and $ \bV$). In fact,
we have the isomorphism of $\Pi$-modules
\begin{equation}
\la{iso}
\Pi\,\ei \! \! \left/ \sum_{u \in U}\,\Pi\, \ei(u - \eps(u))
\stackrel{\sim}{\to} \bL
\right.\ ,\quad [\ei] \mapsto \ei \otimes \xi\ .
\end{equation}

Now, under the equivalence of Lemma~\ref{P7},
$\,\bmu: \bL \to \bV \,$ transforms to a homomorphism
of $ e \Pi e$-modules $\,\mu: L \to V \,$. Since 
$ \ei (\Ker\,\bmu) = 0 $, we have $\, \Ker\,\bmu = 
e\,(\Ker\,\bmu) = \Ker\,\mu $. Thus
$\, \btheta^* (\bV) \cong \Ker\,\mu\,$,
which is naturally an isomorphism of $\D$-modules via
$\,\btheta|_{e \Pi e}: \, e \Pi e \to \D \,$.

Next, we set $ R := T_A \DDer(A) $ and define the algebra map
\begin{equation}
\la{algm}
 R \to e \Pi e \ ,\quad a \mapsto \ha\ , \ d \mapsto \hd\ ,
\end{equation}
where $ a \in A $ and $ d \in \DDer(A)$. Extending the
notation of Proposition~\ref{gene}, we will write $\, \hr \in e
\Pi e  \,$ for the image of any element $\, r \in R \,$ under
\eqref{algm}. Note that the natural projection $\,R \onto
\Pi^1(A) = \D \,$ factors through \eqref{algm}, and the
corresponding quotient map is  $\,\btheta|_{e \Pi e}\,$.
The following observation is an easy consequence of \eqref{iso}
and Lemma~\ref{Lrec1}.
\blemma
\la{iso2}
There is an isomorphism of $R$-modules
\begin{equation}
\la{iso3}
L \cong R \,\I \! \! \left/\sum_{i=1}^N \sum_{r \in R}\,R \biggl[(\Delta_A -1)\,r \v_i -
\sum_{j=1}^N\,\eps(\hw_j \,\hr\,\hv_i)\,\v_j\biggr]
\right. \ ,
\end{equation}
where $ L $ is regarded as an $ R $-module via \eqref{algm}, and $
R\,\I := R \otimes_A \I $.
\elemma
\bproof If we identify $\, A \cong e B e
\subset e \Pi e \,$, $\, \I \cong e B \ei  \subset e \Pi \ei \,$
as in Lemma~\ref{Lrec1}, the required isomorphism is induced by
\begin{equation*}
\la{ppp}
R \, \I \stackrel{\pi_1}{\too}  e \Pi e \otimes_{e B e} e B \ei
\stackrel{\pi_2}{\too} e \Pi \ei \onto
e \Pi \ei\! \! \left/ \sum_{u \in U}\, e \Pi \ei  \,(u - \eps(u))
\right. \ ,
\end{equation*}
where $\,\pi_1 $ is the product of \eqref{algm} with $ \I
$ and $\,\pi_2\,$ is the multiplication map.
\eproof
Now, the tensor algebra filtration on $\,R = T_A \DDer(A)\,$
induces the differential filtration on $
\D $ via the canonical projection and module filtrations on $
L $ and $ M \subseteq L $ via the isomorphism of Lemma~\ref{iso2}.
Writing $\,\DO \,$, $\,\LO\,$, $\ldots $ for the associated graded
objects relative to these filtrations, we have
\begin{equation*}
\la{filtin}
 \MO \,\subseteq \,\LO\, \cong \, \RO\,\I\!\!\left/
\RO\,\bar{\Delta}_A \RO\,\I \right. \,\cong\,
(\RO\!\left/\RO\,\bar{\Delta}_A \,\RO) \otimes_A \I
 \right. \cong \DO \otimes_A \I \cong \DO\,\I\ .
\end{equation*}
It follows that $ M $ is a rank $1$ torsion-free module
(as so is $ \MO $). Moreover, since $\,
\dim\, \LO/\MO = \dim\, L/M < \infty \,$, by Theorem~\ref{T4}$(a)$,
$\, \gamma[M] = [\I]\,$. This completes Step~1.\\*[-1ex]

\subsection{The map $ \omega $ is surjective}
\la{st2}
Given a rank $ 1 $ torsion-free $\D$-module $ M $, we now
construct a $ \Pi$-module $ \bL $, together with a $\,
\Pi^\blambda(B)$-module embedding $\,M \into \bL \,$, such that
 $ \bV := \bL/M $ has dimension $ (n,1) $ and $\, \btheta^*[\bV] \cong M \,$.

We begin with some preparations. We let $\, \DD :=
\bigoplus_{k=0}^{\infty} \,\D_k t^k \,$ denote the Rees algebra
of the ring $ \D $ with respect to its
canonical filtration $\{\D_k\}$, and let $\,\GrM(\DD)\,$ be the
category of graded $\DD$-modules. There is a natural homomorphism
of graded rings $\, p:\,\DD \to \DO \,$, mapping $\,a \,t^k \in
\DD_k \,$ to $\,a\,(\mbox{mod}\, \D_{k-1}) \in \DO_k\,$. Using
this homomorphism, we will regard graded $\DO$-modules as objects
of $\,\GrM(\DD)\,$. Since $\,\Ker(p) = \langle\, t \,\rangle\,$,
we may identify $\,\DO \cong \DD/\langle\, t \,\rangle \,$. This
implies that $ \DD $ is Noetherian, since so is $\,\DO\,$ (see
\cite{L}, Prop.~3.5).

Next, following \cite{AZ}, we define $\,\Fdim(\DD)\,$ to be the
full subcategory of $ \GrM(\DD) $ consisting of torsion modules.
By definition, $\, \MM \in \GrM(\DD) \,$ is {\it torsion}, if for
every $\, m \in \MM \,$ there is $ k_m \in \N $ such that  $
\DD_k\,m = 0 $ for all $ k \ge k_m $. By  \cite{AZ}, Sect.~2,
$\,\Fdim(\DD)\,$ is a localizing subcategory of $ \GrM(\DD)
$: i.~e., the inclusion functor $\,\Tors(\DD) \into \GrM(\DD) \,$ has
a right adjoint $\,\tau:\, \GrM(\DD) \to
\Tors(\DD)\,$, which assigns to a graded module $ \MM $ its
largest torsion submodule
$\, \tau(\MM) = \{m \in \MM\ :\ \DD_{k}\,m = 0\ \mbox{for all}\ k
\gg 0\} $.
The functor $ \tau $ is left exact, and we write $\,\tau_k := \RR^k \tau:\,
\GrM(\DD) \to \Tors(\DD)\,$ for its derived functors.

We also introduce the quotient category $\, \Qgr(\DD) :=
\GrM(\DD)/\Fdim(\DD)\,$. This is an abelian category that comes
equipped with two canonical functors: the (exact) quotient functor
$\,\pi:\,\GrM(\DD) \to \Qgr(\DD) \,$ and its right adjoint (and
hence left exact) functor $\,\omega:\,\Qgr(\DD) \to \GrM(\DD)\,$.
The relation between  $ \pi $, $ \omega $ and $ \tau $ is
described by the following result which is part of standard
torsion theory (see, e.~g., \cite{AZ}, Prop.~7.2).
\bthm \la{cons} $(1)$\ The adjunction map $\, \eta_{\MM}:\, \MM
\to \omega\,\pi(\MM) $ fits into the exact sequence
\begin{equation}
\la{exfun}
0 \to \tau(\MM)\to \MM \xrightarrow{\eta_{\MM}}
\omega \pi(\MM) \to \tau_1(\MM)\ \to 0\ ,
\end{equation}
which is functorial in $ \MM \in \GrM(\DD) $.

$(2)$\ For $ k \ge 1 $, there are natural isomorphisms
\begin{equation}
\la{isex}
\RR^k \omega(\pi \MM) \cong \tau_{k+1}(\MM)\ .
\end{equation}
In particular, if $\, k \ge 1 $, the modules $ \RR^k \omega(\pi
\MM) $ are torsion. \ethm

Now, given a graded module $\,\MM = \bigoplus_{k \in \Z} \MM_k \,$
and $\, n \in \Z \,$, we write  $\, \MM[n] := \bigoplus_{k \in \Z}
\MM_{k+n} \,$ and $\,\MM_{\ge n} := \bigoplus_{k\ge n}\MM_k \,$.
Both are graded modules, $\,\MM_{\ge n}\,$ being a submodule of $
\MM $. With this notation, we compute  $ \RR^k \omega (\pi \DO) $,
regarding $\DO $ as a $\DD$-module via the algebra map $ p:\,\DD
\to \DO $.
\blemma \la{calc} $(1)$ The canonical map $\ \eta_{\DO}:\,\DO
\xrightarrow{\!\sim\!\!} \omega \pi (\DO)_{\ge 0} \,$ is an
isomorphism.

\hspace{11.2ex} $(2)$\ $\, \RR^{k} \omega (\pi \DO) = 0 \,$ for
$\,k \ge 1\,$. \elemma
\bproof For  graded $\DO$-modules $ \MO $ and $ \NO $, we define
(cf. \cite{AZ}, Sect.~3)
$$
\HOM_{\DO}(\MO,\,\NO) := \bigoplus_{k \in Z}\,
\Hom_{\tiny {\rm GrMod}(\DO)}(\MO,\,\NO[k])\ ,
$$
and write $\,\EXT_{\DO}^{n}(\MO,\,\NO) \,$ for the corresponding
Ext-groups. Combining \cite{AZ}, Theorem~8.3 and Proposition~7.2,
we then identify
\begin{equation}
\la{cohg}
R^k \omega (\pi \DO) \cong \dlim\ \EXT^k_{\DO}(\DO_{\ge n},\,\DO)\ ,
\quad \forall\ k\ge 0 \ .
\end{equation}
To compute the Ext-groups in \eqref{cohg} we use the long
cohomology sequence
\begin{equation}
\la{cohs}
\EXT^k_{\DO}(\DO_{n}[-n], \DO) \to
\EXT^k_{\DO}(\DO_{\ge n}, \DO)  \to \EXT^k_{\DO}(\DO_{\ge n+1}, \DO) \to
\EXT^{k+1}_{\DO}(\DO_{n}[-n], \DO)
\end{equation}
arising from the short exact sequence $\, 0 \to \DO_{\ge n+1} \to
\DO_{\ge n} \to \DO_{n}[-n] \to 0\,$, and the following
isomorphisms (for $ n \ge 0 $)
\begin{equation}
\la{extc1}
\EXT^k_{\DO}(\DO_n[-n], \DO) = 0 \quad \mbox{if} \quad k \not= 1 \ ,
\end{equation}
and
\begin{equation}
\la{extc2}
\EXT^1_{\DO}(\DO_n[-n], \DO)_m \cong
\left\{\begin{array}{ll}
0 & \ \mbox{if} \quad m \ne -n-1\\*[1ex]
\Sym^{-m}(\Omega^1 X) &\ \mbox{if} \quad m = -n-1
\end{array}\right.\ ,
\end{equation}
where $\, \Sym^{q} \,$ stands for the $q$-th symmetric power over
$A$. It is easy to see that \eqref{cohs}, \eqref{extc1} and
\eqref{extc2}, together with \eqref{cohg}, formally imply both 
statements of the lemma. (In addition, we have $\, \omega (\pi \DO)_n \cong
\Sym^{-n}(\Omega^1 X)\,$ for $\, n < 0\,$.)

To prove \eqref{extc1} and \eqref{extc2} we observe that
$\, \EXT^k_{\DO}(\DO_n[-n], \DO) \cong \EXT^k_{\DO}(\DO_n, \DO)[n]$, 
where $ \DO_n $ is a graded $\DO$-module with a single component in degree $0$.
Such modules arise by restricting scalars via
the algebra projection $\, \DO \to A \,$. So we can compute
$ \EXT^k_{\DO}(\DO_n, \DO) $ using the spectral sequence
\begin{equation}
\la{chr1}
\Ext_A^p(\DO_n,\,\EXT_{\DO}^q(A,\,\DO)) \ \underset{p}{\Rightarrow}\
\EXT_{\DO}^{p+q}(\DO_n,\,\DO)\ .
\end{equation}
To this end, we identify $\,\DO \,$ with
the symmetric algebra $\Sym_A(\Der\,A)$ and use the canonical resolution
\begin{equation}
\la{resd}
0 \to \DO \otimes_A \Der(A)[-1] \to \DO \to A \to 0\ .
\end{equation}
It follows from \eqref{resd} that $\,\EXT_{\DO}^q(A,\,\DO)= 0 \,$ for $\, q \ne 1 \,$,
so  \eqref{chr1} collapses on the line
$\, q = 1 \,$, giving (after natural identifications) the isomorphisms \eqref{extc1} and \eqref{extc2}.
\eproof
\blemma \la{fram} If $ \I $ is a flat $A$-module, then
$$
\RR^k \omega\,\pi(\MM \otimes_A \I) \cong \RR^k \omega\,(\pi \MM)
\otimes_A \I  \ ,\quad \forall\ k \ge 0 \ ,
$$
for any graded $\DD$-$A$-bimodule $ \MM $.
 \elemma
\bproof By \cite{AZ}, Prop.~7.2(1), we have
$\,
\RR^k \omega\,\pi(\MM \otimes_A \I) \cong
\dlim\ \EXT^k_{\,\DD} (\DD_{\ge n},\,\MM \otimes_A \I)
 \,$.
Since $\,\dlim \,$ commutes with tensor products, it suffices to
prove that
\begin{equation}
\la{exts}
\EXT^k_{\,\DD} (\DD_{\ge n},\,\MM \otimes_A \I) \cong
\EXT^k_{\,\DD} (\DD_{\ge n},\,\MM) \otimes_A \I\quad \mbox{for}\ n
\gg 0\ .
\end{equation}
Furthermore, by functoriality, it suffices to prove \eqref{exts}
only for $ k = 0 $, but in that case the result is well known (see \cite{Ro}, Lemma 3.83).
\eproof

Now, we turn to our problem. As in Section~\ref{SCI}, we choose a
good filtration $ \{M_k\} $ on $M$ so that $ \MO :=
\bigoplus_{k\in\Z} M_k/M_{k-1} $ is a torsion-free $\DO$-module.
Then, by Theorem~\ref{T4}, there is an ideal $\, \I \subseteq A
\,$ (unique up to isomorphism) and a graded embedding
\begin{equation}
\la{embd}
\fo :\, \MO \into \DO \I\ ,
\end{equation}
such that $\,\dim\, \Coker(\fo) < \infty \,$. The filtration
$\,\{M_k\}\,$ is uniquely determined by $ M $ up to a shift of
degree (cf. Lemma~\ref{L6} below); we fix this shift by requiring
$\,\fo \,$ to be of degree $0$. The dimension $\, n :=
\dim\,\Coker(\fo) \,$ is then an invariant of $ M $,
independent of the choice of filtration.

Since $\, \eta:\,\id \to \omega \pi \,$ is a natural
transformation, the map \eqref{embd} fits into the
commutative diagram
\begin{equation}
\la{d222}
\begin{diagram}[small, tight]
\MO              & \rTo^{\fo}              & \DO \I\\
\dTo^{\eta_{\MO}}&                        & \dTo_{\eta_{\DO \I}}\\
\omega \pi(\MO)  & \rTo^{\ \omega \pi(\fo)\ }   & \omega \pi(\DO \I)
\end{diagram}
\end{equation}
As $\,\Ker(\fo) = 0\,$ and $\,\Coker(\fo) \in \Fdim(\DD)\,$, $\,
\pi(\fo) \,$ and, hence, $\,\omega\pi(\fo)\,$ are isomorphisms. On
the other hand, by Lemma~\ref{fram},  $\,\eta_{\DO \I} \,$
can be factored as
$$
\DO \I \cong \DO \otimes_A \I \xrightarrow{\eta_{\DO} \otimes 1}
\omega \pi(\DO)\otimes_A \I \cong \omega \pi(\DO \I)
$$
and hence, by Lemma~\ref{calc}$(1)$, $\,\eta_{\DO \I}:\,\DO \I
\stackrel{\sim}{\to} \omega \pi(\DO \I)_{\ge 0} $ is an
isomorphism. Using these two isomorphisms, we identify
\begin{equation}
\la{idmap}
\omega \pi(\MO)_{\ge 0} \cong \DO \I\ .
\end{equation}
It follows then from \eqref{exfun} and 
\eqref{d222} that $\, \tau(\MO) = 0 \,$ and
$\,\tau_1(\MO)_{\ge 0}  \cong \Coker(\fo)\cong \DO \I/\MO\,$.
Hence
\begin{equation}
\la{rtau}
\dim \,\tau_1(\MO)_{\ge 0} = n\ .
\end{equation}

Next, we set $\,\NN := \bigoplus_{k \in \Z}\, M/M_k \,$ and make
$\NN $ a graded $ \DD$-module in the natural way, with $ t \in \DD
$ acting by the canonical projections $ M/M_k \onto M/M_{k+1} $.
\bprop
\la{propN}
The module $\NN$ has the following properties:

$(1)$ $\ \tau(\NN) = 0\,$,

$(2)$  $\ \dim \, \tau_1(\NN)_{-1} = n\,$, and $\, \dim
\, \tau_1(\NN)_{\ge -1} < \infty\,$,

$(3)$\ The maps $\,\omega\,\pi(\NN)_{k-1} \xrightarrow{t}
\omega\,\pi(\NN)_k \,$ are surjective for all $\,k\ge 0\,$.
 \eprop
\bproof $(1)$ Given $\,\MM \in \GrM(\DD)\,$, we write
$\,p^!(\MM)\,$ for the largest submodule of $ \MM $ annihilated by
the action of $\, t \,$, i.~e. $\, p^!(\MM) = \Ker(\MM
\xrightarrow{t\,\cdot\,} \MM[1])\,$. Then, if $\,\MM \in
\Fdim(\DD)\,$ and $ \MM \ne 0 $, we have $\,p^!(\MM) \ne 0\,$.
So the assumption $
\tau(\NN) \ne 0 $ implies that $\,p^!(\tau \NN) \ne 0 \,$. On the other
hand, $\,p^!(\NN) \cong \MO[1]\,$ and $ \tau(\MO[1]) =
\tau(\MO)[1] = 0 $, so $\, \tau(p^! \NN) = 0 \,$. Since $\,
p^!(\tau \NN)= p^!(\NN) \cap \tau(\NN)= \tau(p^! \NN) \,$, we
arrive at contradiction. It follows that $ \tau(\NN) = 0 $.

$(2)$ For all $ k \in \Z $, we have the exact sequences $\, 0\to
M_k/M_{k-1} \to M/M_{k-1} \to M/M_k \to 0 \,$ defined by the
filtration inclusions. Combining these together, we get the exact
sequence of graded $ \DD$-modules
\begin{equation}
\la{nseq}
0 \to \MO \to \NN[-1] \xrightarrow{t} \NN \to 0\ .
\end{equation}
Since $ \tau(\NN) = 0 $, applying the torsion functor $\,\tau \,$
to \eqref{nseq} yields
\begin{equation}
\la{lseq}
0 \to \tau_1(\MO) \to \tau_1(\NN[-1]) \to \tau_1(\NN) \to
\tau_2(\MO) \to \, \ldots
\end{equation}
By Theorem~\ref{cons}$(2)$ and Lemma~\ref{fram}, the last term of
\eqref{lseq} can be identified as
$$
\tau_2(\MO) \cong \RR^1 \omega (\pi \MO) \cong
\RR^1 \omega[\pi(\DO \I)] \cong \RR^1 \omega (\pi \DO) \otimes_A \I\ ,
$$
so $\, \tau_2(\MO) = 0\,$ by Lemma~\ref{calc}$(2)$. We get
thus the exact sequence
\begin{equation}
\la{kseq}
0 \to \tau_1(\MO) \to \tau_1(\NN[-1])
\xrightarrow{t} \tau_1(\NN) \to 0\ .
\end{equation}
Now, \eqref{rtau} implies that $ \tau_1(\MO)_{\ge 0} $ is
bounded: i.~e. there is an integer $\, d \ge 0 \,$, such that $\,
\tau_1(\MO)_{d} \ne 0 \,$, while $\, \tau_1(\MO)_{k} =  0
\,$ for all $\, k > d \,$. It follows then from \eqref{kseq} that
$\, t \,$ acts as a unit on $\, \tau_1(\NN)_{\ge d} \,$: in
particular, we have $\, p^! (\tau_1(\NN)_{\ge d}) = 0 \,$. But
$\,\tau_1(\NN)_{\ge d} \,$ is a submodule of $\, \tau_1(\NN) \,$, which, by definition, is torsion. Hence, $\, p^!
(\tau_1(\NN)_{\ge d}) = 0 \,$ forces $\,\tau_1(\NN)_{\ge
d} = 0\,$. Now, by induction, it follows from \eqref{kseq} that
$\,\dim\,\tau_1(\NN)_{k-1} = \sum_{j=k}^{d}
\dim\,\tau_1(\MO)_{j}\,$ for $\,k = 0,\,1,\,\ldots\, ,\, d\,$.
In particular, by \eqref{rtau}, $\, \dim\,\tau_1(\NN)_{-1} =
\dim\,\tau_1(\MO)_{\ge 0} = n \,$, and
$\,
\dim\,\tau_1(\NN)_{\ge -1} = \sum_{k\ge 0}
\dim\,\tau_1(\NN)_{k-1} \,$ is finite.

$(3)$ Applying $\,\omega \pi\,$ to \eqref{nseq} gives rise to the
exact sequence
\begin{equation}\la{sseq}
0 \to \omega\pi(\MO) \to \omega\pi(\NN[-1]) \xrightarrow{t}
 \omega\pi(\NN) \to \RR^1 \omega(\pi \MO) \to \, \ldots
\end{equation}
Since $\, \pi(\MO) \cong \pi(\DO \I) \,$, we have
$\,\omega\pi(\MO)_{\ge 0} \cong \omega\pi(\DO \I)_{\ge 0} \cong
\DO \I\,$, see \eqref{idmap}, and $\,\RR^1 \omega(\pi \MO) \cong
\RR^1 \omega(\pi \DO \I) \cong \RR^1 \omega(\pi \DO) \otimes_A \I
= 0 \,$, by Lemma~\ref{calc}$(2)$. Hence, truncating \eqref{sseq}
at negative degrees, we get the exact sequence
\begin{equation}
\la{fseq}
0 \to \DO \I \to \omega\pi(\NN[-1])_{\ge 0} \xrightarrow{t}
\omega\pi(\NN)_{\ge 0} \to 0\ .
\end{equation}
The last statement of the proposition follows.
\eproof

Next, we consider the functorial exact sequence \eqref{exfun},
with $\, \MM = \NN \,$. By Proposition~\ref{propN}$(1)$, the first
term of this sequence is zero, so we have
\begin{equation}
\la{exf1}
0 \to \NN \xrightarrow{\eta_{\NN}} \omega\pi(\NN)
\to \tau_1(\NN) \to 0 \ ,
\end{equation}
Since the canonical filtration on $M$ is positive, $\,\NN_k = M
\,$ for all $\, k < 0 \,$. Thus, setting $\,L := \omega
\pi(\NN)_{-1} \,$ and $\, V := \tau_1(\NN)_{-1} \,$, we get
from \eqref{exf1} the exact sequence of $A$-modules
\begin{equation}
\la{exf2}
0 \to M \xrightarrow{\eta} L \to V \to 0 \ .
\end{equation}

Now, replacing  $ A $ by its one-point extension $\, B = A[\I]\,$,
we lift \eqref{exf2} to an exact sequence of $B$-modules, as
follows. First, we regard $ M $ as a $ B$-module by restricting
scalars via the algebra homomorphism $\,\btheta:\,B \to A \,$, see
\eqref{E17}. Next, we set $\,\bL := L \oplus \c \,$ and make $ \bL
$ a $B$-module by defining its structure map $\, \varphi:\,\I
\otimes \c \cong \I \to L \,$ to be the degree $0$ component of
the canonical embedding $\,\DO \I \into \omega \pi(\NN[-1])_{\ge
0}\,$ in \eqref{fseq}. Every $A$-module homomorphism $ M \to L $
extends then to a unique $B$-module homomorphism $ M \to \bL $,
since $ \Hom_{A}(M, L) \cong \Hom_B(M, \bL) $ via $\,f \mapsto
(f,\,0)\,$. In particular, the map $\, \eta \,$ in \eqref{exf2}
extends to an embedding $\,\bi:\, M \into \bL \,$, and we write $
\bV := \bL/M $ for the cokernel of $ \bi $. Clearly, $\, \bV \cong
V \oplus \c \,$ as a vector space, and $\, \underline{\dim}(\bV) =
(n,\,1)$, by Proposition~\ref{propN}$(2)$. Summing up, we have
constructed an exact sequence of $ B$-modules
\begin{equation}
\la{exf3}
0 \to M \xrightarrow{\bi} \bL \to \bV \to 0 \ ,
\end{equation}
with the quotient term being of dimension $ (n,1) $. Moreover,
using Lemma~\ref{recoll}, we may regard $ M $ as a $ \Pi^\blambda(B)$-module.
\bprop
\la{piext}
The $ B$-module structure on $ \bL $ defined
above admits a {\rm unique} extension to $\Pi^\blambda(B)$, making
$\, \bi:\, M \to \bL \,$ a homomorphism of $
\Pi^\blambda(B)$-modules.
\eprop

We will give a homological proof of this proposition, using
Theorem~\ref{lift} of Section~\ref{DPA}. As explained in (the
proof of) Theorem~\ref{lift}, a $\Pi^\blambda(B)$-module structure
on a $B$-module $ \bM $ is determined by an element of $\,
\End(\bM) \otimes_{\eB} \Omega^1 B\,$, lying in the fibre of
$\,\blambda \,(\,=\, \blambda \cdot \id\,)\,$ under the evaluation
map
\begin{equation}
\la{par1}
\partial_{\bM}:\,\End(\bM) \otimes_{\eB} \Omega^1(B) \to
\End(\bM)\ ,\qquad f \otimes d \mapsto f \Delta_B(d)\ .
\end{equation}
In particular, the given $ \Pi^\blambda(B)$-module structure on $
M $ is determined by an element $\,\delta_M \in \End(M)
\otimes_{\eB} \Omega^1(B) \,$, such that $\,\partial_M(\delta_M) =
\id_M \,$. The $B$-module embedding $\,\bi \,$ induces an
embedding of $ B$-bimodules: $\,\End(M) \into \Hom(M,\,\bL) \,$,
and hence the natural map
\begin{equation}
\la{par2}
 \End(M) \otimes_{\eB} \Omega^1(B)
 \into \Hom(M,\,\bL) \otimes_{\eB} \Omega^1(B) \ .
\end{equation}
Since $ \Omega^1(B) $ is a projective bimodule, this last map is
also an embedding, and we  identify $\, \End(M) \otimes_{\eB}
\Omega^1(B) \,$ with its image in $\,\Hom(M,\,\bL) \otimes_{\eB}
\Omega^1(B)\,$ under \eqref{par2}.

Now, consider the commutative diagram
\begin{equation}
\la{d444}
\begin{diagram}[small, tight]
\End(\bL)\otimes_{\eB}\Omega^1(B)& \rTo^{\tbi_*} &
\Hom(M,\bL)\otimes_{\eB}\Omega^1(B) \\
\dTo^{\partial_{\bL}} && \dTo^{\partial_{M,\bL}} \\
\End(\bL)& \rTo^{\bi_*} & \Hom(M,\bL) \\
\end{diagram}
\end{equation}
where $\,\partial_{M, \bL}\,$ is the evaluation map at $
\Delta_B$, $\,\bi_* \,$ is the restriction (via $\,\bi\,$), and
$\, \tbi_* := \bi_* \otimes \id \,$. Note that $ \bi_* $ and $
\tbi_* $ are both surjective. With above identification, we have
\begin{equation}
\la{cond1}
\bi_*(\blambda) =
\partial_{M, \bL}(\delta_M) = \bi \ ,
\end{equation}
and our problem is to show that there is a unique element $\,
\delta_{\bL} \in \End(\bL)\otimes_{\eB}\Omega^1(B) \,$, such that
\begin{equation}
\la{cond2}
\partial_{\bL}(\delta_{\bL}) = \blambda\quad \mbox{and}\quad
\tbi_*(\delta_{\bL}) = \delta_M \ .
\end{equation}

To solve this problem homologically, we interpret the top and the
bottom rows of \eqref{d444} as $2$-complexes of vector spaces $
X^\bullet $ and $ Y^\bullet $, with nonzero terms in degrees $ 0 $
and $ 1 $ and differentials $ \tbi_* $ and $ \bi_* $,
respectively. The pair of maps $\, (\partial_{\bL},\, \partial_{M,
\bL})\,$ yields then a morphism of complexes $\,\partial^\bullet:
X^\bullet \to Y^\bullet \,$ with mapping cone
\begin{equation} \la{cone} C^\bullet(\partial) := \left[\,
 0 \to X^0
\xrightarrow{d^{-1}} X^1 \oplus Y^0 \xrightarrow{d^0} Y^1 \to 0\,
\right]\ .
\end{equation}
By definition, the differentials in $ C^\bullet(\partial) $ are
given by $\,d^{-1} = (-\tbi_*,\,\partial_{\bL})^T \,$ and $\, d^0
= (\partial_{M, \bL},\,\bi_*) \,$. So  \eqref{cond1}
can be interpreted by saying that $\, (-\delta_{M}, \blambda) \in
X^1 \oplus Y^0 \,$ is a $0$-cocycle in $ C^\bullet(\partial) $.
Then, the cohomology class
\begin{equation}\la{cohcl} c(\blambda, \delta_M) := [(-\delta_{M},
\blambda)]
\end{equation}
represented by this cocycle, vanishes in $\,h^0(C^\bullet)\,$ if
and only if there is $\, \delta_{\bL} \in X^0 \,$ such that $\,
d^{-1}(\delta_{\bL}) = (-\delta_{M}, \blambda)\,$, i.~e.
\eqref{cond2} holds. Clearly, if it exists, such $ \delta_{\bL} $
is unique if and only if $ d^{-1} $ is injective, i.~e. if and
only if $\, h^{-1}(C^\bullet) = 0\,$. Now, a simple
calculation (as in the proof of Theorem~\ref{lift}) shows that
$$
h^{0}(C^\bullet) \cong \HH_0(B,\,\Hom(\bV,\bL))\quad \mbox{and}\quad
h^{-1}(C^\bullet) \cong \HH_1(B,\,\Hom(\bV,\bL))\ .
$$
Proposition~\ref{piext} thus boils down to proving Lemma~\ref{h10}
and Lemma~\ref{c0} below.
\blemma \la{h10} $\ \HH_1(B,\,\Hom(\bV,\bL)) = 0 \,$.
\elemma
\bproof Recall that $\,L := \omega\,\pi(\NN)_{-1} \,$. Now, in
addition, we set $\, L_0 := \omega\,\pi(\NN)_{0} \,$ and make this
a $B$-module by restricting scalars via $\,\btheta: B \to A \,$.
Then, the $A$-module homomorphism $\,t_L:\, L \to L_0 \,$ induced
by the action of $ t $ extends to a unique $B$-module homomorphism
$\,\bL \to L_0 \,$, which we denote by $ \bt $. By Proposition~\ref{propN}$(3)$,
$\, t_L \,$ is surjective, and hence so is $\, \bt \,$. It
is easy to see that $\,\Ker(\bt) \cong B\ei $, so
we have the exact sequence of $B$-modules
\begin{equation} \la{res0} 0 \to B \ei \xrightarrow{\biota} \bL \xrightarrow{\bt} L_0 \to
0\ .
\end{equation}
Since $\, B \ei $ is projective, tensoring \eqref{res0} with $
\bV^* := \Hom(\bV,\,\c) $ yields
$$
0 \to \Tor^B_1(\bV^*,\, \bL) \to \Tor^B_1(\bV^*,\, L_0) \to
\bV^* \otimes_B B \ei \to \bV^*\otimes_B \bL \to
\bV^*\otimes_B L_0 \to 0\ .
$$
On the other hand, since $ \bV $ is finite-dimensional, for
an arbitrary $B$-module $ \bM $, we have natural isomorphisms
$\,\HH_n(B,\,\Hom(\bV,\bM)) \cong \Tor^B_n(\bV^*,\, \bM)\,$.
So the above exact sequence can be identified with
\bg{eqnarray} \la{HHn} &&\  0 \to \HH_1(B,\,\Hom(\bV,\bL)) \to
\HH_1(B,\,\Hom(\bV,L_0)) \to  \\
&& \quad \HH_0(B,\,\Hom(\bV, B\ei)) \to \HH_0(B,\,\Hom(\bV,\bL))
\to \HH_0(B,\,\Hom(\bV,L_0)) \to 0\ . \nonumber
\end{eqnarray}
To prove the lemma it thus suffices to show that
\begin{equation} \la{ttt}
\HH_1(B,\,\Hom(\bV,L_0)) \cong \Tor^B_1(\bV^*,\, L_0) = 0 \ .
\end{equation}
Since $ L_0 $ is an  $A$-module, we can compute this last Tor,
using the spectral sequence
\begin{equation}\la{sps}
 \Tor^A_p(\Tor^B_q(\bV^*,\, A),\, L_0)\
\underset{p}{\Rightarrow}\ \Tor^B_{p+q}(\bV^*,\,L_0)
\end{equation}
associated to the algebra map $\btheta: B \to A $. By
Lemma~\ref{LL6}$(4)$, this map is flat, so \eqref{sps} collapses at $q=0$, giving an isomorphism
\begin{equation}\la{ttt1}
\Tor_1^B(\bV^*,\,L_0) \cong \Tor^A_1(V^*,\,L_0)\ .
\end{equation}
Now, for each $\,k\ge 0\,$, we set $\,L_k := \omega
\pi(\NN)_{k}\,$ and write $\, F_k \,$ for the kernel of the map
$\,L_0 \xrightarrow{t^k} L_k\,$ induced by the action of $\, t^k
\in \DD \,$ on $ \omega\pi(\NN)$. By Proposition~\ref{propN}$(3)$,
the maps $ t^k $ are surjective for all $ k \ge 0 $, and thus $\,0=F_0
\subseteq F_1 \subseteq F_2 \subseteq \,\ldots \,$ is an
$A$-module filtration on $ L_0 $. By Proposition~\ref{propN}$(2)$, this filtration is exhaustive, so that
$\,\dlim\, F_k \cong \bigcup_{k=0}^{\infty} F_k =
L_0\,$, while, by \eqref{fseq}, we have exact sequences
\begin{equation}
\la{filtf}
0 \to F_k \to F_{k+1}\to \DO_{k+1} \I \to 0\quad , \quad \forall\ k\ge 0 \ .
\end{equation}
Since  $\, \DO_{k+1} \I \,$ are projective $A$-modules for $ k \ge 0 $, we conclude from \eqref{filtf} that $ F_k $ are projective for $\, k \ge 1 \,$. The
direct limits of families of projective modules are flat, hence so is
$ L_0 =  \dlim\, F_k $. We have $\,
\Tor^A_1(V^*,\,L_0)=0\,$, and \eqref{ttt} follows from \eqref{ttt1}.
\eproof

\blemma
\la{c0}
$\ c(\blambda,\,\delta_M) = 0\,$ in $\,\HH_0(B,\,\Hom(\bV,\bL))\,$.
\elemma

\bproof By \eqref{HHn} and \eqref{ttt}, we have the exact sequence
\begin{equation} \la{HH0} 0 \to \HH_0(B,\Hom(\bV, B\ei))
\xrightarrow{\biota_*} \HH_0(B,\Hom(\bV,\bL)) \xrightarrow{\bt_*}
\HH_0(B,\Hom(\bV,L_0)) \to 0 ,
\end{equation}
where $ \biota_* $ is induced by the inclusion $\,\biota:\,B\ei
\into \bL \,$ and $ \bt_* $ by the projection $\,\bt:\,\bL \onto
L_0\,$ in \eqref{res0}. We show first that
$\,\bt_*(c(\blambda,\,\delta_M)) = 0 \,$. For
this, we consider the image of the diagram \eqref{d444} under the
natural projection $\,\bt\,$. Under this projection,
the equations \eqref{cond1} become
\begin{equation} \la{cond11} \bi_*(\bt) = \partial_{M, L_0}(\tbt_*(
\delta_M)) = \bt\circ \bi\ ,
\end{equation}
where $\, \tbt_*:\, \Hom(M, \bL)\otimes_{\eB}\Omega^1(B) \to
\Hom(M, L_0)\otimes_{\eB}\Omega^1(B) \,$ is defined by $\, f
\otimes \,\omega \mapsto (\bt \circ f) \otimes \omega\,$. Now,
$\,\bt \,$ induces a morphism of mapping cones \eqref{cone}
associated to \eqref{d444} and its projection, which, in turn,
induces a map $ \bt_* $ on cohomology. The class
$\,\bt_*(c(\blambda,\,\delta_M)) \in \HH_0(B,\Hom(\bV,L_0)) \,$
can thus be viewed as an obstruction for the existence of an
element $\,\delta_{\bL, L_0} \in \Hom(\bL,
L_0)\otimes_{\eB}\Omega^1(B)\,$ satisfying
\begin{equation}
\la{cond22}
\partial_{\bL, L_0}(\delta_{\bL, L_0}) =
\bt \quad \mbox{and}\quad
\tbi_*(\delta_{\bL, L_0}) = \tbt_*(\delta_M) \ .
\end{equation}
We will show that $\,\bt_*(c(\blambda,\,\delta_M)) = 0 \,$ by
constructing such an element explicitly.

By universal property of tensor algebras, the filtered
algebra homomorphism $\,R \onto \D\,$ lifts to a
{\it graded} algebra homomorphism $\,R \onto \DD \,$,
so we may regard graded $\DD$-modules as graded $R$-modules.
The action of $ \DD $ on $ \omega\,\pi(\NN) $ yields an
$A$-bimodule map $\, \DDer(A) \to \Hom(L, \,L_{0})\,$, taking
$\,\Delta_A \mapsto t_L \,$, which in composition with
$\,\DDer(B) \to \DDer(A)\,$ gives a $B$-bimodule homomorphism $\,\DDer(B) \to \Hom(L, \,L_{0})\,$, $\,\Delta_B \mapsto t_L \,$, and hence an element
\begin{equation} \la{bbf} \delta_{L, L_0} \in \Hom_{\eB}(\DDer(B),
\,\Hom(L, \,L_{0})) \cong \Hom(L, \,L_{0}) \otimes_{\eB}\Omega^1(B)\ .
\end{equation}
Now, let $\,\balpha: L = e \bL \into \bL \,$ be the natural inclusion, with $ \bL/L = \ei \bL \cong \c $. Viewing $ L $ as a $B$-module
via $ \btheta $ makes it a $B$-module map.
Dualizing $\ba$ by $ L_0 $ and tensoring with $ \Omega^1(B)
$, we get then the commutative diagram
\begin{equation} \la{d66}
\begin{diagram}[small, tight]
\Hom(\bL, L_0)\otimes_{\eB}\Omega^1(B)& \rTo^{\tba_*} &
\Hom(L,L_0)\otimes_{\eB}\Omega^1(B) \\
\dTo^{\partial_{\bL, L_0}} && \dTo^{\partial_{L,L_0}} \\
\Hom(\bL, L_0)& \rTo^{\ba_*} & \Hom(L, L_0) \\
\end{diagram}
\end{equation}
with $\,\partial_{L, L_0}(\delta_{L,L_0}) = \ba_*(\bt) = t_L \,$.
Since $\, e \in B \,$ acts as identity on $ L_0 $ and as zero
on $ \bL/L $, we have
$\,
\HH_0(B,\,\Hom(\bL/L, L_0)) \cong (\bL/L)^* \otimes_B L_0 = 0\,$.
Hence, there is  $\, \delta_{\bL, L_0} \in \Hom(\bL,
L_0)\otimes_{\eB}\Omega^1(B)\,$, such that $\,\partial_{\bL,
L_0}(\delta_{\bL,L_0}) = \bt\,$ and $\,\tba_*(\delta_{\bL, L_0}) =
\delta_{L, L_0}\,$. A direct calculation using
$\,\bi = \ba \circ \eta \,$ shows that this element satisfies
also \eqref{cond22}.

The existence of $\,\delta_{\bL, L_0}\,$ implies that $\,
\bt_*(c(\blambda, \delta_M)) = 0 \,$. Returning to \eqref{HH0}, we
see then that $\,c(\blambda, \delta_M) =
\biota_*(\tilde{c})\,$ for some $\,\tilde{c} \in \HH_0(B,\Hom(\bV,
B\ei))\,$. Now, to show that $\,\tilde{c} = 0\,$ we consider the trace
map $\,\Tr:\, \Hom(\bV,\, B\ei) \to \Hom(\bV, \bL)
\to \End(\bV) \to \c $, $\,f \mapsto \tr_{\bV}[\,\bp \circ \biota \circ f\,]\,$, where $\, \biota \,$ is defined
in \eqref{res0}, $\, \bp\, $ is the canonical
projection in \eqref{exf3}. Since $ \biota $ and $ \bp $ are homomorphisms, this induces a linear map
$$
\Tr_*:\ \HH_0(B,\Hom(\bV, B\ei)) \xrightarrow{\biota_*}
\HH_0(B,\Hom(\bV, \bL)) \xrightarrow{\bp_*}
\HH_0(B,\End(\bV)) \xrightarrow{\tr_*} \c\ .
$$
We claim that $\,\Tr_*\,$ is an isomorphism. Indeed, it is easy to
see that $\,\Tr_* \ne 0 \,$, while
$$
\HH_0(B,\Hom(\bV, B\ei)) \cong \bV^* \otimes_B B
\ei \cong \bV^* \ei \cong (\ei \bV)^* \cong \c\ .
$$
Now, since $\,\bp \circ \bi = 0 \,$, we have $\,\bp_*(c(\blambda,
\delta_M)) = [\,\blambda \cdot \id_{\bV}\,]\,$, and hence
$$
\Tr_*(\tilde{c}) := \tr_{\bV}[\,\bp_*\,\biota_*(\tilde{c})\,]
= \tr_{\bV}[\,\bp_*(c)\,] = \tr_{\bV}[\,\blambda \cdot
\id_{\bV}\,] = 0\ .
$$
It follows that $\,\tilde{c} = 0\,$ and $\, c(\blambda,\,\delta_M)
= 0 \,$, finishing the proof of the lemma and
Proposition~\ref{piext}. \eproof

Now, by Proposition~\ref{piext}, the given $B$-module structure on
$\,\bV = \bL/M \,$ extends to a unique $ \Pi^\blambda(B)$-module
structure, making \eqref{exf3} an exact sequence of
$\Pi^\blambda(B)$-modules. Since $\,\ei \bV \cong \ei \bL\,$, the natural map $\, \Pi \ei \otimes_{U} \ei \bV \cong
\Pi \ei \otimes_{U} \ei \bL \to \bL \,$ is an isomorphism, which
in combination with projection $\,\bp: \bL \to \bV \,$ becomes
$\bmu_{\bV}: \Pi \ei \otimes_{U} \ei \bV \to \bV $. It follows that
$\, \Ker(\bp) \cong \Ker(\bmu_{\bV})\,$, and thus $\, M \cong
\btheta^*(\bV)\,$. This completes Step~2.

\subsection{The map $ \omega $ is injective and
$\Gamma$-equivariant}
\la{st3} For $\Pi$-modules $ \bV $ and $ \bV' $ of dimension
$ \bn = (n,1) $, we will show that
\begin{equation}
\la{eqv}
\btheta^*(\bV) \cong \btheta^* (\bV')\quad
\Longleftrightarrow \quad \bV' \cong \bV^{\sigma_{\omega}} \quad
\mbox{for some}\ \, \omega = u^{-1} du \in \Omega^1 X \ ,
\end{equation}
where $ \bV^{\sigma} $ denotes the $\Pi$-module $ \bV $ twisted by
an automorphism $ \sigma \in \Aut_S \,\Pi^\blambda(B) $.

We begin by describing the action \eqref{auto4} in terms of
generators of $ \Pi^\blambda(B) \,$ (see Proposition~\ref{gene}).
\blemma \la{actgen} The homomorphism $\,\sigma:\,\Omega^1 X \to
\Aut_S\,\Pi\,$ is given by
\begin{equation}
\la{gener} \sigma_\omega(\ha) = \ha\ ,\quad \sigma_\omega(\hv_i) =
\hv_i \ ,\quad \sigma_\omega(\hw_i) = \hw_i \ ,\quad
\sigma_\omega(\hd) = \hd + \widehat{\omega(d)} \ ,
\end{equation}
where $\, \omega \in \Omega^1 X \,$ acts on $\, d \in \DDer(A)
\,$ via the natural identification
\begin{equation*}
\la{Kara}
\Omega^1 X = (\Omega^1 A)_\natural \cong
\Hom_{\eA}((\Omega^1 A)^\star,\,A)
\cong \Hom_{\eA}(\Der(A, \,\AA),\,A) \ .
\end{equation*}
\elemma
\bproof By Lemma~\ref{DPApres}, we can define \eqref{auto4} in
terms of relative differentials
\begin{equation}
\la{relset}
\sigma:\, \Omega^1 X \stackrel{\alpha}{\too}
(\Omega_S^1 B)_{\natural} \stackrel{\tsigma}{\too}
\Aut_{B}[T_B (\Omega_S^1 B)^\star] \to
\Aut_S\,\Pi^\blambda(B)\ ,
\end{equation}
where $\, \alpha\,$  is now an isomorphism. In fact,
with identification \eqref{omegaS}, the elements of
$\, (\Omega^1_S B)_\natural = \Omega_S^1 B\!/[B, \Omega_S^1 B] \,$
can be represented by matrices $\, \hat\omega = \begin{pmatrix} \omega & 0 \\
0 & 0
\end{pmatrix}
\,$ with $\,\omega \in (\Omega^1 A)_\natural = \Omega^1 X\,$,
and $\, \alpha \,$ is given explicitly by $\, \omega \mapsto \hat\omega\ \mbox{mod}\, [B,
\Omega_S^1 B]\,$. It follows then that $ \sigma_\omega $ acts on $
\Pi^\blambda (B) $ as in \eqref{gener}.
\eproof

We may also describe the algebra map $\, \btheta:\,\Pi^\blambda(B)
\to \D \,$  in terms of generators of $ \Pi^\blambda(B) $:
\begin{equation}
\la{bth}
\btheta(\ha) = \overline{a}\ , \quad \btheta(\hd) = \overline{d} \ ,
\quad \btheta(\hv_i) = \btheta(\hw_i) = 0\ ,
\end{equation}
where $ \overline{a} $ and $ \overline{d} $ denote the classes of
$\, a \in A \,$ and $\, d \in \DDer(A) \,$ in $\, T_A \DDer(A) \,$ modulo the ideal $\, \langle \Delta_A - 1 \rangle \,$.
Comparing now \eqref{gener} and \eqref{bth}, we get
\blemma
\la{Lproj}
The group homomorphism $ \hsigma:\,\Omega^1 X
\stackrel{\sigma}{\too} \Aut_S\, \Pi \to \Aut_\c\,\D\,$ induced by
$ \sigma $  is given by
\begin{equation}
\la{ath}
\hsigma_\omega(a) = a \ ,\quad
\hsigma_\omega(\partial) = \partial + \omega(\partial)
\ ,
\end{equation}
where $\,a \in A \,$, $\,\partial \in \Der(A)\,$ and $\,\omega \in
\Omega^1 X\,$.
\elemma

In particular, if $\, \omega = u^{-1} du \,$ for some $\,u \in \Lambda \,$, then $\, \hsigma_\omega(a) = a
= u \,a \, u^{-1} \,$ (since $A$ is commutative), and $\,
\hsigma_\omega(\partial) = u \, \partial \, u^{-1} \,$. Thus, the
induced action of $\, \Lambda \subset \Omega^1 X \,$ on $ \D $ is
given by {\it inner} automorphisms. In contrast, $\,\Lambda\,$
does not act by inner automorphisms on the whole of $
\Pi^\blambda(B) $.

Now, by functoriality, $\, \bV' \cong \bV^{\sigma_{\omega}} \,$
implies $\, \bL' \cong \bL^{\sigma_{\omega}}\,$ and $\, \btheta^*(\bV')
\cong \btheta^*(\bV)^{\sigma_{\omega}} \,$ for
any $ \omega \in \Omega^1 X $. So the map $ \CC_n(X,\,\I) \to
\R(\D) $ induced by $ \btheta^* $ is equivariant under the action
of $ \Omega^1 X$. On the other hand, $ \btheta^*(\bV) $ is a $
\D$-module, on which the twisting by $\, \omega \,$ acts via
\eqref{ath}, i.~e. $\,\btheta^*(\bV)^{\sigma_{\omega}} =
\btheta^*(\bV)^{\hsigma_{\omega}} \,$. Since the inner
automorphisms induce trivial auto-equivalences, we have
$\, \btheta^*(\bV)^{\sigma_{\omega}} \cong \btheta^*(\bV) \,$
for $ \omega = u^{-1} du $. This
proves the implication `$\, \Leftarrow \,$' in \eqref{eqv} and,
in combination with Step~1, yields a $\Gamma$-equivariant map
$\, \omega_n:\, \overline{\CC}_n(X, \I) \to
\gamma^{-1}[\I]\,$.

It remains to show that $ \omega_n $ is injective. For this,
we will use the following result, which is a version
of \cite{BW2}, Lemma~10.1, and \cite{NS}, Lemma~3.2. (In particular,
the proof given in the last reference extends trivially to our situation.)

\blemma
\la{L6}  Let $ M $ be a (nonzero) ideal of $ \D $ equipped
with two good filtrations $ \{ M_k \} $ and $ \{ M_k' \} $, such
that the associated graded modules $ \MO $ and $ \MO' $ are both
torsion-free. Then, there is $ \, k_0 \in \Z \,$, such that $ \,
M_k = M_{k-k_0}' \,$ for all $ k \in \Z $.
\elemma

Given two $ \Pi$-modules $ \bV $ and $ \bV' $ of dimension
$\, \bn\,$, we set $\, \bL := \Pi \ei \otimes_{U} \ei \bV$, $\, L
:= e \bL\,$, $\,M := \btheta^*(\bV)$, and similarly for $ \bV' $.
In addition, we denote by $\, \eta:\, M \into L\,$ and $\, \bi:\,
M \into \bL\,$ the natural inclusions (and similarly for $ M' $).
\bprop
\la{bstr} If $\, M \cong M' \,$ as $\D$-modules, then $\,
\bL \cong \bL' \,$ as $B$-modules. \eprop
\bproof First, we show that every $\D$-module isomorphism $\,f:\,M
\to M' \,$ lifts to an $A$-module isomorphism $\,f_{L}:\,L \to
L'\,$.  For this, we identify $ L $ as in Lemma~\ref{iso2}, filter it by $ \{F_k L\}$ as in Section~\ref{st1}, and set $\,\LL := \bigoplus_{k \in \Z} L/F_k L\,$.
By \eqref{iso3}, we have
$\, \Delta_A \cdot x \equiv x\,(\mbox{mod}\,F_0 L)\,$
for all $\, x \in L \,$, so $\,
\Delta_A\,[x]_k = [x]_{k+1} = t\,[x]_k \,$ for $ k \ge -1 $.
Since $\,R[t]/\langle \Delta_A - t \rangle \cong \DD \,$, we may
regard $\,\LL_{\ge -1} \,$ as a graded $\DD$-module.

Next, we equip $ M $ with the induced filtration $\, M_k := M \cap
F_k L \, $ via the inclusion $\,\eta:\, M \into L\,$, and put
$\,\NN := \bigoplus_{k\in \Z} M/M_k \,$. The map $\,\eta\,$
naturally extends to $\,\teta:\, \NN \into \LL
\,$, and $\,\NN\,$ becomes a graded $ \DD$-module via the induced
action of $\,R[t]\,$ on $ \LL $. It follows from \eqref{filtin}
that $\, \MO := \bigoplus_{k \in \Z} M_k/M_{k+1} \,$ is a
torsion-free $ \DO$-module, and hence $ \tau(\NN) = 0 $ by
Proposition~\ref{propN}$(1)$. Let  $\, \eta_{\NN}:\,\NN \into \omega \pi(\NN) \,$, see \eqref{exfun}. Since $\,\Coker\,\teta\,$ is finite-dimensional in degree $ \ge -1$, the map $\, \eta_{\NN}
\,$ extends to an embedding: $\,\LL_{\ge -1
} \into \omega \pi(\NN)_{\ge -1}\,$. By induction in grading,
using Proposition~\ref{propN}$(2)$ and
\eqref{fseq}, it is easy to show that this embedding
is an isomorphism.

Now, replacing $ L $ by $ L'$, we repeat the above construction.
The $\D$-module $ M' $ comes then equipped with two filtrations:
one is induced from $ L' $ via  $\, \eta':\,M' \into L'
$, and the other is transferred from $ M $ via $\, f:\, M \stackrel{\sim}{\to} M' $. Both filtrations satisfy the
assumptions of Lemma~\ref{L6} and, hence, coincide up to a shift in degree. Since $ \MO' $ and $ f(\MO) $
have finite codimension in $ \LO' $, this last shift must be $0$
so  $\,M_{k}' = f(M_k) \,$ for all $ k \in \Z $.
The map $ f $ extends then to an isomorphism
$\,\tf:\, \NN \to \NN' \,$ and further, by functoriality, to $\,\omega\pi(\tf):\, \omega\pi(\NN) \to
\omega\pi(\NN') \,$. As a result, we get
$\, \LL_{\ge -1} \cong \omega\pi(\NN)_{\ge -1} \stackrel{\sim}{\to} \omega\pi(\NN')_{\ge -1} \cong
\LL_{\ge -1}' $, which in degree $(-1)$ yields the required extension $\, f_L:\,L \to L' \,$.

Now, with our identifications of $ L $ and $ L' $, the $B$-modules $ \bL $ and $ \bL' $ are determined (up to
isomorphism) by the triples $\, (L,\,\c,\,\varphi)\,$ and $\,
(L',\,\c,\,\varphi')\,$, where $\, \varphi:\,\,\I \into L \,$ and
$\, \varphi':\, \I \into  L' \,$ are the canonical embeddings with
images $\, F_0 L \,$ and $ F_0 L' $ respectively. Since $\,F_0
L\,$ is the kernel of $\,\LL_{-1} \stackrel{t}{\onto} \LL_0\,$, the map $\,f_L
\,$ restricts to $ F_0 L $, giving an isomorphism $\,f_L|_0:\, F_0
L \to F_0 L' \,$. Letting  $\, u := (\varphi')^{-1} \circ (f_L|_0) \circ
\varphi\, \in \Aut_A(\I) \,$ and identifying $\, \Aut_A(\I) =
\End_A(\I)^{\times} \cong A^{\times}\,$ via the action map, we
have $\,u\,\varphi' = \varphi' u = f_L\,\varphi \,$. Hence
\begin{equation}
\la{rem}
\bgg := (u^{-1} f_L,\,\id):\ L\oplus \c \to L'
\oplus \c
\end{equation}
makes the diagram \eqref{1-hom} commutative and thus defines
an isomorphism of $B$-modules $\,\bL \stackrel{\sim}{\to} \bL'\,$.
\eproof

Now, keeping the notation of Proposition~\ref{bstr}, consider two
$ \Pi$-modules $ \bV $ and $ \bV' $ of dimension $ \bn $, with $\,
M \cong M' \,$. Fix an isomorphism $\,f: \,M \to M'\,$ and define
$ \bgg $ as in \eqref{rem}. Taking $\,\omega = u^{-1} d u \in
\Omega^1 X\,$ and twisting $ \bi $ by $ \sigma = \sigma_{\omega}
\in \Aut_S\,\Pi\,$, consider the diagram
\begin{equation} \la{comd3}
\begin{diagram}[small, tight]
\bL^{\sigma} & \rTo^{\bgg} & \bL'\\
\uTo^{\bi} && \uTo_{\bi'} \\
M^{\sigma} & \rTo^{\ f u^{-1} \ } & M'
\end{diagram}
\end{equation}
From the construction of $f$ and $\bgg$, it follows that this diagram is commutative, with all arrows being {\it
$\Pi$-module} homomorphisms and horizontal ones being
isomorphisms. Thus, identifying $\,M^\sigma \cong M' \,$  and
$ \bL^\sigma \cong \bL' \,$ in \eqref{comd3}, we
get two ({\it a priori}\, different) $\Pi$-module structures on $
\bL' $. Both of these are extensions of the given $\Pi$-module
structure on $ M' $. Hence, by
Proposition~\ref{piext}, they must coincide. It follows that $\, \bgg:
\,\bL^\sigma \to \bL'$ is an isomorphism of $\Pi$-modules, which,
by commutativity of \eqref{comd3},  induces an isomorphism
$\, \bV^\sigma \cong \bV' $. This completes Step~3.

\subsection{The equivariance of $ \omega $ under the action of
$\Pic(\D)$} \la{st4}  As in Section~\ref{CMMap}, we will assume that
$ X \ne \A^1 $. By \cite{CH1}, Prop.~1.4, the automorphism group
of $ \D $ is then isomorphic to the product $\, \Aut(X) \ltimes \Omega^1 X \,$:
\begin{equation}
\la{autd}
\Aut(X) \ltimes \Omega^1 X  \stackrel{\sim}{\to} \Aut(\D) \ ,
\quad (\nu,\,\omega) \mapsto \hnu \, \hsigma_\omega\ ,
\end{equation}
where $\, \hnu \in \Aut(\D):\, D \mapsto \nu \,D\,\nu^{-1}\,$, and
$ \hsigma_\omega $ is defined by \eqref{ath}.
Now, for a line bundle $ \F $ on $X$, $\,\End_{\D}(\F \D) \,$ is canonically isomorphic to the ring of
twisted differential operators on $X$ with coefficients in $\F$.
As $X$ is affine, this last ring is isomorphic to $\D$,
so the set of all algebra isomorphisms: $\,\D \to \End_{\D}(\F\D)
\,$ is non-empty and equals $\, \psi_0\,\Aut(\D) \,$, where $
\psi_0 $ is a fixed isomorphism. By \cite{CH1}, Th.~1.8, the
isomorphism $\, \psi_0 $ can be chosen in such a way that $
\psi_0\,|_A = \id \,$: specifically, fixing dual bases $\, \{\alpha_i\} \subset \F \,$,
$\,\{\beta_i\} \subset \F^\vee $, and identifying $\,\End_{\D}(\F
\D) = \F \, \D \,\F^\vee \,$ as in Section~\ref{ad}, we  define
$\, \psi_0:\,\D \stackrel{\sim}{\to} \End_{\D}(\F\D) \,$  by
\begin{equation}
\la{vpo}
\psi_0(a) = a \ , \quad  \psi_0(\partial) = \sum_i\, \alpha_i \,
\partial \, \beta_i\ , \quad a \in A\ ,\quad \partial \in \Der(A)\ .
\end{equation}
With \eqref{autd} and \eqref{vpo}, every isomorphism $\,\psi:\, \D \to \End_{\D}(\F \D) \,$ can then be decomposed as
\begin{equation}
\la{decp}
\psi = \psi_0 \,\hnu\,\hsigma_{\omega}\ ,
\end{equation}
where $\, \nu \in \Aut(X) \,$ and $\, \omega \in \Omega^1 X \,$
are uniquely determined by $ \psi $.

\begin{proof}[Proof of Proposition~\ref{Lex}] Given a line bundle $ \I $
and an invertible bimodule $\,\P = (\D\L)_\vp\,$, with $\,\vp:
\,\D \stackrel{\sim}{\to} \End_{\D}(\D\L)\,$, we set $\, \tau :=
\vp|_A \,$, $\, \J := \L \tau(\I) \,$, $\, \F := \L^\tau =
\tau^{-1}(\L) \,$, and $\,\psi = \vp^{-1}:\,\D \to \End_{\D}(\F
\D)\,$, as in Section~\ref{ad}. To construct an isomorphism
$\,\bvp\,$, satisfying Lemma~\ref{Lex}, we
decompose $\,\psi \,$ as in \eqref{decp}, and extend each
factor through $ \btheta $. Since $
\psi_0 $ and $ \hsigma_\omega $ act on $A$ as identity, we have
$\, \nu = \psi|_A = \tau^{-1}\,$, so $\, \hnu = \htau^{-1} \,$ in
\eqref{decp}.  Thus we set
$$
\bvp \,:\  \Pi^\blambda(A[\J]) \xrightarrow{\sigma_\omega} \Pi^\blambda(A[\J]) \xrightarrow{\ttau^{-1}} \Pi^\blambda(A[\J^\tau])
\xrightarrow{\bvp_0} \End_{\Pi^\blambda(B)}(\bP)\ ,
$$
where $\,\sigma_\omega \,$ is
defined in Section~\ref{ad} (see \eqref{auto4}, with $ B $
replaced by $ A[\J] $) and $\,\ttau^{-1}\,$ is induced by
$\, A[\J] \to A[\J^\tau] \,$. The relation
$\,\btheta\,\sigma_{\omega} = \bar{\sigma}_\omega\,\btheta\,$
is then immediate, by Lemma~\ref{Lproj}.

It remains to define $ \bvp_0 \,$. To this end, we use
identification \eqref{endp}. Since $\, \J^\tau = \F \I\,$, we have then $\, A[\J^\tau] \cong \tilde{\F} \otimes_{\tA} B
\otimes_{\tA} \tilde{\F}^\vee \into \tilde{\F} \otimes_{\tA}
\Pi^\blambda(B) \otimes_{\tA} \tilde{\F}^\vee\,$,
which we take as a definition of $\,\bvp_0 \,$ on
$\,A[\J^\tau]\,$. This induces the identity on $ A $, as required. Next, we construct a bimodule isomorphism:
\begin{equation}
\la{mdo}
\Der_S(A[\F\I], \, A[\F\I]^{\otimes 2}) \to
\tilde{\F} \otimes_{\tA} \DDer_S(B)
\otimes_{\tA} \tilde{\F}^\vee\ ,
\end{equation}
using the dual bases for $ \F $ and $ \I $. By Lemma~\ref{DerS},
we first identify the domain of \eqref{mdo} with
\begin{equation}
\la{rai}
\left(
\begin{array}{cc}
 \DDer(A)  & \Der(A,\, \F\I\otimes A)\\*[1ex]
 0 & 0
\end{array}
\right) \bigoplus \left(
\begin{array}{cc}
 \F\I \otimes (\F\I)^\vee  &  \F\I \otimes A\\*[1ex]
 (\F\I)^\vee & A
\end{array}
\right)
\end{equation}
and the codomain with
\begin{equation*}
\left(
\begin{array}{cc}
 \F \otimes \DDer(A) \otimes \F^\vee & \Der(A, \I \otimes \F)\\*[1ex]
 0 & 0
\end{array}
\right) \bigoplus \left(
\begin{array}{cc}
 \F\I \otimes (\F\I)^\vee  &  \F\I \otimes A\\*[1ex]
 (\F\I)^\vee & A
\end{array}
\right)\ .
\end{equation*}
The first summand of \eqref{rai} is generated by the elements $\, \hd \in e\,\DDer(A)\,e \,$ (see Prop.~\ref{gene}): so we define the map \eqref{mdo} on this first summand by
\begin{equation}
\la{rinh}
\hd = \left(
\begin{array}{cc}
d  & 0 \\
0  & 0
\end{array}
\right)
\mapsto \left(
\begin{array}{cc}
\sum_{i} \alpha_{i} \otimes d \otimes \beta_i & 0 \\
0  & 0
\end{array}
\right)\ ,
\end{equation}
while letting it be the identity on the second. This
yields an isomorphism of bimodules and induces the required
algebra map $ \bvp_0 $. The commutativity
$\,\btheta\,\tilde{\psi}_0 = \btheta\,\psi_0\,$
is verified by an easy calculation, using \eqref{bth}.

To finish the proof of Proposition~\ref{Lex} it remains to show
the uniqueness of  $ \bvp $. For this, arguing as in Proposition~\ref{piext},
it suffices to show that $\, \HH_1(A[\J], \, \Ker\,\btheta^\otimes) = 0\,$, where
$\,\btheta^\otimes := 1 \otimes \btheta \otimes 1\,$, see \eqref{lso1}. Since
$\, A[\J] \cong \tilde{\F} \otimes_{\tA} B \otimes_{\tA}
\tilde{\F}^\vee \,$, see \eqref{lso}, we may identify $\,
\HH_1(A[\J], \, \Ker\,\btheta^\otimes) \cong \HH_1(B,
\,\Ker\,\btheta)\,$. On the other hand, by Lemma~\ref{Lrec1},
$\,\Ker\,\btheta \cong \Pi^{\blambda}(B)\ei \otimes_{U} \ei
\Pi^{\blambda}(B)\,$, which is easily seen to be a flat
$B$-bimodule. Thus $ \HH_1(B, \,\Ker\,\btheta) = 0 $, as required. \end{proof}
\begin{proof}[Proof of Proposition~\ref{omeac}]$(1)$ We will keep
the notation of Proposition~\ref{Lex}. For $\,\P
=\D_{\hsigma_\omega}\,$, we have then $\, \L \cong A \,$,
$\,\varphi = \hsigma_{\omega}\,$, $\, \tau = \id_A \,$ and $\,
\psi = \hsigma_{\omega}^{-1} \,$. Now, since $\, \F = \L^\tau
\cong A \,$, we may choose $\, \psi_0 =\id_{\D} $. Then $\, \bvp =
\sigma_{\omega}^{-1} \,$, and the
bimodule $\,\bP \,$ is isomorphic to $\, \Pi^\blambda(B)\,$ with
left multiplication twisted by $\,\sigma_{\omega}^{-1} \,$. Hence,
for $\, \P=\D_{\hsigma_\omega} \,$, the isomorphism \eqref{eigg}
is given by $\,[\bV] \mapsto
[\bV^{\sigma_{\omega}^{-1}}]\,$, which agrees with our definition of $\, \sigma^*_\omega \,$, see \eqref{act1}.

$(2)$ For  $\, \P = (\D \L)_{\vp} \,$, the map
$\,f_{\P}:\,\CC_n(X, \I) \to \CC_n(X, \J) \,$ is equivariant under
$ \Lambda $ in the sense that
\begin{equation}
\la{eql}
f_{\P} \circ \sigma_{\omega}^* =
\sigma_{\omega_{\tau}}^* \circ f_{\P}\ ,\quad \forall\, u \in \Lambda\ ,
\end{equation}
where $\,\omega = \dlog(u) \,$ and $\, \omega_{\tau}
= \dlog[\tau(u)]\,$. Indeed,
$\,f_{\P} \circ \sigma_{\omega}^* \,$ is induced by tensoring $
\Pi$-modules with the bimodule $\,{}_{\psi}\bP_{\sigma_{\omega}} =
{}_{\psi}\bP \otimes_\Pi \Pi_{\sigma_{\omega}}\,$, on which $\,
 \Pi^\blambda(A[\J]) \,$ acts on the left  via $ \bvp $. Since $\,\ttau\,\sigma_{\omega} =
\sigma_{\omega_{\tau}}\,\ttau \,$, we have
$\,
{}_{\psi}\bP_{\sigma_{\omega}} \cong {}_{\sigma^{-1}_{\omega}
\psi}\bP \cong {}_{\psi \sigma^{-1}_{\omega_{\tau}}}\bP \cong
(\Pi_{\sigma_{\omega_{\tau}}}') \otimes_{\Pi'} ({}_{\psi}\bP)\,$, where $ \Pi' := \Pi^\blambda(A[\J]) $. This implies
\eqref{eql}. Now, it follows from \eqref{eql} that $ f_{\P} $
induces a well-defined map $ \bar{f}_{\P} $ on the quotient
varieties. The map $ \bar{f}_{\P} $ depends only on the class $\, [\P] \in \Pic(\D) \,$, since $ [\P] $ determines $ \vp $ (and
$ \psi = \vp^{-1} $) up to an {\it inner} automorphism of
$\,\D \,$. By Prop.~\ref{Lex}, this means that $ \bvp $ (and
hence, $\,f_{\P}\,$) are determined by $ [\P] $  up to an
automorphism $\, \sigma_{\omega} \in \Aut_S[\,\Pi'\,] \,$ with $\,
\omega = \dlog(u)$, $\, u \in \Lambda \,$. Since
such automorphisms act trivially on $ \overline{\CC}_n(X, \I) $,
the map $ \bar{f}_{\P} $ is uniquely determined by $ [\P]
\in \Pic(\D) $.
\eproof

Finally, we prove the last part of Theorem~\ref{Tmain}.

\bproof[Proof of Theorem~\ref{Tmain}$(c)$] Let $ \bV $ be a
$\Pi^\blambda(B)$-module representing a point of $ \CC_n(X,\I) $.
The class $\,  \omega_n[\bV] \in \gamma^{-1}[\I] \,$
can then be represented by an ideal $ M $ fitting into the exact
sequence
\begin{equation}
\la{lsts}
0 \to M \to \bL \to \bV \to 0\ ,
\end{equation}
where $\,\bL = \Pi\ei \otimes_U \ei \bV \,$. Now, given an
invertible bimodule $ \P = (\D\L)_{\vp} $, we write  $ \Pi' =
\Pi^\blambda(A[\J]) $, $\,U' = \ei \Pi' \ei\,$ and observe that
$\,
\bP \otimes_{\Pi} (\Pi\ei \otimes_U \ei\Pi) \otimes_{\Pi} \bP^* \cong
\Pi' \ei \otimes_{U'} \ei\Pi' $,
where $\,\bP \,$ is the progenerator from $\,\Pi \,$ to $\,\Pi'
\,$ determined by $ \P $. On the other hand, we have
$$
{}_{\psi}\bP \otimes_{\Pi} \D \cong {}_{\psi}(\tilde{\F} \otimes_{\tA}
\Pi \otimes_{\Pi} \D) \cong {}_{\psi}(\tilde{\F} \otimes_{\D} \D) \cong
{}_{\psi}(\F\D) \cong (\D\L)_{\vp}= \P\ .
$$
Tensoring now \eqref{lsts} with $ \bP $ shows that the $\Pi'$-modules $\,\bV' := \bP \otimes_{\Pi} \bV \,$ and $ M' := \P \otimes_\D M $ fit into the exact sequence
$\,0 \to M' \to \bL' \to \bV' \to 0\,$,
with $\,\bL' = \Pi'\ei \otimes_{U'} \ei \bV' \,$. This means that
$\, [M'] \in \gamma^{-1}[\J] \,$ corresponds under $\, \omega_n
\,$ to $\, [\bV'] \in \CC_n(X, \J) \,$, verifying the
commutativity of \eqref{CMdim} and finishing the proof of
Theorem~\ref{Tmain} .
 \eproof
\section{Explicit Construction of Ideals. Examples}
\la{explicit}

\subsection{Distinguished representatives}

Given a rank $1$ torsion-free $ \D$-module $M$, we choose an
embedding $\,e :\,M \into Q \,$, where $ Q = \Frac(\D) $. Such an embedding is unique up to automorphism of $Q$. We will fix this automorphism at a later stage of our calculation.
Now, regarding $ M $ and $ Q $ as modules over $\,R=T_A
\DDer(A)\,$, we may try to extend $\,e\,$ to $ L $ through
$\,\eta: M \into L \,$. It is easy to see, however,
that such an extension does not exist in $ \Mod(R)$.
On the other hand, we have
\blemma \la{exta} There is a unique $A$-linear map $ e_L:\,L \to
Q\,$ extending $ e $ in $\Mod(A)$.
\elemma
\bproof Let $\,\eta_*:\,\Hom_A(L,\,Q) \to \Hom_A(M,\,Q) \,$ be the
restriction map. We have $\,\Ker(\eta_*) \cong \Hom_A(V,\,Q) =
0\,$, since $\,V\,$ is a torsion $A$-module, while $ Q $ is
torsion-free. On the other hand, $\,\Coker(\eta_*)\,$ is
isomorphic to a submodule of $\,\Ext_A^1(V,\,Q)\,$, while
$\,\Ext_A^1(V,\,Q) = 0\,$, since $ Q $ is an injective $A$-module.
It follows that $\eta_*$ is an isomorphism.
\eproof

Our aim is to compute $\, e_L \,$  explicitly, in terms of
representation $V$. First, we consider the map
\begin{equation} \la{dera} \ad:\ \Hom_A(L,\,Q) \to
\Der_A(R,\,\Hom(L, \,Q))\ ,
\end{equation}
taking $\,f: L \to Q \,$ to the inner derivation $\, \ad_f(r)(x)
:= r f(x) - f(rx) \,$, where $\, r \in R \,$ and $ x \in L $.
Since $\,\Ker(\ad) \cong \Hom_R(L,\,Q) = 0 \,$, the map
\eqref{dera} is injective, and every $\, f \in \Hom_A(L,\,Q) \,$
is uniquely determined by $\, \ad_f \,$. In addition, if $\, f \,$
restricts to an $R$-linear map $\, M \to Q \,$, then
$\,\eta_*(\ad_f) = 0 \,$ in $\,\Der_A(R,\,\Hom(M,\,Q))\,$, and
$\,\ad_f \,$ is determined by a (unique) derivation in
$\,\Der_A(R,\,\Hom(V,\,Q))\,$.
Thus $\, e_L $ is uniquely determined by $\,\delta_V
\in \Der_A(R,\,\Hom(V,\,Q)) \,$ satisfying
\begin{equation}
\la{fder}
e_L(r x) - r e_L(x) = \delta_V(r)[\pi(x)] \
,\quad \forall\,r \in R\ ,\quad \forall \,x \in L \ ,
\end{equation}
where $\,\pi:\,L \to V \,$. Furthermore, by the Leibniz rule,
the restriction map
$$
\Der_A(R,\,\Hom(V,\,Q)) \stackrel{\sim}{\to}
\Hom_{\eA}(\DDer(A),\,\Hom(V,\,Q))
$$
is an isomorphism: we thus need to compute $\,\delta_V \,$ on $ \DDer(A) $ only.

Let $\,\c(X \times X)^{\reg} \,$ be the subring of rational
functions on $\, X \times X \,$, regular outside the diagonal of $
X \times X $. Geometrically, we can think
of $\,\Omega^1(A) \subset \AA \,$ as the ideal of
the diagonal in $X\times X$, and $\, \Omega^1(A)^\star :=
\Hom_{A^{\otimes 2}}(\Omega^1A\,, A^{\otimes 2})$ as the subspace of functions in $\c(X\times X)^{\reg} $ with (at
most) simple poles along the diagonal; the canonical pairing
between $ \Omega^1 A $ and $ \Omega^1(A)^\star$ is the
given by multiplication in $ \O(X \times X) $. Translating this
into algebraic language, we have
\begin{lemma} \la{23}
Let $ \flat $ be the involution on $\,\c(X\times X)^{\reg}\,$ induced by interchanging the factors in $\,X\times X\,$.

$(1)$ The assignment $\,d \mapsto \left[\,d(a)/(a\otimes 1-1\otimes a)\,\right]^{\,\flat}\,$ defines an injective bimodule homomorphism
$\,\nu:\, \DDer(A) \to \c(X\times X)^{\reg}$.

$(2)$ If $\,a\in A \,$, $\, d\in \DDer(A)\,$ and
$\,d(a)=\sum_j f_j\otimes g_j\,$, then $\,
[d,\,a]=\sum_j g_j\,\Delta_A\, f_j\,$.
\end{lemma}

Now, to compute $\,\delta_V(d)\in \Hom(V,Q)\,$ we identify $ \Hom(V,Q) \cong Q \otimes V^* \,$. There is a
natural action of $ \, R^{\rm e} := R \otimes R^{\circ}\,$ on this
space: $\, R^{\rm e} \to  Q \otimes \End(V^*) \,$,
which is the tensor product of the dual representation
$\,\varrho^*:\, R^\circ \to \End(V^*) $ with composition of
natural maps $\, R \onto \Pi^1(A) \cong \D \into Q $. Abusing  notation, we will write $\, a \otimes b^* \,$ for the
image of $ a \otimes b^\circ \in R^{\rm e} $ in
$\, Q \otimes \End(V^*)\,$. Restricting to $ \AA \subset R^{\rm e} $, we now get a ring homomorphism
$\,\AA \to Q\otimes \End(V^*)\,$. Since $ \dim(V) < \infty $,
this homomorphism takes the elements $\,a\otimes 1-1\otimes
a\,$, with $\,a\in A\setminus \c \,$, to units in $ Q\otimes \End(V^*) $  and hence extends canonically to
\begin{equation*}
\c(X\times X)^{\mathrm{reg}} \to Q\otimes\End(V^*)\,.
\end{equation*}
Combining this last homomorphism with the embedding of Lemma~\ref{23}, we define a bimodule map
\begin{equation}
\la{nur}
\nu_V:\, \DDer(A) \to Q\otimes \End(V^*)\ ,\quad \Delta_A
\mapsto 1\otimes \id_{V^*}\ .
\end{equation}

We can now compute $\,\delta_V\,$ in terms of
$\nu_V$. To this end, we choose dual bases $\,\{\v_i\}\,$ and
$\{\w_i\}$ for $ \I $ and $ \I^\vee $; by Proposition~\ref{gene},
this gives generators $ \ha $, $ \hd $, $\hv_i$ and $ \hw_i $ for
$ \Pi $. Identifying $\,L_\infty \cong V_{\infty}
\cong \c\,$, we think of $\,\hv_i\,$ and $ \hw_i $ acting on
$\,\bL\,$ as linear maps $ \v_i:\,\c \to L$ and $\w_i:\, L \to
\c$, i.~e. as elements of $L$ and $L^*$. Similarly,
when acting on $\bV$, $\, \hv_i$ and $\hw_i$ give rise
to vectors $ \sv_i\in V$ and covectors $ \sw_i\in V^*$. Note that
$\sv_i=\pi \v_i$ and $ \sw_i \pi = \w_i$, where $\,\pi:\,L\onto
V\,$. Further, we fix  $\,
\I \into A \,$ and identify $ L $ as in Lemma~\ref{iso2}. Then we
twist $\, e:\,M \into Q \,$ by an automorphism of $Q$ in such a
way that $\, e_L(\v) = \v \,$ for all $ \v \in \I \subset A
\subset Q \,$. This is possible, since $\, e_L:\,L \to Q \,$ is an
$A$-linear extension of $ e $, by Lemma~\ref{exta}.
With this notation, we have
\begin{proposition} The derivation
$\,\delta_V:\, \DDer(A)\to Q\otimes V^*\,$ is given by
$\, \delta_V(d)= \,\sum_i \,\nu_V(d)[\v_i\,\sw_i]\,$.
\end{proposition}

\begin{proof} First, using the fact that $\Delta_A$ acts
as $\, 1 + \sum_i \v_i \w_i \,$ on $ L $ and as identity on $Q$,
it is easy to compute $\,\delta_V(\Delta_A)=\sum_i \v_i\sw_i \,$.
Now, if $\,r=[d,\,a] \in \DDer(A)\,$, then
$\, \delta_V(r)=[\delta_V(d), a]\,$, since $\,\delta_V(a)=0\,$. On
the other hand, by Lemma~\ref{23}$(2)$,
we have $\,[d,a]=\sum_j\, g_j \,\Delta_A\, f_j\,$, so
$\,\delta_V(r)=\sum_j \,g_j\,\delta_V(\Delta_A)\,f_j\,$. Thus,
$[\delta_V(d), a]=\sum_j\, g_j\,\delta_V(\Delta_A)\,f_j$, or, if
we think of $\delta_V(d)$ as an element of $Q\otimes V^*$, then
$\, (1\otimes a^*- a\otimes 1)\,\delta_V(d)=\left(\sum_j g_j\otimes
f_j^*\right)\, \delta_V(\Delta_A)\,$.
Lemma \ref{23}$(1)$ shows now that $\,\delta_V(d)=\nu_V(d)
\,[\delta_V(\Delta_A)]\,$.
\end{proof}

Now, we can state the main result of this section.
For $\,\v\in \I\,$ and $\,d \in \DDer(A)\,$, we define
\begin{equation}
\label{kappa} \kappa(d,\v) := \v - (1\otimes d^*  - d\otimes
1)^{-1}\, \delta_V(d)\,[1 \otimes \sv]\ \in Q\ ,
\end{equation}
where $\, \sv = \pi(\v) \in V \,$ and $\, (1\otimes d^*  -
d\otimes 1)^{-1} \in Q\otimes \End(V^*)  \,$.

\bthm \la{excon}
Let $\, \bV \,$ be a $\Pi^\blambda(B)$-module of
dimension $ \bn = (n,1) $ representing a point in $\,
\CC_n(X,\,\I) \,$. Then the class
$\,\omega[\bV] \in \R(\D) \,$ can be represented by the
(fractional) ideal $M$ generated by the elements
$\, \det{\!}_{V^*}(1\otimes a^* - a \otimes 1)\, \v \,$ and
$\,\det{\!}_{V^*}(1\otimes d^* - d \otimes 1)\, \kappa(d, \v)\,$,
where $\,a\in A\,$,\, $\,d \in \DDer(A) $ and $\v\in\I$.
\ethm
\noindent Theorem~\ref{excon} needs some explanations.

1. Formally, by \eqref{kappa}, $\, \kappa(d,\v) \,$ is well defined only when $\, 1\otimes d^* - d \otimes 1 \,$ is invertible in $\,
Q\otimes\End(V^*) \,$. It is easy to see, however, that the
product $\, \det{\!}_{V^*}(1\otimes d^*- d\otimes 1)\, \kappa(d,
\v) \in M \,$ makes sense for all $\,d \in \DDer(A) \,$
(cf. \cite{BC}, Remark~2, p.~83).

2. For generators of $ M $ it suffices to take the above determinants with $\, a \,$, $\, d \,$ and $\, \v \, $ from some (finite) sets generating $A$, $\, \DDer(A)$ and the ideal $ \I$.

\begin{proof}  By \eqref{iso}, the class $\,\omega(\bV)\,$ can be
represented by $\,\widetilde M = \Ker[\pi: \,L\to V]\,$. Our goal
is to show that the two kinds of determinants given in the
proposition generate $ M:=e_L(\widetilde M) $. To simplify the notation, we denote the elements of $ \I $ (resp., $\I^\vee$) and the corresponding elements of $V $ (resp., $ V^*$) by the same letter.
Using the Leibniz rule, for any $r\in R$ and $ m \ge 1 $,  we have
\begin{equation}\la{le}
\delta_V(r^m)=\left(\sum_{s=0}^{m-1} r^s\otimes
(r^*)^{m-s-1}\right)\delta_V(r)=\frac{1\otimes(r^*)^m- r^m\otimes
1}{1\otimes r^*- r\otimes 1}\,\delta_V(r)\  ,
\end{equation}
provided $\, 1\otimes r^*- r\otimes 1\in Q\otimes \End(V^*)\,$ is
invertible. Now, consider the characteristic polynomial
$p(t)=\chi_r(t):=\det{\!}_\rho(r-t\,\id_V)$ of $r\in R$ in the
representation $V$. It is clear that, for
any $x\in L$, $\,p(r)x$ lies in the kernel of $\pi:
L\onto V$, thus $p(r)x\in \widetilde M$. To compute its image
under $e_L$, we write
$\, e_L(p(r)x)=p(r)e_L(x)+ \delta_V(p(r))[1\otimes \overline x]\,$, where $\,\overline x=\pi(x)\,$.
Using \eqref{le} and the fact that $p(t)=\chi_r(t)$ annihilates $r^*
\in\End(V^*)$, we get
$\,\delta_V(p(r))= -(p(r)\otimes 1)(1\otimes r^* - r\otimes
1)^{-1}\delta_V(r)\,$. As a result, for $\,x=\v \in \I\,$,
\begin{equation}\la{rel}
e_L(\chi_r(r)\v)=\chi_r(r)\left(\v-( 1\otimes r^*-r\otimes 1
)^{-1}\delta_V(r)[1\otimes \sv]\right)\in M\,.
\end{equation}
Choosing different $r\in R$, we obtain in this way various
elements of $M$. In particular, for $r=a\in A$, we have
$\delta_V(a)=0$, so \eqref{rel} produces the elements
of the first kind $\chi_a(a)\v\in M$. On the other hand, taking $r=d\in\DDer(A)$ results in $\chi_d(d)\kappa(d,\v)$,
which are the elements of the second kind in $M$.

Finally, a simple filtration argument shows that the elements
$\chi_a(a)\v$ and $\chi_d(d)\v$, with $a$,$\,d$ and $\v$
running over some generating sets of $A$, $\,\DDer(A)$ and $\I $,
generate a submodule
$\widetilde N\subset \widetilde M$ of finite codimension in $L$.
Hence $\widetilde N = \widetilde M$, and the images of these elements generate thus $M=e_L(\widetilde M)$.
\end{proof}
\subsection{Examples}
\la{examples}
\subsubsection{The affine line}
\la{weyl0} Let $ X = \A^{\! 1} $. Choosing a global coordinate on
$X$, we identify $ A = \O(X)\cong \c[x] $. In this case,
$\,\DDer(A)\,$ is a free bimodule of rank $ 1 $; as a
generator of $\,\DDer(A)\,$, we may take the derivation $y$
defined by $\,y(x)=1\otimes 1\,$. It is easy to check that
$\,\Delta_A = yx-xy \,$ in $ \DDer(A)$. The algebra
$\,R = T_A \DDer(A)\,$ is isomorphic to the free algebra $
\c\langle x,\, y\rangle\,$, and $\,\Pi^{1}(A)\cong \c\langle x, y
\rangle/ \langle xy - yx + 1 \rangle \,$ is the Weyl algebra
$A_1(\c)$. The map $\nu$ of Lemma~\ref{23} is given by
\begin{equation*}
\la{nua} \nu(y)=(1\otimes x-x\otimes 1)^{-1}\,, \qquad
\nu(\Delta)=1\ .
\end{equation*}
All line bundles on $X$ are trivial, so we only need to
consider $B=A[\I]$ with $\I=A$. The $n$-th Calogero-Moser variety
$\, \CC_n := \CC_n(X, A) \,$ can be described as the space of
equivalence classes of matrices
$$
\{(\sX, \sY, \sv, \sw)\,:\, \sX \in \End(\c^n), \, \sY \in
\End(\c^n),\, \sv \in \Hom(\c, \c^n), \, \sw \in \Hom(\c^n, \c)\}\ ,
$$
satisfying the relation $\,\sY\sX- \sX\sY = \id_n+ \sv \sw\,$,
modulo the natural action of $ \GL_n(\c) $:
$$
(\sX,\, \sY,\, \sv, \,\sw)\ \mapsto\ (g\sX g^{-1}, \,g\sY g^{-1},\, g
\sv,\,\sw g^{-1}) ,
 \quad g \in \GL_n(\c)\ .
$$
If we choose $\v=1$ as a generator of $\I=A$, then the ideal $M$
of $\D\cong \Pi^1(A)$ corresponding to a point $(\sX, \sY, \sv,
\sw)$ is given by
\begin{equation*}\la{maf}
M = \D \cdot \det{\!}(\sX - x\,\id_n)\ +\ \D \cdot \det{\!}(\sY-
y\,\id_n)\, \kappa\ ,
\end{equation*}
where $\, \kappa=1-\sv^t(\sY^t-y\,\id_n)^{-1}(\sX^t-x\,\id_n)^{-1}\sw^t$.
This agrees with the description of ideals of $ A_1(\c) $ given in \cite{BC}.

\subsubsection{The complex torus}
\la{weyl} Let $ X = \c^* $. We identify $ A =\O(X)$ with $\c[x,
x^{-1}]$, the ring of Laurent polynomials. As in the affine line
case, $\DDer(A)$ is freely generated by the
derivation $y$ defined by $y(x)=1\otimes 1$. The algebra $R$ is isomorphic to the free product $ \c\langle
x^{\pm 1}, y \rangle := \c[x, x^{-1}] \star \c[y] $, and $\,
\Delta_A = yx-xy \,$ in $ R $. The matrix description of the
Calogero-Moser spaces $\CC_n$ and the formulas for the
corresponding fractional ideals of $\D\cong \Pi^{1}(A)=\c\langle
x^{\pm 1}, y \rangle/ \langle xy - yx + 1 \rangle \,$ are the same
as above, except for the fact that $x$ and $\sX$ are now
invertible. A new feature is that $A$ has nontrivial units
$x^r$, $r\in \Z$. The corresponding group $\Lambda$ can be
identified with $\Z$ and its action on $ \CC_n $ is given by
$$
r. (\sX, \,\sY,\, \sv, \,\sw) = (\sX,\,\sY + r\sX^{-1},\,\sv,\,\sw)\
, \quad r \in \Z  \ .
$$
Thus, by Theorem~\ref{Tmain}, the classes of ideals of $\,\D\cong
\Pi^1(A)\,$ are parameterized by the points of the quotient
variety $\,\overline{\CC}_n = \CC_n/\Z\,$. It is worth mentioning that one may choose a different generator
for the  bimodule $\DDer(A)$: for example, $\,z=yx,$
instead of $y$. Then $\,\Delta_A = z-xzx^{-1} $, which gives
an alternative matrix description of $\CC_n $ and the corresponding ideals.

\subsubsection{A general plane curve}
\la{genpc} Let $ X $ be a smooth curve in $ \c^2 $ defined by the
equation $\, F(x,y)=0\,$, with $\, F(x,y) :=
\sum_{r,s}a_{rs}x^ry^s \in \c[x, y]\,$. In this case, the algebra
$\, A \cong \c[x,y] / \langle F(x,y)\rangle\,$ is generated by $x$
and $y$ and the module $\,\Der(A)\,$ is (freely) generated by the
derivation $\,\partial\,$ defined by
$\,\partial(x)= F'_y(x,y)\,$,$\ \partial(y)= - F'_x(x,y)\,$.
The bimodule $\,\DDer(A)\,$ is generated  by the
derivation $ \Delta = \Delta_{A} $ and the element $\,z\,$ defined
by
\begin{equation*}
    z(x)=\sum_{r,s} a_{rs}\frac{x^ry^s\otimes 1- x^r\otimes
    y^{s}}{y\otimes 1-1\otimes y}\ ,\qquad
    z(y)=-\sum_{r,s} a_{rs}\frac{x^r\otimes y^s-1\otimes x^ry^s}{x\otimes 1-1\otimes x}\ .
\end{equation*}
These generators satisfy the following commutation relations
\begin{equation}
\la{rr2}
[z,\, x]=\sum_{r,s} a_{rs}\sum_{k=0}^{s-1}y^{s-k-1}\Delta y^kx^r
    \ ,\qquad
[z,\, y]=-\sum_{r,s}a_{rs}\sum_{l=0}^{r-1}y^sx^{r-l-1}\Delta x^l    \,.
\end{equation}
By Proposition~\ref{gene}, the algebra $\,\Pi^\blambda(B) \,$ is
then generated by the elements $\,\hat{x} $, $\,\hat{y} $,
$\,\hat{z} $, $\,\hv_i $, $\,\hw_i $ and $\, \hdel $, subject to
the relations \eqref{rr2} and \eqref{r1}. The
assignment $\,x\mapsto x$, $\,y\mapsto y$, $\,z\mapsto\partial$,
$\,\Delta\mapsto 1 $ extends to an isomorphism between $\Pi^1(A)$
and the ring $ \D $ of differential operators on $X$.
The bimodule map $\nu$ of Lemma~\ref{23} is given by
\begin{equation}
\la{nuplane} \nu(z)=-\frac{\sum_{r,s} a_{rs}y^s\otimes
x^r}{(1\otimes x-x\otimes 1)(1\otimes y-y\otimes 1)}\,, \qquad
\nu(\Delta)=1\,.
\end{equation}

Now, let us describe generic points of the varieties
$ \CC_n(X, \I) $; for simplicity, we consider only the case when
$ \I $ is trivial. Choose $\,n\,$ distinct points $\, p_i = (x_i,\,y_i)\in X\,$,
$\,i=1,\dots, n$,\, and define
\begin{equation}
\la{ldata}
(\sX,\,\sY,\,\sZ,\,\sv,\,\sw) \in \End(\c^n)\times \End(\c^n)
\times \End(\c^n)\times \Hom(\c, \c^n)\times \Hom(\c^n,\c)
\end{equation}
by the following formulas
\begin{equation}
\la{XYm} \sX =\mathrm{diag}(x_1,\dots,x_n)\ ,\ \sY
=\mathrm{diag}(y_1,\dots,y_n)\ ,\ \sv^t=-\sw =(1,\,\dots\,,\,1)\ ,
\end{equation}
\begin{equation*}
    \sZ_{ii}=\alpha_i\quad\text{and}\quad
    \sZ_{ij}= \frac{F(x_j,\, y_i)}{(x_i-x_j)(y_i-y_j)}
    \quad (\text{for}\ i\ne j)\ ,
\end{equation*}
where $\,\alpha_1,\dots ,\alpha_n \,$ are arbitrary scalars. Then,
a straightforward calculation, using the relations \eqref{rr2}, shows that the assignment
\begin{equation*}
\hat{x} \mapsto \sX\ ,\quad  \hat{y} \mapsto \sY\ ,\quad \hat{z}
\mapsto \sZ\ ,\quad \hv \mapsto \sv\ , \quad \hw \mapsto \sw\ ,
\quad \hdel \mapsto \id_n + \sv\,\sw
\end{equation*}
extends to a representation of $\,\Pi^\blambda(B)\,$, with $\,B=
A[A]\,$ and $ \blambda = (1, -n) $, on the vector space $\,\bV =
\c^n \oplus \c\,$.

\vspace{1ex}

\remark\ The matrix $ \sZ $ defined above is a generalization of
the classical {\it Moser matrix} in the theory of integrable
systems (see \cite{KKS}).

\vspace{1ex}

To illustrate Theorem~\ref{excon} we now describe the fractional
ideal representing the class $ \omega[\bV] $ for an arbitrary $
[\bV] \in \CC_n(X, \I)$. We consider first the case when $
\I $ is trivial. In that case, we identify $\, \I=\I^\vee = A $
and choose $\,\v = \w = 1\,$ as the generators of $\,\I\,$ and
$\,\I^\vee $. A representation $ \bV=\c^n \oplus \c $ may then be
described by the matrices \eqref{ldata}, which, apart from
\eqref{rr2}, satisfy the following relations
$$
F(\sX,\sY)=0\ ,\quad [\sX,\,\sY] = 0\quad \mbox{and}\quad \sdel
= \id_n+\sv\,\sw \ .
$$
The dual representation $\,\varrho^*:\, \Pi^\circ \to \End(\bV^*)
\,$ is given by the transposed matrices.

Now, \eqref{kappa} together with \eqref{nuplane} show that
$\,\kappa = \kappa(z, 1)\in Q \,$ is given by
\begin{equation*}
\la{kappz}
\kappa=1+\sv^t\,(\sZ^t-z\,\id_n)^{-1}(\sX^t-x\,\id_n)^{-1}(\sY^t-y\,\id_n)^{-1}
 F(\sX^t,\, y\,\id_n)\, \sw^t \ .
\end{equation*}
Thus, if $\, [\bV] \in \CC_n(X, A)\,$ is determined by the data
$\,(\sX,\,\sY,\sZ,\,\sv,\,\sw)\,$, then the class
$\,\omega[\bV] \,$ is represented by the (fractional) ideal
\begin{equation*}\la{pca}
M = \D\cdot\det(\sX-x\,\id_n) + \D \cdot \det(\sY-y\,\id_n) +
\D\cdot \det(\sZ-z\,\id_n)\,\kappa\ .
\end{equation*}
In the general case, when $ \I $ is arbitrary, $\,\kappa \,$ is
replaced by
\begin{equation}\la{kappag}
\kappa(\v)=\v +\sum_i\,\left(\,
\sv^t(\sZ^t-z\,\id_n)^{-1}(\sX^t-x\,\id_n)^{-1}(\sY^t-y\,\id_n)^{-1}
F(\sX^t,\, y\,\id_n)\,\sw_i^t\,\right)\,\v_i\,,
\end{equation}
and the corresponding  class $\, \omega[\bV] \in \gamma^{-1}[\I]
\,$ is given by
\begin{equation}\la{pci}
M = \sum_i [\,\D \cdot \det(\sX-x\,\id_n)\,\v_i + \D \cdot
\det(\sY-y\,\id_n)\,\v_i + \D \cdot \det(\sZ-
z\,\id_n)\,\kappa(\v_i)\,]\ .
\end{equation}
%

%
\subsubsection{A hyperelliptic curve} 
This is a special plane curve described by the equation $\,y^2 = P(x)\,$, where $\,P(x)= \sum_s a_s x^s \,$ 
is a polynomial with simple roots. Some of the above formulas simplify in this case.
We have $A \cong \c[x,y] / \langle y^2- P(x) \rangle\,$, $\,\Der(A)\,$ is freely generated by
$\,\partial\,$, with $\partial(x)= 2y$ and $\partial(y)= P'(x)$, and
the bimodule $\,\DDer(A)\,$ is generated by $ \Delta $ and the element $\,z\,$ defined by
\begin{gather*}
    z(x)=y\otimes 1+1\otimes y\,,\qquad
    z(y)=(P(x)\otimes 1-1\otimes P(x))/(x\otimes 1-1\otimes x)\ .
\end{gather*}
The commutation relations \eqref{rr2} in $ \DDer(A) $ are 
\begin{gather}
\la{rrr2}
    [z,\, x]=y\Delta +\Delta y
    \,,\qquad
    [z,\, y]=\sum_{s} a_{s}\sum_{l=0}^{s-1}x^{s-l-1}\Delta x^l
    \,.
\end{gather}

Now, for a hyperelliptic curve, a point of $\CC_n(X,\I)$ is determined by the following data: $\,(1)$ a representation of $A$ on the vector space $\,V=\c^n \,$, i.e. a pair of matrices
$\,(\sX, \sY) \in \End(\c^n) \times \End(\c^n) $ satisfying 
$\,\sY^2=P(\sX)$; $\,(2)$ a pair of $A$-module maps $\,\I\to V\,$ 
and $\,\I^\vee\to V^*$, with chosen images 
$\sv_i\in V$ and $\sw_i\in V^*$ of dual bases of $ \I $ and $ \I^\vee $; $\,(3)$ a matrix $\,\sZ\in \End(\c^n)\,$, such that
$\,\sX,\,\sY,\,\sZ\,$ and $\,\sdel := \id_n + \sum_i \sv_i\sw_i\,$
satisfy \eqref{rrr2}.  In this case, formula \eqref{kappag} reads
\begin{equation*}
\kappa(\v)= \v-\sum_i\,\left(\,
\sv^t(\sZ^t-z\,\id)^{-1}(\sX^t-x\,\id)^{-1}(\sY^t+y\,\id)\,
\sw_i^t\,\right)\,\v_i\ ,
\end{equation*}
and the corresponding ideal is given by \eqref{pci}.

\appendix
\section{Half-Forms on Riemann Surfaces}
\section*{\bf George Wilson}
In this note I provide a proof for one of the key facts
(Proposition~\ref{main} below) needed to understand the
relationship between deformed preprojective algebras and rings of
differential operators. The note owes a great deal to
conversations with Graeme Segal.

\subsection*{Statement of problem}

Let $\,X \,$ be a compact Riemann surface, and let $\, \Delta \,$
be the diagonal divisor in $\, X \times X \,$.  We have the
inclusion
\begin{equation*}
\label{inc}
\O_{X \times X}(- \Delta) \hookrightarrow \O_{X \times X}(\Delta)
\end{equation*}
of the sheaf of functions that vanish on $\, \Delta \,$ into the
sheaf of functions that are allowed a simple pole on $\, \Delta
\,$. The quotient sheaf $\, \O_{X \times X}(\Delta) / \O_{X \times
X}(- \Delta)$ is supported on the first infinitesimal
neighborhood $\, \Delta_1 \,$ of $\, \Delta \,$.  Similarly, if
$\, \L \,$ is a line bundle on $\, X \,$, we have the sheaf $\,
\D_1(\L) \,$ of differential operators of order $\, \leq 1 \,$ on
$\, \L \,$.  This is usually regarded as a sheaf on $\, X \,$, but
since we can compose a differential operator with a function
either on the left or on the right, it has two commuting
structures of $\, \O_X$-module, so it too can be regarded as a
sheaf on $\, X \times X \,$, again supported on $\, \Delta_1 \,$.

Fix a square root $\, \Omega^{1/2} \,$ of the canonical bundle
$\,\Omega_X \,$; the choice of square root will be immaterial,
because the corresponding sheaves of differential operators $\,
\D(\Omega^{1/2}) \,$ are canonically isomorphic to each other.
Our aim is to understand the following fact stated in \cite{G}.
\begin{proposition}
\label{main}
There is a canonical isomorphism (of sheaves over $\, X \times X
\,$)
$$
\chi \,:\, \O_{X \times X}(\Delta) / \O_{X \times X}(- \Delta)
\to \D_1(\Omega^{1/2}) \ .
$$
\end{proposition}

A consequence is that the sheaf of deformed preprojective algebras
formed from $\, \O_X \,$ is canonically isomorphic to the the
sheaf $\, \D(\Omega^{1/2}) \,$ of differential operators on $\,
\Omega^{1/2}  \,$. This is explained in \cite{G}, Section 13.

The isomorphism in Proposition~\ref{main} does not seem to be a
well-known fact, and at first sight looks puzzling, because there
are no half-forms in the left hand side.  The proof sketched in
the current version of \cite{G} is not very convincing, so it
seems worth recording the following simple explanation shown to me
by Segal: although Proposition~\ref{main} itself does not look
familiar, it can be obtained by combining two familiar facts, of a
slightly different nature. While we are about it, we shall deal
also with a slight generalization, twisting by an arbitrary line
bundle $\, \L \,$ on $\, X \,$.

We use the following notation: $\, \Delta_n \,$ is the $n$th
infinitesimal neighborhood of the diagonal in $\, X \times X \,$,
so that we have a canonical identification
\begin{equation}
\label{ln}
\L \,/\, \L(-(n+1) \Delta) \; \simeq \; \L \,\vert\, \Delta_n \ .
\end{equation}
The two projections  $\, X \times X \to X \,$ are denoted by $\,
p_1 \,$ and $\, p_2 \,$. If $\, U \,$ is a simply-connected
coordinate patch on $\, X \,$ and $\, z \,$ is a parameter on $\,
U \,$, we write $\, (z_1, z_2) \,$ for the induced parameters on
$\, U \times U \subset X \times X \,$. The parameter $\, z \,$
determines a trivialization (non-vanishing section) $\, dz \,$ of
$\, \Omega_X \,\vert\, U \,$. Fixing also an
isomorphism\footnote{Of course $\kappa$ is uniquely determined up
to a constant multiple. The isomorphism $\chi$ in
Proposition~\ref{main} does not depend on this multiple, but some
of the intermediate steps below do.} $\, \kappa :
(\Omega^{1/2})^{\otimes 2} \simeq \Omega_X \,$, we may choose a
trivialization $\, dz^{1/2} \,$ of $\,  \Omega^{1/2} \,\vert\, U
\,$ such that $\, \kappa(dz^{1/2}\otimes dz^{1/2}) = dz \,$ (there
are only two choices, differing by a sign).

\subsection*{A proof of Proposition~\ref{main}}
We shall use the following description of differential operators, which
goes back to Cauchy (see \cite{C}, p.~60, formule (4)).
\begin{proposition}
Let $\, \L \,$ be a line bundle on $\, X \,$. Then there is a canonical identification 
(of sheaves over $\, X \times X \,$)
\begin{equation}
\label{difff}
p_1^*(\L) \otimes p_2^*(\L^* \otimes \Omega_X)\,((n+1) \Delta)
\,\vert\, \Delta_n \, \simeq \, \D_n(\L) \,.
\end{equation}
\end{proposition}
\begin{proof}
The action of a (local) section of the sheaf on the left of
\eqref{difff} on a section of $\, \L \,$ is given by contracting
with the factor $\, p_2^*(\L^*) \,$ and then taking the residue on
the diagonal of the resulting differential. Let us spell that out
in more detail in the case where $\, \L \,$ is the trivial bundle
and $\, n=1 \,$.  The sheaf on the left of \eqref{difff} is then
just $\, p_2^*(\Omega_X)(2\Delta) \,\vert\, \Delta_1 \,=\,
p_2^*(\Omega_X)(2\Delta) \,/\, p_2^*(\Omega_X) \,$. In terms of a
parameter $\, z \,$, a local section of this sheaf has the form
\begin{equation*}
\label{phi}
\frac{\varphi(z_1, z_2) \, dz_2}{(z_2 - z_1)^2}  \,
\text{ \ modulo regular terms}
\end{equation*}
(where $\, \varphi \,$ is regular). To see how this  acts on a
function $\, f(z) \,$, we have to calculate the residue
$$
\res_{z_2 = z_1}
\frac{f(z_2) \varphi(z_1, z_2) dz_2}{(z_2 - z_1)^2}
$$
($z_1$ is held fixed during the calculation). Expanding
$$
f(z_2) = f(z_1) + f'(z_1)(z_2 - z_1) + \;\ldots \ ,
$$
and
$$
\frac{\varphi(z_1, z_2)}{(z_2 - z_1)^2} =
\frac{a(z_1)}{(z_2 - z_1)^2} + \frac{b(z_1)}{z_2 - z_1} + \;\ldots \ ,
$$
we find that the residue is
$$
a(z) \frac{df}{dz} + b(z) f \Bigr\vert_{z=z_1}  \; .
$$
The proposition is now clear.
\end{proof}

Now let $\, U \,$ be a coordinate patch on $\, X \,$.  We consider
the classical\footnote{it is the principal part of the {\it
Szeg\"o kernel} on $\, X \times X \,$.} differential $\, \gamma
\,$ given in terms of a parameter $\, z \,$ by
\begin{equation}
\label{gamma}
\gamma \, :=\, \frac{dz_1^{1/2} dz_2^{1/2}}{z_1 - z_2} \ .
\end{equation}
It is a non-vanishing section (over $\, U \times U \,$) of the
line bundle
$$
p_1^*(\Omega^{1/2}) \otimes p_2^*(\Omega^{1/2})(\Delta)\;.
$$
It depends on the parameter $\, z \,$; however, its restriction to
$\, \Delta \,$  does not. Indeed, when we identify $\, \O_{X
\times X}(- \Delta) \,\vert\, \Delta \,$ with the canonical bundle
on the diagonal, $\, z_1 - z_2 \,$ corresponds to $\, dz \,$, so
$\, \gamma \,\vert\, \Delta \,$ becomes the constant section $1
\in \O(U) \,$. Furthermore, because $\, \gamma \,$ is skew in the
two variables, its restriction to $\, \Delta_1 \,$ is also
independent of the choice of $\, z \,$.  Thus for any sheaf $\,
\mathcal{M} \,$ over $\, X \times X \,$, multiplication by  $\,
\gamma \,$ gives a well-defined global isomorphism
$$
\mathcal{M} \,\vert\, \Delta_1 \,\simeq\,
\mathcal{M} \otimes p_1^*(\Omega^{1/2}) \otimes p_2^*(\Omega^{1/2})(\Delta)
\,\vert\, \Delta_1 \ .
$$
In particular, for any line bundle $\, \L \,$ over $\, X \,$, we
get an isomorphism
\begin{equation}
\label{is}
p_1^*(\L) \otimes p_2^*(\L^*)(\Delta) \,\vert\, \Delta_1 \,\simeq \,
p_1^*(\L \otimes \Omega^{1/2}) \otimes p_2^*(\L^* \otimes
\Omega^{1/2})(2 \Delta) \,\vert\, \Delta_1 \; .
\end{equation}
Tensoring our chosen isomorphism $\, \kappa :
(\Omega^{1/2})^{\otimes 2} \simeq \Omega_X \,$ with $\,
(\Omega^{1/2})^* \,$, we get an isomorphism $\, \Omega^{1/2} \,
\simeq \, (\Omega^{1/2})^* \otimes \Omega_X \,$, and hence for any
$\, \L \,$ an isomorphism
$$
\L^* \otimes \Omega^{1/2} \, \simeq \,
(\L \otimes \Omega^{1/2})^* \otimes \Omega_X \,.
$$
Inserting this into \eqref{is} and taking account of \eqref{ln}
and \eqref{difff} gives us an isomorphism (now independent of $\,
\kappa \,$)
\begin{equation}
\label{lis}
p_1^*(\L) \otimes p_2^*(\L^*)(\Delta) \,/ \,
p_1^*(\L) \otimes p_2^*(\L^*)(- \Delta) \,\simeq \,
\D_1(\L \otimes \Omega^{1/2}) \,.
\end{equation}
Taking $\, \L = \O_X \,$, we get Proposition~\ref{main}.
\subsection*{Remarks}
1.  Reversing the arguments in \cite{G}, we easily get from
\eqref{lis} a construction of any $\, \D(\L) \,$ as a sheaf of
`twisted deformed preprojective algebras'.\\

\noindent
2. The differential $\, \gamma \,$ in \eqref{gamma} is invariant
under a {\it linear fractional} change of parameter.  Thus if we
fix a projective structure on $\, X \,$ (thought of as an atlas
with linear fractional transition functions), then $\, \gamma \,$
is well-defined globally on some analytic neighbourhood of $\,
\Delta \,$, not merely on $\, \Delta_1 \,$.  This remark is the
starting point for the papers \cite{BR}.\\

\noindent
3.  The considerations above give an explicit formula for the
isomorphism  $\, \chi \,$ in Proposition~\ref{main}: an element of
$\, \O_{X \times X}(\Delta) / \O_{X \times X}(- \Delta) \,$ has a
unique local representative of the form
\begin{equation*}
\label{rep}
a(z_1)(z_2 - z_1)^{-1} + \,b(z_1) \;+ \;\ldots,
\end{equation*}
and $\, \chi \,$  maps this to the operator
\begin{equation*}
\label{def}
f\,dz^{1/2} \, \mapsto \, \Bigl( a(z) \frac{df}{dz} \, + \, b(z) f \Bigr)
\,dz^{1/2} \ .
\end{equation*}
In an earlier version of this note I verified
Proposition~\ref{main} by checking directly that the map $\, \chi
\,$ defined by this formula is independent of the chosen parameter
$\, z \,$; however,  the calculation is surprisingly complicated
(and unilluminating).

\bibliographystyle{amsalpha}

\begin{thebibliography}{A}
%
\bibitem[AMM]{AMM}
H.~Airault, H. P. McKean and J. Moser, \textit{Rational and elliptic solutions of the 
Korteweg-de Vries equation and a related many-body problem}, Commun. Pure Appl. Math.
\textbf{30} (1977), 95--148.
%
\bibitem[AZ]{AZ} M.~Artin and J.~Zhang,
\textit{Noncommutative projective schemes}, Adv. Math.
\textbf{109} (1994), 228--287.
%
\bibitem[BGK1]{BGK1}
V.~Baranovsky, V.~Ginzburg and A.~Kuznetsov, \textit{Quiver varieties 
and a noncommutative $ {\mathbb P}^2 $}, Compositio Math. \textbf{134} 
(2002), 283--318.
%
\bibitem[BGK2]{BGK2}
V.~Baranovsky, V.~Ginzburg and A.~Kuznetsov, \textit{Wilson's
Grassmannian and a noncommutative quadric}, Int. Math. Res. Not.
\textbf{21} (2003), 1155--1197.
%
\bibitem[Ba]{B}
H.~Bass, \textit{Algebraic K-theory}, W.~A.~Benjamin Inc.,
New York-Amsterdam, 1968.
%
\bibitem[BBD]{BBD}
A.~A.~Beilinson, J.~Bernstein and P.~Deligne, \textit{Faisceaux
pervers}, Ast\'erisque \textbf{100}, Soc. Math. France, 1982, pp.
5--171.
%
\bibitem[Be]{Ben}
B.~Bendiffalah, \textit{Modules d'extensions des alg\`ebres triangulaires},
C. R. Acad. Sci. Paris Ser. I \textbf{339} (2004), 387--390.
%
\bibitem[BN]{BN}
D.~Ben-Zvi and T.~Nevins, \textit{Perverse bundles and
Calogero-Moser spaces}, Compos. Math. \textbf{144}(6)
(2008), 1403--1428.
%
\bibitem[BN1]{BN1}
D.~Ben-Zvi and T.~Nevins, \textit{From solitons to many-body systems},
Pure Appl. Math. Q. \textbf{4}(2) (2008), 319--361.
%
\bibitem[B]{Be}
Yu.~Berest, \textit{Calogero-Moser spaces over algebraic curves},
Selecta Math. (N.S.) \textbf{14}(3) (2009), 373--396.
%
\bibitem[BC]{BC}
Yu.~Berest and O.~Chalykh, \textit{${\sf A}_{\infty}$-modules and
Calogero-Moser spaces},  J. reine angew. Math. \textbf{607}
(2007), 69--112.
%
\bibitem[BC1]{BC1}
Yu.~Berest and O.~Chalykh, \textit{A note on the Dunkl representation 
of Cherednik algebras on algebraic curves}, in preparation.
%
\bibitem[BCE]{BCE}
Yu.~Berest, O.~Chalykh and F.~Eshmatov, \textit{Recollement of
deformed preprojective algebras and the Calogero-Moser
correspondence}, Moscow Math. J. \textbf{8}(1) (2008), 21--37.
%
\bibitem[BW]{BW}
Yu. Berest and G. Wilson, \textit{Differential operators on an
affine curve: ideal classes and Picard groups}, preprint, 
{\tt arXiv:0810.0223}, Quart. J. Math. Oxford Ser.(2)  (to appear).
%
\bibitem[BW1]{BW1}
Yu. Berest and G. Wilson,
\textit{Automorphisms and ideals of the Weyl algebra},
Math. Ann. \textbf{318}(1) (2000), 127--147.
%
\bibitem[BW2]{BW2}
Yu. Berest and G. Wilson,
\textit{Ideal classes of the Weyl algebra and noncommutative
projective geometry}
(with an Appendix by M. van den Bergh),
Internat. Math. Res. Notices \textbf{26} (2002), 1347--1396.
%
%
\bibitem[BR]{BR}
I.\ Biswas and A.\ K.\ Raina,  \textit{Projective structures
on a Riemann surface}.
I. Internat.\ Math,\ Res.\ Notices 1996, 753--768.
II. Ibid.\ 1999, 685--716.
III. Differential Geom.\ Appl.\ \textbf{15} 2001, 203--219.
%
\bibitem[Bj]{Bj}
J.-E. Bj\"ork, \textit{Rings of Differential Operators},
North-Holland Publishing, Amsterdam, 1979.
%
\bibitem[CH]{CH}
R. C. Cannings and M. P. Holland, \textit{Right ideals of rings
of differential operators},
J. Algebra \textbf{167} (1994), 116--141.
%
\bibitem[CH1]{CH1}
R. C. Cannings and M. P. Holland, \textit{Etale covers, bimodules
and differential operators}, Math.~Z.~\textbf{216} (1994),
179--194.
%
\bibitem[CE]{CE}
H.~Cartan and S.~Eilenberg, \textit{Homological Algebra},
Princeton University Press, Princeton, 1956.
%
\bibitem[CC]{CC}
D.~V. Chudnovsky and G. V. Chudnovsky, \textit{Pole expansion of nonlinear partial 
differential equations}, Il Nuovo Cimento \textbf{40B} (1977), 339--353.
%
\bibitem[C]{C}
A.~L.~Cauchy, \textit{Oeuvres} (2) \textbf{12}, Gauthier-Villars, Paris, 1882.
%
\bibitem[CPS]{CPS}
E.~Cline, B.~Parshall and L.~Scott, \textit{Algebraic
stratification in representation categories}, J. Algebra
\textbf{117} (1988), 504--521.
%
\bibitem[CB]{CB}
W.~Crawley-Boevey, \textit{Preprojective algebras, differential operators and a
Conze embedding for deformations of Kleinian singularities},
Comment. Math. Helv.  \textbf{74}(4) (1999), 548--574.
%
\bibitem[CB1]{CB1}
W.~Crawley-Boevey, \textit{Representations of quivers,
preprojective algebras and deformations of quotient
singularities}, lectures at the workshop on `Quantizations of
Kleinian singularities', Oberwolfach, May 1999.
%
\bibitem[CB2]{CB2}
W.~Crawley-Boevey, \textit{Geometry of the moment map for
representations of quivers}, Compositio Math. \textbf{126} (2001),
257--293.
%
\bibitem[CB3]{CB3}
W.~Crawley-Boevey, \textit{Geometry of representations of
algebras}, lecture notes, Oxford, 1993.
%
\bibitem[CEG]{CEG}
W.~Crawley-Boevey, P.~Etingof and V.~Ginzburg,
\textit{Noncommutative geometry and quiver algebras}, Adv. Math.
\textbf{209} (2007), 274--336.
%
\bibitem[CBH]{CBH}
W.~Crawley-Boevey and M.Holland, \textit{Noncommutative
deformations of Kleinian singularities}, Duke Math. J. \textbf{92}
(1998), 605--636.
%
\bibitem[CQ]{CQ}
J.~Cuntz and D. Quillen, \textit{Algebra extensions and nonsingularity},
J. Amer. Math. Soc. \textbf{8}(2) (1995), 251--289.
%
\bibitem[E]{E}
P.~Etingof, \textit{Cherednik and Hecke algebras of varieties with
a finite group action}, preprint, {\tt arXiv:math.QA/0406499}.
%
\bibitem[EG]{EG}
P.~Etingof and V.~Ginzburg, \textit{Symplectic reflection
algebras, Calogero-Moser space, and deformed Harish-Chandra
homomorphism}, Invent. Math. \textbf{147} (2002), 243--348.
%
\bibitem[FG]{FG}
M.~Finkelberg and V.~Ginzburg, \textit{Cherednik algebras for algebraic curves}, 
preprint, {\tt arXiv:0704.3494}.
%
\bibitem[G]{G}
V.~Ginzburg, \textit{Lectures on Noncommutative geometry},
preprint, {\tt arXiv:\!\! math.AG/0506603}.
%
\bibitem[KKO]{KKO}
A.~Kapustin, A.Kuznetsov, and D.~Orlov, \textit{Noncommutative
instantons and twistor transform}, Comm. Math. Phys. \textbf{220}
(2001), 385--432.
%
%
\bibitem[KKS]{KKS}
D.~Kazhdan, B.~Kostant and S.~Sternberg, \textit{Hamiltonian group
actions and dynamical systems of Calogero type}, Comm. Pure Appl.
Math. \textbf{31} (1978), 481--507
%
\bibitem[Ko]{Ko}
S.~K\"onig, \textit{Tilting complexes, perpendicular categories
and recollements of derived module categories of rings}, J. Pure
and Appl. Algebra \textbf{73} (1991), 211--232.
%
\bibitem[Kon]{Kon}
M.~Kontsevich, \textit{XI Solomon Lefschetz Memorial Lecture series:
Hodge structures in non-commutative geometry}, Contemp. Math. \textbf{462} (2008), 1--21.
%
\bibitem[KR]{KR}
M.~Kontsevich and A.~Rosenberg, \textit{Noncommutative smooth
spaces}, The Gelfand Mathematical Seminars 1996-1999, 85--108,
Birkh\"auser, Boston, 2000.
%
\bibitem[K]{K}
H.~Kraft, \textit{Geometrische Methoden in der Invariantentheorie},
Aspects of Mathematics, Vieweg \& Sohn, Braunschweig, 1984.
%
\bibitem[Kr]{Kr}
I. M. Krichever, \textit{ On rational solutions of the Kadomtsev-Petviashvili equation
and in tegrable systems of $ N $ particles on the line}, Funct. Anal. Appl. \textbf{12:1} 
(1978), 76--78 (Russian), 59--61 (English).
%
\bibitem[L]{L}
T.~Levasseur, \textit{Some properties of non-commutative regular
graded rings}, Glasgow Math. J. \textbf{34} (1992), 277--300.
%
\bibitem[LeB]{LeB}
L. Le Bruyn, \textit{Moduli spaces of right ideals of the Weyl
algebra}, J. Algebra \textbf{172}, 32--48 (1995).
%
\bibitem[Lo]{Lo}
J.-L.~Loday, \textit{Cyclic Homology}, Springer-Verlag
Springer-Verlag, Berlin-New York, 1992.
%
\bibitem[Lu]{Lu}
D.~Luna, \textit{Slices \'etales}, Bull. Soc. Math. France
\textbf{33} (1973), 81--105.
%
%
\bibitem[NB]{NvdB}
K. de Naeghel and M. Van den Bergh, \textit{Ideals classes of
three-dimensional Sklyanin algebras}, J.~Algebra \textbf{276} (2)
(2004), 515--551.
%
\bibitem[NS]{NS}
T.~A.~Nevins and J.~T.~Stafford, \textit{Sklyanin algebras and
Hilbert schemes of points}, Adv. Math. \textbf{210}(2) (2007),
405--478.
%
\bibitem[R]{R}
C.~M.~Ringel, \textit{Tame Algebras and Integral Quadratic Forms},
Lecture Notes in Math \textbf{1099}, 1984.
%
\bibitem[Ro]{Ro}
J.~Rotman, \textit{An Introduction to Homological Algebra},
Academic press, New York, 1979.
%
\bibitem[SS]{SS}
S. P. Smith and J. T. Stafford, \textit{Differential
operators on an affine curve}, Proc. London Math. Soc. (3) \textbf{56} (1988), 229--259.
%
\bibitem[S]{S}
J.-P. Serre, \textit{Local Algebra}, Springer-Verlag,
Berlin-Heidelberg-New York, 2000.
%
\bibitem[St]{St}
J.~T.~Stafford, \textit{Endomorphisms of right ideals of the Weyl
algebra}, Trans. Amer. Math. Soc. \textbf{299} (1987), 623--639.
%
\bibitem[vdB]{vdB}
M.~Van den Bergh, \textit{Double Poisson algebras}, Trans. Amer.
Math. Soc. \textbf{360} (2008), 5711--5769.
%
\bibitem[W]{Wi}
G. Wilson, \textit{Collisions of Calogero-Moser particles and an
adelic Grassmannian} (with an Appendix by I. G. Macdonald),
Invent. Math. \textbf{133} (1998), 1--41.
%
\bibitem[W1]{W1}
G. Wilson,
\textit{Bispectral commutative ordinary differential
operators}, J. Reine Angew. Math. \textbf{442} (1993), 
177--204.
%
%
\end{thebibliography}

\end{document}